\documentclass[10pt,reqno]{amsart}
\usepackage{amssymb,mathrsfs,color}

\usepackage{enumerate}

\usepackage{graphicx}
\usepackage{xcolor} 
\usepackage{tensor}
\usepackage{slashed}
\definecolor{green}{rgb}{1,0.5,0} 

\newcommand{\nrm}[1]{\Vert#1\Vert}
\newcommand{\brk}[1]{\langle#1\rangle}
\newcommand{\set}[1]{\{#1\}}

\newcommand{\dist}{\mathrm{dist}}

\newcommand{\aleq}{\lesssim}
\newcommand{\ageq}{\gtrsim}

\newcommand{\lap}{\Dlt}

\newcommand{\ud}{\mathrm{d}}
\newcommand{\rd}{\partial}
\newcommand{\nb}{\nabla}

\newcommand{\slap}{\! \not \!\! \lap}
\newcommand{\snb}{\! \not \hskip-.25em \nb}
\newcommand{\sdiv}{\not \!\!\! \mathrm{div} \,}

\newcommand{\bb}{\Big}

\newcommand{\0}{\emptyset}

\newcommand{\alp}{\alpha}

\newcommand{\gmm}{\gamma}

\newcommand{\dlt}{\delta}
\newcommand{\Dlt}{\Delta}
\newcommand{\eps}{\epsilon}

\newcommand{\lmb}{\lambda}

\newcommand{\sgm}{\sigma}

\newcommand{\omg}{\omega}

\newcommand{\bfa}{{\bf a}}
\newcommand{\bfb}{{\bf b}}
\newcommand{\bfc}{{\bf c}}
\newcommand{\bfd}{{\bf d}}
\newcommand{\bfe}{{\bf e}}

\newcommand{\bfg}{{\bf g}}

\newcommand{\bbA}{\mathbb A}

\newcommand{\bbG}{\mathbb G}
\newcommand{\bbH}{\mathbb H}

\newcommand{\bbR}{\mathbb R}
\newcommand{\bbS}{\mathbb S}

\newcommand{\calE}{\mathcal E}

\newcommand{\calH}{\mathcal H}

\newcommand{\calL}{\mathcal L}

\newcommand{\calP}{\mathcal P}

\usepackage{wrapfig}
\usepackage{tikz}
\usetikzlibrary{arrows,calc,decorations.pathreplacing}
\definecolor{light-gray1}{gray}{0.90}
\definecolor{light-gray2}{gray}{0.80}

\definecolor{deepgreen}{cmyk}{1,0,1,0.5}

\newcommand{\B}{\mathcal{B}}
\newcommand{\E}{\mathcal{E}}

\newcommand{\HH}{\mathcal{H}}

\newcommand{\T}{\mathcal{T}}

\newcommand{\QQ}{\mathcal{Q}}
\newcommand{\PP}{\mathcal{P}}
\newcommand{\SC}{\mathcal{SC}}

\newcommand{\Hp}{\mathbb{H}}
\newcommand{\N}{\mathbb{N}}
\newcommand{\R}{\mathbb{R}}
\newcommand{\Sp}{\mathbb{S}}

\newcommand{\GG}{\mathbb{G}}
\newcommand{\K}{\mathbb{K}}

\newcommand{\g}{\mathbf{g}}

\newcommand{\m}{\mathbf{m}}

\newcommand{\ori}{\mathbf{0}}
\newcommand{\al}{\alpha}

\newcommand{\ga}{\gamma}
\newcommand{\de}{\delta}
\newcommand{\e}{\varepsilon}

\newcommand{\om}{\omega}
\newcommand{\la}{\lambda}

\newcommand{\s}{\sigma}

\newcommand{\De}{\Delta}

\newcommand{\p}{\partial}
\newcommand{\na}{\nabla}

\newcommand{\supp}{\operatorname{supp}}

\makeatletter

\newcommand{\Rmnum}[1]{\expandafter\@slowromancap\romannumeral #1@}
\makeatother

\newcommand{\lec}{\lesssim}

\newcommand{\I}{\infty}

\newcommand{\ti}{\widetilde}

\newcommand{\ha}{\widehat}

\newcommand{\ang}[1]{\left\langle{#1}\right\rangle}
\newcommand{\abs}[1]{\left\lvert{#1}\right\rvert}

\newcommand{\ali}[1]{\begin{align}\begin{split} #1 \end{split}\end{align}}
\newcommand{\ant}[1]{\begin{align*}\begin{split} #1 \end{split}\end{align*}}
\newcommand{\EQ}[1]{\begin{equation}\begin{split} #1 \end{split}\end{equation}}
\newcommand{\pmat}[1]{\begin{pmatrix} #1 \end{pmatrix}}

\newcommand{\Del}[1]{}

\newcommand{\pt}{&}
\newcommand{\pr}{\\ &}
\newcommand{\pq}{\quad}
\newcommand{\pn}{}

\numberwithin{equation}{section}

\newtheorem{thm}{Theorem}[section]
\newtheorem{cor}[thm]{Corollary}
\newtheorem{lem}[thm]{Lemma}
\newtheorem{prop}[thm]{Proposition}

\newtheorem{claim}[thm]{Claim}
\theoremstyle{remark}
\newtheorem{rem}{Remark}
\newtheorem{defn}{Definition}

\newcommand{\mand}{{\ \ \text{and} \ \  }}
\newcommand{\mor}{{\ \ \text{or} \ \ }}
\newcommand{\mif}{{\ \ \text{if} \ \ }}
\newcommand{\mfor}{{\ \ \text{for} \ \ }}
\newcommand{\mas}{{\ \ \text{as} \ \ }}

\newcommand{\euc}{\mathrm{euc}}
\newcommand{\hyp}{\mathrm{hyp}}
\newcommand{\dvol}{\mathrm{dVol}}
\newcommand{\Id}{\textrm{Id}}

\def\glei{\mathrm{eq}}

\newcommand{\zero}{\mathbf{0}}
\newcommand{\nl}{\mathrm{nl}}

\usepackage{cite}

\begin{document}

\title[Profiles for energy critical waves on $\R \times \Hp^d$]{Profile decompositions for wave equations on hyperbolic space with applications}

\author{Andrew Lawrie}
\author{Sung-Jin Oh}
\author{Sohrab Shahshahani}

\begin{abstract} The goal for this paper is twofold. Our first main objective is to develop Bahouri-G\'erard type profile decompositions for waves on hyperbolic space. Recently, such profile decompositions have proved to be a versatile tool in the study of the asymptotic dynamics of solutions to nonlinear wave equations with large energy. With an eye towards further  applications, we  develop this theory in a fairly general framework, which includes the case of waves on hyperbolic space perturbed by a time-independent potential. 

Our second objective is to use the profile decomposition to address a specific nonlinear problem, namely the question of global well-posedness and scattering for the defocusing, energy critical, semi-linear wave equation on three-dimensional hyperbolic space, possibly perturbed by a repulsive time-independent potential. Using the concentration compactness/rigidity method introduced by Kenig and Merle, we prove that all finite energy initial data lead to a global evolution that scatters to linear waves as $t \to \pm \infty$. 
This proof will serve as a blueprint for the arguments in a forthcoming work, where we study the asymptotic behavior of large energy equivariant wave maps on the hyperbolic plane.
\end{abstract}

\thanks{Support of the National Science Foundation, DMS-1302782 and NSF 1045119 for the first and third authors, respectively, is gratefully acknowledged. The second author is a Miller Research Fellow, and acknowledges support from the Miller Institute.}
\maketitle

\section{Introduction}
The main goal of this paper is to develop a framework for studying large energy asymptotic dynamics of solutions to  nonlinear wave equations on a $d$-dimensional real hyperbolic space $\Hp^d$. In particular, we prove  hyperbolic space analogs of the linear and nonlinear Bahouri-G\'erard  profile decompositions, \cite{BG}, which have become invaluable tools in the study of the dynamics of nonlinear waves outside of the perturbative regime.  We then give a specific nonlinear application, including a proof of global well-posedness and scattering for the energy critical, defocusing, semilinear wave equation for all finite energy data.

\subsection{Linear theory} We begin by describing the linear theory. With an eye toward further nonlinear  applications we establish a  profile decomposition relative to  a class of linear equations 
\EQ{  \label{lin eq}
u_{tt} - \Delta_{\Hp^d} u  + V u = 0 
}
where $\Delta_{\Hp^d}$ is Laplacian on $\Hp^d$, and
$V: \Hp^d \to \R$ is a time-independent potential satisfying a collection of assumptions given in the sequel, which in particular include Strichartz estimates for the left-hand side of~\eqref{lin eq}; see ~\eqref{eq:Vdecay}--\eqref{eq:Str4halfVWave:2}. The point of this formulation is that linear equations of the form~\eqref{lin eq}  arise naturally in the study of nonlinear problems, in particular after linearizing  about an asymptotically stable nontrivial stationary solution. A particular example of such a situation is given in the recent work of the authors in~\cite{LOS1}, where a  family of equations of the form~\eqref{lin eq} are obtained by linearizing the $2d$ equivariant wave map equations from $\R \times \Hp^2 \to \Hp^2$ or into $\Sp^2$, about the members of a continuous family of harmonic maps. A profile decomposition relative to the linearized equation for a stable soliton would be a natural step towards proving large data asymptotic stability, i.e., soliton resolution. 
Indeed, in a forthcoming work, we will employ the tools developed in the present paper to make progress on this problem.

The point of a profile decomposition is to characterize a failure of compactness of a sequence of free waves with uniformly bounded energy, in particular the failure of compactness at the level of Strichartz estimates. Indeed, given a sequence of free waves $\vec v_n(t)$ in Euclidean space with bounded free energy, i.e., 
\ant{
\Box_{\R^3} v_n = 0,\quad \| \vec v_n(0)\|_{\dot{H}^1 \times L^2(\R^3)} \le C,
}
the linear result of Bahouri-G\'erard in~\cite{BG} roughly states that  one can find a sequence of free waves $\vec V^j_L$, called profiles or limiting profiles, that are independent of $n$,  and  sequences of non-compact ``symmetries", $\rho_{n, j}$,  which leave the wave equation and the free energy invariant, so that 
\EQ{ \label{bg euc0}
&\vec v_n =  \sum_{j <J} \rho_{n, j} \circ  \vec V^j_L   + \vec w_{n, L}^J, \quad  \lim_{J\to \infty} \limsup_n \| w^J_{n, L}\|_{L^5_tL^{10}_x} = 0\\
&\| \vec v_n\|_{\dot{H}^1 \times L^2}^2 =  \sum_{j <J} \|  \vec V^j_{L}\|_{\dot{H}^1 \times L^2}^2   + \|\vec w_{n, L}^J\|_{\dot{H}^1 \times L^2}^2 +o_n(1) \mas n \to \infty
}
where we note that the errors $\vec w_{n, L}^J$  vanish as  $J \to \infty$ in Strichartz norms such as $L^5_tL^{10}_x$ -- but not in the energy space. In the case of  free Euclidean waves as above, the symmetries $\rho_{n, j}$ consist of space and time translations, as well as the $\dot{H}^1 \times L^2$-scaling symmetry. Although Lorentz transforms also constitute a non-compact symmetry group acting on free waves, the uniform bound on the free energy effectively compactifies this action, and thus they do not appear in the $\rho_{n, j}$.    For $j \neq \ell$,  $\rho_{n, j}$ and $\rho_{n, \ell}$ diverge from each other as $n \to \infty$  and this divergence leads to the almost orthogonality of the free energy above. In addition, one sees that this divergence implies that if there are at least two nonzero profiles, $V^1$ and $V^2$, the sequence $\vec v_n(0)$ will fail to be relatively compact,  however this failure is characterized by~\eqref{bg euc0}, up to an error with vanishing Strichartz norm, entirely in terms of the symmetries of the equation -- in particular the limiting profiles $V^j_L$ are independent of $n$.   

In order to make sense of such a decomposition for a sequence of solutions to~\eqref{lin eq} with uniformly bounded energy, one must first understand the possible failures of compactness in Strichartz estimates in this more complicated setting.  Let us begin with a simpler equation, namely the free wave equation on $\Hp^d$. Indeed, let $\vec u_n(t) \in \HH:= H^1 \times L^2(\Hp^d)$ be a uniformly bounded sequence of solutions to $\Box_{\Hp^d} u_n = 0$, and let $S(I):= L^p_t(I;  L^{q}_x(\Hp^d))$ be an admissible Strichartz norm relative to $\HH(\Hp^d)$ (this will be made precise in Section~\ref{sec:lin}).  We seek to characterize the failure of compactness relative to the inequality 
\ant{
\| u_n \|_{S(I)} \le C \| \vec u_n(0)\|_{\HH}
}
 As in the Euclidean case, time translation as well as space translation can account for a failure of compactness -- here note that we can view $\Hp^d = SO(d, 1)  / SO(d)$ as a  symmetric space, and the spatial translations correspond to the action of $SO(d, 1)$ on $\Hp^d$, and are denoted by $h \cdot x$ where $h \in SO(d, 1)$ and $x \in \Hp^d$.  Indeed, suppose that $\vec U^1(t)$ and $\vec U^2(t)$ are two free waves and $t_{n, 1}, t_{n, 2}$  are sequences of times and $h_{n, 1}, h_{n, 2} \in \GG$ are sequences of translations. Then the sequence, 
 \ant{
 \vec u_n(t, x) =  \vec U^1( t +  t_{n, 1}, h_{n, 1} \cdot x) +  \vec U^2( t+ t_{n, 2}, h_{n, 2} \cdot x)
 }
 would fail to be relatively compact as long as 
 \ant{
  \abs{t_{n, 1} - t_{n, 2}} +  \bfd_{\Hp^d}( h_{n,1} \cdot \zero, h_{n, 2} \cdot \zero) \to \infty \mas n \to \infty
  }
where $\zero \in \Hp^d$ denotes the origin and $\bfd_{\Hp^d}(x, y)$ denotes the distance in hyperbolic space between the points $x, y \in \Hp^d$. 
 However, although time and space translations  constitute the only true symmetries on $\R \times \Hp^d$, there can also be a failure of compactness that arises from  waves that concentrate at very small scales. This key observation was made by Ionescu, Pausader, and Staffilani in their recent related work on the energy critical nonlinear Schr\"odinger equation on $\Hp^3$,~\cite{IPS}. They also developed a method for extracting profiles which arise from a small scale (or equivalently, a high frequency) concentration, which greatly motivated the approach we take in this paper.
 
An example of a small scale concentration scenario can be described in the Euclidean setting by a simple rescaling and appears naturally in the profile decomposition~\eqref{bg euc0} via the scale invariance of the free Euclidean waves, e.g., one could have 
\ant{
\rho_{n, j}  \circ V^j_L(t, x)  = \la_{n, j}^{-\frac{d-2}{2}} V^j_L(  t/ \la_{n, j}, x/ \la_{n, j})
} 
for a sequence $\la_{n, j}  \to \infty$ of positive numbers, and where we have used the $\dot{H}^1$-invariant scaling above.  In the decomposition~\eqref{bg euc0} one could  have for example
\EQ{
u_n(t, x)  = \la_{n, 1}^{-\frac{d-2}{2}} V^1(  t/ \la_{n, 1}, x/ \la_{n, 1}) + \la_{n, 2}^{-\frac{d-2}{2}} V^2(  t/ \la_{n, 2}, x/ \la_{n, 2})
}
where $\la_{n, 1}/ \la_{n, 2} \to \infty \mas n \to \infty$, which would amount to two distinct profiles that live at dramatically different scales. 

Waves on hyperbolic space  are not scale invariant. However, Ionescu, Pausader, and Staffilani made precise the intuition  that a solution  to the underlying scale invariant Euclidean equation is a very good approximation to the corresponding solution to the same equation on hyperbolic space when the initial data is highly localized. Take for example  a sequence of compactly supported initial data $(f_n, g_n)$ that concentrate to a point in hyperbolic space, say by rescaling a single compactly supported profile, i.e.,  
\ant{
(f_n, g_n)(r, \om)  =  ( \la_n^{\frac{d-2}{2}} f( \la_n r, \om), \la_n^{\frac{d}{2}} g(\la_nr , \om)), \quad \la_n \to \infty
}
where $(r, \om)$ are geodesic polar coordinates on $\Hp^d$.  The corresponding sequence of hyperbolic free waves $\vec u_n$ with this data will be well-approximated by the rescaled Euclidean evolution for a fixed and finite time, due to the finite speed of propagation and the fact that highly concentrated data do not ``see" the hyperbolic metric. In other words, the difference between $\Box_{\Hp^d}$ and $\Box_{ \R^d}$ is negligible up to a finite time for data that are highly localized in space. We make this heuristic precise in the context of linear and nonlinear equations in Sections~\ref{subsec:linApprox} and~\ref{sec:nonlinApprox}, respectively. Such an approximation theory plays a crucial role in defining the proper notion of nonlinear profiles; see Section~\ref{sec: bg nl}.

We now return to the full equation~\eqref{lin eq}, by adding a time-independent potential term $V u$. Assume that $V(x)$ decays sufficiently fast as $\ud_{\bbH^{d}}(x, \zero) \to \infty$ (we refer to \eqref{eq:Vdecay}--\eqref{eq:Str4halfVWave:2} for the full list of assumptions that $V$ is required to satisfy). The presence of a non-trivial decaying potential $V$ breaks yet another symmetry of the equation, namely the translation invariance. Nevertheless, much like the case of scaling symmetry, the non-compact action of $SO(d,1)$ on the solutions to~\eqref{lin eq} still causes a failure of compactness. A typical scenario consists of a sequence of solutions $\vec{u}_{n, L}(t)$ to~\eqref{lin eq}, whose data `travels' out to infinity by translation, i.e.,
\begin{equation*} 
\vec{u}_{n, L}(0, x) = \vec{U}(0, h_{n} \cdot x), \quad \bfd_{\bbH^{d}}(h_{n}^{-1} \cdot \zero, \zero) \to \infty
\end{equation*}
where $\vec{U} \in H^{1} \times L^{2}(\bbH^{d})$ is a fixed initial data and $\set{h_{n}} \subseteq SO(d,1)$. Thanks to the decay of $V$, the linear waves $\vec{u}_{n, L}$ are well-approximated by translates of a single solution to the underlying translation-invariant equation, namely 
\begin{equation} \label{bg hyp0}
	u_{n}(t, x) = U(t, h_{n}^{-1} \cdot x)
\end{equation}
where $\vec{U}_{L}$ is a solution to $\Box_{\bbH^{d}} U_{L}  = 0$ with initial data $\vec{U}_{L}(0) = \vec{U}(0)$. This approximation theory, for linear and nonlinear equations, is also developed in Sections~\ref{subsec:linApprox} and~\ref{sec:nonlinApprox}, respectively. 
Analogous to the case of concentrating solutions, it will provide the basis for a proper definition of nonlinear profiles in Section~\ref{sec: bg nl}.

The central idea of the profile decomposition is that as $n \to \infty$ the Strichartz norm of the sequence $\vec u_n(t)$ of solutions to~\eqref{lin eq}  is asymptotically distributed in a discrete collection of limiting profiles $\vec U^j$, which are concentrating in space, time, and frequency according to parameters, $\set{t_{n, j}} \subseteq \R$, $\set{h_{n, j}} \subseteq \GG$ and $\set{\la_{n, j}} \subseteq [1, \infty)$. Without getting into too much detail in the introduction, and without specifying the exact definition, we denote these concentrating profiles by  $\vec U_{n, L}^j(t)$ where the index $n$ captures the effect of the parameters $\{t_{n, j}, h_{n, j}, \la_{n, j}\}$ and the index $L$ is meant to signify that $\vec U_{n, L}^j$ is a linear wave, i.e., a solution to~\eqref{lin eq}. We define $\vec U^j_{n, L}(t)$ precisely in Definition~\ref{def:linprof} in Section~\ref{sec: bg}, and let it suffice to say  for now that they include natural analogues of the Euclidean concentrating profiles $ \rho_{n, j} \circ \vec V^j_L$ that appear in~\eqref{bg euc0} and the potential-free traveling profiles $\vec{U}^{j}_{L}(t, h_{n, j}^{-1} \cdot x)$ in~\eqref{bg hyp0}. 


In the statement below we denote by $ E_V( \vec u)$ the conserved energy $$ E_V( \vec u):=\frac{1}{2}\int_{\Hp^3}  \left( u_t^2 + \abs{ \na u}^2  + V \abs{u}^{2} \right)\, \dvol_{\Hp^d}$$ of a solution $\vec u$ to~\eqref{lin eq} and denote by $S(I)$ any admissible Strichartz norm relative to the energy.   We prove the following theorem, which is made precise and restated in Section~\ref{sec: bg}. 

\begin{thm}[To be more precisely restated as Theorem~\ref{thm: BG}]\label{thm:mainbg}
Let $\vec u_{n, L}$ be a sequence of solutions to~\eqref{lin eq} with uniformly bounded energy. Then, up to passing to a subsequence, there exists a sequence of limiting profiles $\vec U^j$ along with concentration parameters  $\{t_{n, j}, h_{n, j}, \la_{n, j}\} \in \R \times \GG \times  [1, \infty)$ so that for errors $\vec w_{n, L}^J$ defined by 
\ant{
\vec u_{n, L}(t)  = \sum_{j <J} \vec  U^j_{n, L} (t) +  \vec w_{n, L}^J(t) 
} 
we have 
\ant{
 \limsup_{n \to \infty}\| w_{n, L}^J\|_{S(\R) } \to 0 \mas J \to \infty 
}
For each $J \in \N$ and for every $j \neq k  <J$  we have the following orthogonality of parameters, 
\ant{
&\textrm{either}  \quad \frac{\la_{n, j}}{ \la_{n, k}} + \frac{\la_{n, k}}{\la_{n, j}}  \to \infty\\
 & \textrm{or} \quad  \la_{n, j} \simeq \la_{n, k}  \mand   \la_{n, j}\abs{t_{n, j} - t_{n, k}} +  \la_{n, j}\bfd_{\Hp^d}(h_{n, j} \cdot \zero,  h_{n, k} \cdot \zero) \to \infty
  }
Moreover, we have the following Pythagorean decomposition of the free energy: 
For each fixed $J \in \N$,  
\ant{
E_V( \vec u_{n, L}) = \sum_{j<J} E_V( \vec U^j_{n, L}) + E_V(\vec w^J_{n,L}) + o_n(1) \mas n \to \infty
}
\end{thm}

Roughly speaking, given a sequence of linear waves $\vec{u}_{n, L} \in \HH$ with bounded energy, Theorem~\ref{thm:mainbg} identifies the failure of compactness at the level of Strichartz norms with the action of the non-compact (semi-)\footnote{To be pedantic, one must add the prefix `semi-' to indicate that the scaling action is limited to the semi-group $[1, \infty)$.}group $\bbR \times SO(d, 1) \times [1, \infty)$, where the factors correspond to time translation, spatial translation and scaling, respectively. Note that the scaling group $(0, \infty)$ is compactified to the expanding direction $\to 0$ (as in the case of the Lorentz transformations on $\bbR^{1+d}$), since the uniform boundedness of energy implies uniform boundedness of the $L^{2}$ norm by Poincar\'e's inequality on $\bbH^{d}$.

\subsection{Nonlinear results} 
The linear profile decomposition~\eqref{bg euc0} has played a central role in many recent developments concerning  the understanding of large energy dynamics  for the energy critical \emph{nonlinear} wave equation, see for example~\cite{BG, KM08, DKM1, DKM2, DKM3, DKM4, DKM6, DKM7, KS13ajm} -- there have also been many more applications, too many to name here. In particular,  one can build from~\eqref{bg euc0} a corresponding nonlinear profile decomposition. Indeed, consider in $\R^{1+3}$ the Cauchy problem for the quintic nonlinear wave equation 
\EQ{\label{nlw5}
 &\Box_{\R^{1+3}} v =  \pm \abs{v}^4 v \\
&\vec v(0) = (v_0, v_1) \in \dot{H}^1 \times L^2(\R^3)
 }
To a sequence $\vec v_n(0)$ of initial data, one can associate a corresponding sequence of nonlinear evolutions $\vec v_n(t)$. Moreover, to each linear profile $\vec V^j_L$ in the decomposition~\eqref{bg euc0} one can associate, via the local Cauchy theory for~\eqref{nlw5}, a nonlinear evolution $\vec V^j_{\nl}$ so that 
\ant{
 \| \rho_{n, j} \circ \vec V_{L}^j - \rho_{n, j} \circ V^j_{\nl}(0) \|_{\dot{H}^1 \times L^2} \to 0 \mas n \to \infty
 }
Although the superposition principle fails for nonlinear waves, the divergence of the parameters $\rho_{n, j}$ means that the nonlinear waves built from the linear profiles will interact very little and one can prove a nonlinear profile decomposition 
\EQ{
 \vec v_n(t)  =  \sum_{j <J}   \rho_{n, j} \circ\vec V^j_{\nl}(t)  + w_{n, L}^J + \gamma_{n}^J(t),  \quad  \lim_{J \to \infty} \limsup_{n \to \infty}\|\vec \ga\|_{L^{\infty}_t \dot{H}^1 \times L^2}  = 0, 
 }
see \cite[Main Theorem]{BG} for more details. 

The linear profile decomposition of  Theorem~\ref{thm:mainbg} can be thought of as a first step in understanding large energy dynamics for solutions to semilinear nonlinear Cauchy problems of the form 
\ant{
&u_{tt} - \Delta_{\Hp^d} u    + V u = F(u) \\
&\vec u(0)= (u_0, u_1) \in \HH := H^1 \times L^2 (\Hp^d)
} 
with a nonlinearity $F(u)$ that is amenable to a small data scattering theory in the energy space~$\HH$ based on Strichartz estimates, and to a nonlinear perturbation theory for approximate solutions; see  for example Proposition~\ref{small data} and Lemma~\ref{lem: pert}. Indeed, given a small data scattering theory and a nonlinear perturbation theory, one can build an associated nonlinear profile decomposition from the linear one; see Theorem~\ref{thm:nonlinbg}.  

In this paper, we restrict to a specific nonlinear question, namely that of large data global existence and scattering for critical \emph{defocusing} equations, and we execute a version of the concentration compactness/rigidity method developed by Kenig and Merle in~\cite{KM06, KM08}, in the hyperbolic space setting. In particular,  we consider energy critical, power-type nonlinearities with the defocusing sign,  
\ant{
F(u) = - \abs{u}^p u , \, \quad p = \frac{d+2}{d-2}.
}
For simplicity, we restrict to the dimension $d = 3$. We allow the presence of a potential $V$ that is assumed to be \emph{compactly supported} and \emph{repulsive} (for the precise assumptions on $V$, we refer to the statement of Theorem \ref{main:V} below). As discussed earlier, the motivation for including $V$ in our analysis comes from consideration of large energy equivariant wave maps from the hyperbolic plane $\bbH^{2}$, where a potential arises as a result of linearizing the equation around non-trivial harmonic maps. 
This subject will be taken up in a forthcoming work.

Our  nonlinear model problem is then the Cauchy problem for the  defocusing, energy critical, semilinear wave equation on $\R \times \Hp^3$, 
\EQ{\label{u eq d}
&u_{tt} - \De_{\Hp^3} u  + Vu = - \abs{u}^{4} u,\\
& \vec u(0) = (u_0, u_1),
}
where $V$ is smooth, compactly supported and repulsive, in the sense that \eqref{eq:Vdefocusing} below is satisfied. The conserved energy is given by 
\EQ{
\E(\vec u)(t):= \frac{1}{2}\int_{\Hp^3}  \left( u_t^2 + \abs{ \na u}^2  + V \abs{u}^{2} \right)\, \dvol_{\Hp^d} + \frac{1}{p+1} \int_{\Hp^d} \abs{u}^{p+1} \, \dvol_{\Hp^d}
}
We will consider initial data $\vec u(0) = (u_0, u_1)$ for~\eqref{u eq d} in the energy space 
\EQ{
\HH(\Hp^3): = H^1 \times L^2( \Hp^3). 
}

We prove the following theorems for \eqref{u eq d}. In the statements below, for a time interval $I$ we denote by $S(I)$ the Strichartz  norm 
 \EQ{
 S(I):= L^{5}_{t} (I; L^{10}(\bbH^{3})).
 }
Our first nonlinear theorem concerns the special case $V = 0$, which is precisely the \emph{energy-critical defocusing} semilinear wave equation on $\bbH^{3}$. 
\begin{thm}\label{main} 
Consider the equation 
\begin{equation} \label{u eq:noV}
u_{tt} - \De_{\Hp^3} u = - \abs{u}^{4} u,
\end{equation}
Let $(u_0, u_1) \in \HH(\Hp^3)$. Then there is a unique global-in-time solution $\vec u(t) \in \HH$ to~\eqref{u eq d} with $\vec u (0) = (u_0, u_1)$. Moreover, there exists a non-decreasing function $A: [0, \infty) \to [0, \infty)$ so that 
\EQ{
\| u \|_{S(\R)} \le A ( \E( \vec u(0))).
}
In particular, this means that $\vec u(t)$ scatters to a free hyperbolic wave as $t \to \pm \infty$, i.e., there exist  solutions $\vec u_L^{\pm}(t) \in \HH$ to
$
\Box_{\Hp^d} u_L = 0
$
so that
\EQ{
\| \vec u(t) - \vec u_L^{\pm}(t) \|_{\HH} \to 0 \mas t \to  \pm \infty.
}
\end{thm}
Using the tools developed in the earlier part of the paper, this theorem is proved in Section~\ref{sec:main:pf}.

Since the tools we develop easily allows for inclusion of a potential, we are in fact able to prove the following generalization of Theorem~\ref{main}, where $V$ is a smooth, compactly supported repulsive potential; see \eqref{eq:Vdefocusing} below for the precise assumptions on $V$.
\begin{thm}\label{main:V} 
Consider the equation 
\begin{equation} \label{u eq}
u_{tt} - \De_{\Hp^3} u  + V u= - \abs{u}^{4} u,
\end{equation}
where the time-independent potential $V$ is \emph{smooth, repulsive} and \emph{compactly supported} in the following sense:
\begin{equation} \label{eq:Vdefocusing}
	V \in C^{\infty}_{0}(\bbH^{d}), \quad V \geq 0, \quad \rd_{r} V \leq 0.
\end{equation}
where $\rd_{r}$ is the radial directional derivative in the polar coordinates $(r, \omg)$ on $\bbH^{3}$. Let $(u_0, u_1) \in \HH(\Hp^3)$. Then there is a unique global-in-time solution $\vec u(t) \in \HH$ to~\eqref{u eq d} with $\vec u (0) = (u_0, u_1)$. Moreover, there exists a non-decreasing function $A: [0, \infty) \to [0, \infty)$ so that 
\EQ{
\| u \|_{S(\R)} \le A ( \E( \vec u(0))).
}
In particular, this means that $\vec u(t)$ scatters to a free hyperbolic wave as $t \to \pm \infty$, i.e., there exist  solutions $\vec u_L^{\pm}(t) \in \HH$ to
$
\Box_{\Hp^d} u_L = 0
$
so that
\EQ{
\| \vec u(t) - \vec u_L^{\pm}(t) \|_{\HH} \to 0 \mas t \to  \pm \infty.
}
\end{thm}

Although Theorem~\ref{main:V} is clearly a more general result, we will in fact use the special case proved in  Theorem~\ref{main}  to deduce Theorem~\ref{main:V}, again following the Kenig-Merle concentration compactness/rigidity method; see Section~\ref{sec:main:V:pf}.

\begin{rem}
We point out that Theorem~\ref{main:V} claims scattering of solutions to \eqref{u eq} to \emph{free hyperbolic waves}, instead of solutions to the perturbed equation $(\Box_{\bbH^{d}} + V) u = 0$. This statement holds because the assumptions on $V$ in Theorem~\ref{main:V} are strong enough to ensure that solutions to the perturbed equation $(\Box_{\bbH^{d}} + V) u = 0$ scatter to solutions to the free equation $\Box_{\bbH^{d}} u = 0$. We refer to the proof of Proposition~\ref{small data} for details.
\end{rem}


\begin{rem} 
Theorems~\ref{main} and~\ref{main:V} are intended to be simple examples of applications of the methods developed in this paper, 
and hence may be extended in many directions.
For example, the argument used to prove Theorems~\ref{main} and~\ref{main:V} can be easily adapted to dimensions $d = 4$ and $5$, and most of it also carries over to even higher dimensions, modulo technical nuisances concerning low regularity of the nonlinearity. The assumption \eqref{eq:Vdefocusing} on $V$ can certainly be relaxed as well. 
\end{rem}

\begin{rem}\label{shift remark}
Since the spectrum of the Laplace operator on $\Hp^d$ is given by $\sigma(-\lap_{\Hp^d})=[\frac{(d-1)^2}{4},\infty)$, one can replace the linear operator in \eqref{lin eq}, and hence in Theorems~\ref{thm:mainbg},\\~\ref{main}, and~\ref{main:V}, by $\partial_{tt}-\lap_{\Hp^d}+V+\mu$ with $\mu>-\frac{(d-1)^2}{4}$ without changing the analysis. The conclusions of Theorems~\ref{thm:mainbg},~\ref{main}, and~\ref{main:V} remain valid in this more general setting with the obvious modifications. In particular a free hyperbolic wave should now be interpreted as a solution of the linear wave equation $\Box_{\Hp^d}u+\mu u=0.$
\end{rem}



\subsection{History of ideas in the paper} 
Here we given a brief and incomplete summary of the history of some of the ideas that went into this paper.  The notion of a profile decomposition such as~\eqref{bg euc0} for dispersive equations  began with the work of Bahouri, Ger\'ard~\cite{BG}, Merle, Vega~\cite{MV}, and later of Keraani~\cite{Ker}.  These works are intimately related to the celebrated concentration compactness trichotomy of P.L. Lions for measures, and in particular to the more explicit form given by G\'erard \cite{Ger98}. The first such result in the setting of dispersive equations on hyperbolic space was given for the nonlinear hyperbolic Schr\'odinger equation by Ionescu, Pausader and Staffilani,~\cite{IPS}, which  served as a starting point for the present work. 

Regarding the nonlinear portion of this paper, the energy critical wave equation in Euclidean space, $\Box_{\R^3} u =  \pm u^5$, has been extensively studied. Global regularity of  solutions corresponding to  finite energy data was proved by Struwe,~\cite{Struwe88}, for the radial defocusing equation and later by Grillakis,~\cite{Gri90}, for the nonradial, defocusing equation. Later a satisfactory description of the global dynamics in the defocusing case, in particular the question  of scattering, was resolved by Shatah, Struwe~\cite{SS93}, Bahouri, Shatah~\cite{BS}, and Bahouri G\'erard~\cite{BG}. 

For the  focusing energy critical equation, type-II blow can and does occur,  as explicitly demonstrated by Krieger, Schlag, and Tataru~\cite{KST3}, via a concentration of  energy  culminating with  the bubbling  off of the unique radial  ground state, $W$, which solves the underlying elliptic equation; see also \cite{DHKS, DK, KS12}. 

In~\cite{KM06, KM08}, Kenig and Merle began a powerful program for understanding large energy dynamics for semilinear equations   with the now ubiquitous  concentration compactness/rigidity method. The concentration compactness aspect of the method is rooted in the  profile decompositions of Bahouri and G\'erard,~\cite{BG}. In~\cite{KM08}, Kenig and Merle gave a characterization of all possible dynamics for solutions with energy  below the threshold energy of the ground state, $W$. The remarkable work of Duyckaerts, Kenig, and Merle \cite{DKM1,DKM3, DKM2, DKM4} gives a classification of possible dynamics for large energies. To be precise,  all type-II radial solutions asymptotically resolve into a sum of rescaled solitons  plus pure radiation.  The explicit  dynamics at the threshold energy of $W$ have been studied by Duyckaerts and Merle~\cite{DM} and slightly above the threshold energy by Krieger, Nakanishi, and Schlag in~\cite{KNS13ajm, KNS13DCDS, KNS14CMP}. 

Strichartz estimates for the linear wave equation on hyperbolic space were established independently by Metcalfe, Taylor~\cite{MT11, MTay12} and by Anker and Pierfelice~\cite{AP}. These authors showed that free waves experience faster dispersion on hyperbolic space due to the exponential growth of the volume of concentric spheres. The stronger dispersion led to  a wider range of admissible Strichartz estimates than in the Euclidean setting and allowed the authors to prove well posedness results for the semi-linear equations with powers below the famous Strauss exponent. 

There have also been several results on the nonlinear Schr\"odinger equation on hyperbolic space. For results in the radial setting see Banica~\cite{Ban07} and Banica, Carles, Staffilani~\cite{BCS}  and Banica, Carles, Duyckaerts~\cite{BCD}. Strichartz estimates as well as  scattering in $H^1$ for the subcritical problem in the nonradial setting were proved by Ionescu, Staffilani,~\cite{IS} and Anker, Pierfelice~\cite{AP09}.  In~\cite{IS}, Ionescu and Staffilani proved a strong hyperbolic Morawetz-type estimate for the defocusing Schr\"odinger problems. The same multiplier is used in the present work to prove a Morawetz estimates in the setting of the nonlinear wave equation. Finally, global existence and scattering for the energy critical defocusing nonlinear Schr\"odinger equation was established by Ionescu, Pausader, and Staffilani,~\cite{IPS}.  

A related equation that has garnered attention recently is the ``shifted" wave equation on $\R \times \Hp^d$. Here the word shifted refers to taking $\mu = -\frac{(d-1)^2}{4}$ in Remark~\ref{shift remark}, which eliminates the entire spectral gap enjoyed by the Laplacian on $\Hp^d$. This equation exhibits several qualitative differences than the type of equations considered in the present work.  In particular, one can conjugate the shifted D'Alembertian on hyperbolic space to obtain the Euclidean D'Alembertian with respect to the hyperbolic foliation of  the forward light cone in Minkowski space, see for example~\cite{Tat01hyp}. Dispersive estimates for the shifted equation were studied by Tataru in~\cite{Tat01hyp} and by Anker, Pierfelice and Vallarino~\cite{APV}.  Recently, scattering by way of hyperbolic Morawetz estimates was established for subcritcal defocusing semilinear shifted equations by Shen, Staffilani~\cite{ShenS14}, and for the radial energy critical shifted equation by Shen~\cite{Shen14}. 
 
Finally, we also mention the recent work of Jia, Liu, Xu~\cite{JLX} who proved asymptotic relaxation to steady states for the defocusing energy critical wave equation on $\R^{1+3}$ with a potential using the Duyckaerts, Kenig, Merle approach,  and of Hong~\cite{Hong14}, who studied the cubic focusing nonlinear Schr\"odinger equation on $\bbR^{3}$ with a real-valued short-range potential that has a small negative part using, in particular, the concentration compactness/rigidity approach. 
 




\subsection{Outline of the paper}
We begin in Section~\ref{sec: prelim} with some preliminaries that set the stage for the rest of the paper. There we recall some basic facts regarding the structure of hyperbolic space $\Hp^d$. Furthermore, we extend the results of Ionescu, Pausader, and Stafflani, \cite{IPS}, concerning Littlewood-Paley theory and refined Sobolev embedding theorem on $\bbH^{d}$; see Lemma~\ref{lem: se}. The latter result is essential in the extraction of the linear profiles in Theorem~\ref{thm:mainbg}. We introduce an alternative approach for proving Lemma~\ref{lem: se}, which avoids the use of the delicate Helgason-Fourier transform and relies directly on the properties of the linear heat equation $(\rd_{s} - \lap_{\bbH^{d}})f = 0$, which are known to hold in a more general setting.

The remainder of the paper is divided into two parts, plus an auxiliary Section~\ref{sec:morawetz}. Part~\ref{p:lin}, which consists of Sections~\ref{sec:lin} and Section~\ref{sec: bg}, concerns the linear theory, culminating in the proofs of the linear profile decomposition, namely Theorem~\ref{thm: BG}. Part~\ref{p:nonlin}, consisting of Sections~\ref{sec: wp}--\ref{sec:main:V:pf}, is devoted to the the nonlinear theory, including the nonlinear profile decomposition, i.e., Theorem~\ref{thm:nonlinbg}, and the proofs of Theorems~\ref{main} and~\ref{main:V}. 

Section~\ref{sec:lin} contains the dispersive theory for the linear equation~\eqref{lin eq}, including the improved dispersive estimates on $\Hp^d$ of Anker and Pierfelice, \cite{AP}, and Metcalfe and Taylor,~\cite{MTay12}, along with Strichartz estimates for~\eqref{lin eq}. We also develop linear approximation theory for sequences of solutions to~\eqref{lin eq} that either escape to spatial infinity or concentrate to a point; see Section~\ref{subsec:linApprox}. The latter statement is an analogue of the result of Ionescu, Pausader, and Staffilani, \cite{IPS}, which says that solutions with data that is highly localized in $\Hp^d$ can be approximated by solutions to the underlying Euclidean scale invariant equation. 

Section~\ref{sec: bg} is devoted to the proof of the linear profile decomposition, Theorem~\ref{thm: BG}, which is a more precise version of Theorem~\ref{thm:mainbg}. 
The main technical tools are the refined Sobolev inequality, Lemma~\ref{lem: se}, the dispersive estimate for the free wave equation on $\bbH^{d}$ and the linear approximation theory developed in Section~\ref{subsec:linApprox}.

Section~\ref{sec: wp} marks the beginning of Part~\ref{p:nonlin}. There we establish the local Cauchy theory, including the small data scattering theory and nonlinear perturbation theory, for solutions to~\eqref{u eq d}, see Proposition~\ref{small data} and Lemma~\ref{lem: pert}. This is of course a standard argument based on the Strichartz estimates from Section~\ref{sec:lin}. 

In Section~\ref{sec:nonlinApprox} we use the nonlinear perturbation theory to extend the linear approximation theory in Section~\ref{subsec:linApprox} to the nonlinear setting; see Propositions~\ref{prop: h h} and~\ref{prop: e h}. These results play a crucial role in properly defining the nonlinear profiles in the next section.

In Section~\ref{sec: bg nl} we prove a nonlinear profile decomposition for sequences of solutions to~\eqref{u eq d}, namely Theorem~\ref{thm:nonlinbg}. The key point here is that there are three types of nonlinear profiles, corresponding to the action of different factors of the non-compact (semi-)group $\bbR \times SO(d, 1) \times [1, \infty)$. They are: (i) \emph{stationary} profiles, which can be translated in time but live at a fixed location and scale; (ii) \emph{traveling} profiles, which can be translated in space and time but live at a fixed scale; and (iii)  \emph{concentrating} profiles, which can also be translated in space and time but are concentrating at small scales. These profiles are also called \emph{perturbed hyperbolic}, \emph{free hyperbolic} and \emph{Euclidean}, respectively, after the underlying equations that govern their behaviors.

In Section~\ref{sec:main:pf}, we establish Theorem~\ref{main} concerning the energy-critical defocusing semilinear wave equation on $\bbH^{3}$ without a potential.
In Section~\ref{sec: ce}, we execute the concentration compactness part of the Kenig-Merle scattering blueprint. The goal is to use the nonlinear profile decomposition of the previous section to prove the following: In the event that Theorem~\ref{main} fails, there exists a minimal non-scattering solution to~\eqref{u eq}, called the critical element, which has a pre-compact trajectory in the energy space $\HH$, up to spatial translations. Then in Section~\ref{sec: rig} we prove a rigidity result, namely any solution with a pre-compact trajectory in $\HH$ up to translations must be identically zero. This result yields a contradiction with  the existence of the the critical element and completes the proof of Theorem~\ref{main}. 

The main tool in the rigidity argument is a Morawetz-type estimate for solutions to~\eqref{u eq d} of the same ilk as the Morawetz estimates proved in~\cite{IPS}. In particular, we show that every solution to~\eqref{u eq d}, say for $d=3$ has finite $L^{6}_{t, x}$ norm. This estimate is \emph{not enough} to conclude however, since the equation is supercritical with respect to $L^{6}_{t, x}$, and thus the Morawtetz estimate alone cannot rule out a solution that, say,  concentrates at a point. However, we can combine the Morawetz-type estimate with the assumed pre-compactness of the trajectory -- in particular with the fact  that we have ruled out Euclidean profiles -- see the proof of Proposition~\ref{rig}. 

In Section~\ref{sec:main:V:pf}, we add in a nontrivial potential $V$ and prove Theorem~\ref{main:V}. The proof proceeds along the same line as in the previous section, except now we use Theorem~\ref{main} to say that all traveling profiles must be global and scattering. If Theorem~\ref{main:V} fails, we then produce a critical element which plainly has a pre-compact trajectory in $\HH$. As the assumptions on $V$ allow us to derive an analogous Morawetz-type estimate as in Theorem~\ref{main}, the same rigidity argument applies, thereby finishing the proof of Theorem~\ref{main:V} up to proofs of a few statements in the next section.

Finally, in Section~\ref{sec:morawetz}, we give a proof of the integrated local energy decay \eqref{eq:VLED} and Morawetz-type estimates under the presence of a potential $V$, which satisfies the assumptions of Theorem~\ref{main:V}. The proof depends on a multiplier of the same type as considered earlier by Ionescu, Pausader, Staffilani~\cite{IPS} and Shen, Staffilani,~\cite{ShenS14}. We also establish Strichartz estimates for the perturbed equation $(\Box_{\bbH^{d}} + V) u = 0$, which are necessary in the proof of Theorem~\ref{main:V}.

\section{Preliminaries}\label{sec: prelim}

 \subsection{Hyperbolic space, convolution, Littlewood-Paley theory} Consider $(d+1)$-dimensional Minkowski space, endowed with the metric $\m = \textrm{diag}(1,  \dots, 1, -1)$ and coordinates $(x^1, \ldots, x^d, x^0)$. Define the bi-linear form
\ant{
&[ \cdot, \cdot] :  \R^{d+1} \to  \R \\
&[x, y] = x^0y^0 - \sum_{j=1}^d x^jy^j 
}
We then define the $d$-dimensional hyperbolic space $\Hp^d$ by 
\ant{
\Hp^d := \{ x \in \R^{d+1} \mid [x, x] = 1, \, \, x^0>0\}
}
The Minkowksi metric on $\R^{1+d}$ induces  a Riemannian structure on $\Hp^d$ via the pull-back of the inclusion map. We view the point $\zero :=(0 , \ldots , 0, 1) \in \R^{d+1}$, which is the vertex of the hyperboloid,  as the origin in $\Hp^d$.

Next, define $(\GG, \circ):= (SO(d, 1), \circ)$ as the connected Lie group of $(d+1) \times (d+1)$ matrices that leave the bilinear form $[ \cdot, \cdot]$ invariant. We can characterize $\GG$ by
\ant{
A \in \GG \Longleftrightarrow A^t \m A = \m, \,\, \det(A) = 1,   \mand A_{00} >0
}
where again $\m = \textrm{diag}(1, \ldots, 1, -1)$.  Next, we view $\K= SO(d)$, as a subgroup of $\GG$ which fixes the origin $\mathbf{0}$.  Indeed $\K$ is a compact subgroup of rotations acting on the variables $(x^1, \ldots, x^d)$.  We can thus identify $\Hp^d$ with the symmetric space $\GG/ \K$.  With this set-up we define translation in $\Hp^d$ as the action of $\GG$ on $\Hp^d$ . For every $h \in  \GG$ we can define the map
\ant{
&L_h: \Hp^d \to \Hp^d\\
&L_h(x) = h \cdot x
}
This is an isometry of $\Hp^d$ and gives rise to the  isometry
\ant{
\tau_h: L^2( \Hp^d)  \to L^2(\Hp^d)\\
\tau_h(f)(x) = f( h \cdot x)
}
A function, $f: \Hp^d \to \R$ is called $\K$-invariant, or radial, if for all $k \in \K$ and for all $x \in \Hp^d$ we have
\ant{
f(k \cdot x) = f(x)
}
A function $f: \GG \to \R$ is called bi-$\K$-invariant if $f(k_1 \circ g \circ k_2) = f(g)$ for all $k_1, k_2 \in \K$. 

It will be useful to keep in mind the \emph{Cartan decomposition} of $h \in \bbG$, namely 
\EQ{ \label{eq:cartanDecomp}
  &h = k \circ a_{r}  \circ \ti{k}, \quad a_{r} \in \bbA_+, \, \, k , \ti{k} \in \K,  \\
  &a_{r}:=  \pmat{ \Id_{d-1 \times d-1} & 0 &0\\0 & \cosh r & \sinh r \\ 0 & \sinh r & \cosh r}, \quad \bbA_+:=\{a_r \,:\, r\in[0,\infty) \}.
  }

There are several convenient global coordinate systems one can consider on $\Hp^d$, one of which is geodesic polar coordinates:
\ant{
\R_+ \times \Sp^{d-1} \ni (r, \om) \mapsto ( \sinh r \cdot \om, \cosh r) \in \R^{d+1}
}
Denote this map by $\Psi : [0, \infty) \times \Sp^{d-1}  \to (\R^{d+1}, \m)$ where $\R^{d+1}$ denotes $d+1$ dimensional Minkowski space and $\m$ is the Minkowski metric.
The hyperbolic metric, $\g$, in these coordinates is given by the pull-back of the Minkowksi metric by $\Psi$, i.e., $\g= \Psi^* \m$. 
The volume element $\mu(dx)$ on $\Hp^d$  in these coordinates is given by $ \sinh^{d-1} r dr \sigma(d \om)$ 
and hence for $f : \Hp^d \to \R$ we have
\ant{
\int_{\Hp^d} f(x) \, \mu(dx)   =  \int_{\Sp^{d-1}} \int_0^{\infty} (f \circ \Psi)(r, \om) \, \sinh^{d-1} r  \, dr\, \sigma(d\om)
}
The Laplace-Beltrami operator is  given by
\ant{
\Delta_{\Hp^d} =  \p_r^2 + (d-1) \coth r \,  \p_r + \sinh^{-2} r \, \Delta_{\Sp^{d-1}}.
}
It is also worth noting that $\K$ invariance, or radiality, means that a function $f: \Hp^d \to \R$ depends only on the radial variable $r$,  and we will often abuse notation by writing $f(x) = f(r)$ in this case. We remark that the radial variable $r$ is in fact the hyperbolic distance from the point $x = ( \sinh r \cdot \om, \cosh r) $ to the origin $\mathbf{0} = (0, \ldots, 0, 1)$ and we write $r = \bfd_{\Hp^d}(x, \mathbf{0}) = \cosh^{-1}([x, \mathbf{0}])$. In general we have $\bfd_{\Hp^d}(x, y) = \cosh^{-1}([x, y])$.

It will often be convenient to recast the integration formulas above with a group theoretic interpretation. We note that $\GG$ is  semi-simple and hence unimodular and $\K$ is compact. We normalize the Haar measures on $\K$ and $\GG$ so that 
\EQ{
\int_{\K} 1 \, dk =1, \quad \int_{\GG} f(g \cdot \mathbf{0}) \, d g = \int_{\Hp^d} f(x) \mu (dx)
} 

In the group theoretic formulation, we can define convolution. Let $f_1, f_2 \in C^{\infty}_0( \GG)$. Then we have 
\EQ{
f_1 \ast f_2( h)&:= \int_{\GG} f_1(g) f_2( g^{-1} \circ h) \, dg  = \int_{\GG} f_1( h \circ \ti g^{-1}) f_2( \ti g) \, d \ti g \\
&= \int_{\GG} f_1( h \circ g) f_2( g^{-1}) \, dg
}
where in the equalities above we have used  change of variables and the various invariances of the Haar measure, i.e, 
\ant{
d( g \circ h)  = d( h \circ g) = dg = d( g^{-1})
}
which are due to the unimodularity of $\GG$. 
We can refine the definition of convolution  in the case where $f, K: \Hp^d \to \R$ and $K$ is a radial ($\K$-invariant) function. In particular we can recast convolution in terms of the group action ``$\cdot$" as opposed to the group operation ``$\circ$". For $x:= h \cdot \zero$, we have 
\EQ{ \label{convo}
f \ast K( x)  &=  \int_{\GG} f( h \circ g) K( g^{-1}) d g = \int_{\GG} f ( h \circ g) K( g) \, dg\\
& = \int_{\GG} f (h \cdot (g \cdot \zero)) K( g \cdot \zero) \, dg \\
&  = \ang{  \tau_h f \mid K}_{L^2( \Hp^d)} = \ang{ f \mid \tau_{h^{-1}} K}_{L^2( \Hp^d)}
}
where above we have used the fact that if $K$ is radial, then $K(g \cdot \zero) = K( g^{-1} \cdot \zero)$.

 \subsection{Function Spaces} The $L^p( \Hp^d)$ spaces are defined as usual for $1 \le p < \infty$ with
\ant{
\|f \|_{L^p(\Hp^d)} =  \left( \int_{\Hp^d} \abs{ f(x) }^p\,  \mu(dx)\right)^{\frac{1}{p}}.
}

There are two possible approaches to defining Sobolev spaces $W^{s, p}( \Hp^d)$, one via the Riemannian structure and the other by using the spectral theory of $-\Dlt_{\bbH^{d}}$. These two approaches are in fact equivalent.

We now give a definition of $W^{s, p}(\bbH^{d})$ using spectral theory. For $ s \in \bbR$, the fractional Laplacian $(-\lap_{\bbH^{d}})^{\frac{s}{2}}$ is a well-defined operator on (say) $C^{\infty}_{0}(\bbH^{d})$ by the spectral theory of $-\lap_{\bbH^{d}}$. For $f \in C^{\infty}_{0}(\bbH^{d})$, we then set
\ant{
\|f \|_{W^{s, p}(\Hp^d)} = \|(- \Delta_{\Hp^d})^{\frac{s}{2}} f\|_{L^p(\Hp^d)}
}
and define $W^{s, p}( \Hp^d)$ to be the completion of $C^{\infty}_0( \Hp^d)$ under the above norm. For $p=2$ we write $W^{s, 2}(\Hp^d) =: H^s(\Hp^d)$.

In fact, the fractional Laplacian $(- \Delta_{\Hp^d})^{\frac{s}{2}}$ is bounded on $L^p$ for all $s \le 0$ and all $p \in (1, \infty)$. This can be used to show that the above definition of Sobolev spaces is equivalent  to the usual definition of Sobolev spaces using the Riemannian structure of $\Hp^d.$  For a proof of this fact we refer the reader to \cite{Tat01hyp}. In particular, we remark that we have for $p \in (1, \infty)$,
\ant{
\|f\|_{W^{1, p}( \Hp^d)} = \| (- \Delta_{\Hp^d})^{\frac{1}{2}} f\|_{L^p(\Hp^d)} \simeq  \left( \int_{\Hp^d} \abs{ \nabla f}_{\g}^p \mu(dx) \right)^{\frac{1}{p}}
}
where  in local coordinates, $\abs{\nabla f}_{\g}^2 = \g^{ij} \p_i f \p_j f$.  

 
\subsection{Littlewood-Paley theory via heat equation and refined Sobolev inequality}
We have the Sobolev embedding theorem:
\EQ{
W^{s, p}(\Hp^d) \hookrightarrow L^q(\Hp^d) \mfor 1 < p \le q < \infty, \mand \frac{1}{q} = \frac{1}{p}- \frac{s}{d}
}
In particular, we see that
\ant{
H^1(\Hp^d) \hookrightarrow L^{\frac{2d}{d-2}}( \Hp^d). 
}

To prove the Bahouri-G\'erard profile decomposition in Section~\ref{sec: bg} we will need a refined version of the above embedding. For this purpose, we need to define a suitable replacement for the usual Littlewood-Paley frequency projections in Euclidean space. Our approach will be based on the linear heat equation $(\rd_{s} - \lap_{\bbH^{d}}) f = 0$ on $\bbH^{d}$. Such an idea is implicit\footnote{More precisely, in the language below, the operator $P_{\lmb}$ in \cite{IPS} is defined to as $P_{\lmb} := 2 \lmb^{-2} \lap_{\bbH^{d}} e^{-\lmb^{-2} \lap_{\bbH^{d}}}$.} in the definition of the projections $P_{\lmb}$ in \cite{IPS}; in contrast to \cite{IPS}, however, we give an alternative approach which does not depend on the Helgason-Fourier transform, but rather directly on properties of the linear heat equation on $\bbH^{d}$. 

Let $p(x, y ; s)$ be the \emph{heat kernel} on $\bbH^{d}$, i.e., the kernel of the heat semi-group operator $e^{s \lap_{\bbH^{d}}}$ acting on scalar functions on $\bbH^{d}$. Since $\lap_{\bbH^{d}}$ is invariant under the symmetries of $\bbH^{d}$, it follows that 
\begin{equation*}
	e^{s \lap_{\bbH^{d}}} f (x) = \int p(x, y; s) f(y) \, \mu(\ud y) = f \ast p_{s} (x)
\end{equation*}
where $p_{s}(x) = p(0, x; s)$ is a radial function on $\bbH^{d}$. Abusing the terminology a bit, we will also refer to $p_{s}(x)$ as the heat kernel on $\bbH^{d}$. In the following lemma, we collect the properties of $p_{s}(x)$ that we will need in the sequel.
\begin{lem} \label{lem:p-est}
The following statements concerning the heat kernel on $\bbH^{d}$ ($d \geq 2$) hold.
\begin{enumerate}
\item \label{item:p-est:1}
	For $0 < s \leq 2$ and $k = 0, 1, 2$, there exists $N_{k} > 0$ so that the following \emph{short-time heat kernel estimate} holds:
\begin{equation} \label{eq:p-est:shorttime}
	\abs{\rd_{s}^{k} p_{s}(x)} \aleq_{k} s^{-\frac{d}{2} - k} \bb( 1+ \frac{\abs{x}^{2}}{s} \bb)^{N_{k}} e^{-\frac{\abs{x}^{2}}{4s}}.
\end{equation}
where $\abs{x} := \bfd_{\bbH^{d}}(x, \zero)$.
\item \label{item:p-est:2}
	For $s > 2$ and any $1 \leq p \leq \infty$, we have the \emph{long-time estimate}
\begin{equation} \label{eq:p-est:longtime}
	\nrm{e^{s \lap} f}_{L^{p}(\bbH^{d})} = \nrm{p_{s} \ast f}_{L^{p}(\bbH^{d})} \aleq e^{-\frac{(d-1)^{2}}{4} s} \nrm{f}_{L^{p}(\bbH^{d})}.
\end{equation}
\end{enumerate}
\end{lem}

\begin{proof} 
By Young's inequality, part (\ref{item:p-est:2}) is equivalent to showing that 
\begin{equation} \label{eq:p-est:longtime:pf}
\nrm{p_{s}}_{L^{1}(\bbH^{d})} \aleq e^{-\frac{(d-1)^{2}}{4} s}.
\end{equation}
Part (\ref{item:p-est:1}) for $k = 0$ and~\eqref{eq:p-est:longtime:pf} follow immediately from the following classical heat kernel bound on $\bbH^{d}$:
\begin{equation} \label{eq:DM}
	p_{s}(x) \aleq s^{-\frac{d}{2}} e^{-\frac{(d-1)^{2}}{4} s - \frac{\abs{x}^{2}}{4s} - \frac{d-1}{2} \abs{x}} (1+\abs{x}+s)^{\frac{d-3}{2}} (1 + \abs{x}).
\end{equation}
For a proof of~\eqref{eq:DM}, we refer to \cite[Theorem 3.1]{DaviesMandouvalos}. Then from~\eqref{eq:p-est:shorttime} in the case $k =0$, the cases $k=1, 2$ follow by a standard machinery for deriving estimates for the time derivatives of the heat kernel from that of the heat kernel itself; we refer to \cite{Davies} and \cite[Theorems 1.1 and 1.2]{Grigoryan} for details. \qedhere
\end{proof}

\begin{rem}
As stated earlier, Lemma \ref{lem:p-est} essentially contains all the properties for the heat kernel that will be used in this paper (i.e., except for manifestly more elementary ones).
Although in our proof we have relied on the explicit heat kernel bound on $\bbH^{d}$ for simplicity, the estimates in Lemma \ref{lem:p-est} are known to hold in a much wider generality. A sufficient condition is that the manifold $M$ be a smooth $d$-dimensional complete Riemannian manifold satisfying the following properties:
\begin{itemize}
\item The lowest eigenvalue $\lmb_{1}(M)$ of the Laplace-Beltrami operator is strictly positive. 

\item The manifold $M$ has \emph{bounded geometry}, in the sense that there exists $\eps > 0$ such that each geodesic ball $B_{\eps}(x) = \set{y \in M : \bfd_{M}(x, y) < \eps}$ of radius $\eps$ is \emph{uniformly quasi-isometric} to the Euclidean ball $B_{\eps}^{\euc} = \set{y \in \bbR^{d} : \abs{y} < \eps}$. That is, for all $B_{\eps}(x)$ there is a diffeomorphism $\varphi_{\eps} : B_{\eps}^{\euc} \to B_{\eps}(x)$ satisfying $\frac{1}{C} \bfe \leq \varphi_{\eps}^{\ast} \bfg_{M} \leq C \bfe$ with a uniform constant $C > 0$, where $\varphi_{\eps}^{\ast} \bfg_{M}$ is the pullback of the metric of $M$ by $\varphi_{\eps}$ and $\bfe$ is the Euclidean metric on $B_{\eps}^{\euc}$.
\end{itemize}
For a proof, see~\cite[Theorem~15.4, Remark~15.5]{Grigoryan-book} and~\cite[Theorem~3.1]{Grigoryan}.

The subject of heat kernels on manifolds is a vast topic that we cannot survey adequately in this limited space; we recommend the reader to consult the monograph~\cite{Grigoryan-book} for further results and references.
\end{rem}

Next, we proceed to define a version of Littlewood-Paley projections in our context. For $f \in C^{\infty}_{0}(\bbH^{d})$, we may write
\begin{align*}
	f 
	= & - \int_{0}^{\infty} \rd_{s} e^{s \lap_{\bbH^{d}}} f \, \ud s \\
	=&  \lim_{s \to 0} s \rd_{s} e^{s \lap_{\bbH^{d}}} f - \lim_{s \to \infty} s \rd_{s} e^{s \lap_{\bbH^{d}}} f + \int_{0}^{\infty} s^{2} \rd_{s}^{2} e^{s \lap_{\bbH^{d}}} f \, \frac{\ud s}{s} 
\end{align*}
Writing $s \rd_{s} e^{s \lap_{\bbH^{d}}} f = s e^{s \lap_{\bbH^{d}}} \lap_{\bbH^{d}} f$, it is clear that the two limits in the preceding expression vanish for $f \in C^{\infty}_{0}(\bbH^{d})$. Therefore, defining the frequency projection $P_{\lmb}$ as
\begin{equation*}
	P_{\lmb} := 2 \lmb^{-4} \lap_{\bbH^{d}}^{2} e^{\lmb^{-2} \lap_{\bbH^{d}}}
\end{equation*}
for $f \in C^{\infty}_{0}(\bbH^{d})$ we have the identity
\begin{equation*}
	f = \int_{0}^{\infty} P_{\lmb} f \, \frac{\ud \lmb}{\lmb}.
\end{equation*}

%
%

We are now ready to prove a version of refined Sobolev inequality in our setting.
\begin{lem}[Refined Sobolev inequality] \label{lem: se} If $f \in H^{1}(\bbH^{d})$, where $d\geq3$, then
\begin{equation} \label{eq:impSob}
	\nrm{f}_{L^{\frac{2d}{d-2}}(\bbH^{d})}
	\aleq \nrm{\nb f}_{L^{2}(\bbH^{d})}^{\frac{d-2}{d}} 
	\cdot \bb( \sup_{\lmb \geq 1, x \in \bbH^{d}} \lmb^{-\frac{d-2}{2}} \abs{P_{\lmb} f(x)} \bb)^{\frac{2}{d}}.
\end{equation}
\end{lem}

\begin{proof} 
By homogeneity, we may assume that
\begin{equation} \label{eq:impSob:B=1}
	\sup_{\lmb \geq 1, \, x \in \bbH^{d}} \lmb^{-\frac{d-2}{2}} \abs{P_{\lmb} f} \leq 1.
\end{equation}
Furthermore, in view of density of $C^{\infty}_{0}(\bbH^{d})$ in $H^{1}(\bbH^{d})$, we may furthermore assume that $f \in C^{\infty}_{0}(\bbH^{d})$. Given $\lmb_{0} > 0$, we decompose $f$ into two parts as
\begin{equation*}
f = \int_{0}^{\lmb_{0}} P_{\lmb} f \, \frac{\ud \lmb}{\lmb} 
	+ \int_{\lmb_{0}}^{\infty} P_{\lmb} f \, \frac{\ud \lmb}{\lmb}
=: f_{\leq \lmb_{0}} + f_{> \lmb_{0}}.
\end{equation*}
We claim that there exists a function $c(\lmb) : (0, \infty) \to [0, \infty)$ such that for all $\lmb_{0} > 0$, the following inequalities hold:
\begin{align} 
	\nrm{f_{\leq \lmb_{0}}}_{L^{\infty}(\bbH^{d})} 
	& \aleq \lmb_{0}^{\frac{d-2}{2}}, 
	\label{eq:impSob:lowFreq} \\
	\nrm{f_{> \lmb_{0}}}_{L^{2}}^{2}
	& \aleq \int_{\lmb_{0}}^{\infty} c^{2}(\lmb) \, \frac{\ud \lmb}{\lmb},
	\label{eq:impSob:hiFreq:1} \\
	\int_{0}^{\infty} \lmb^{2} c^{2}(\lmb) \, \frac{\ud \lmb}{\lmb} 
	& \aleq \nrm{\nb f}_{L^{2}(\bbH^{d})}^{2}.
	\label{eq:impSob:hiFreq:2} 
\end{align}
Assuming these inequalities, we first prove~\eqref{eq:impSob}. We begin by writing
\begin{align*}
\nrm{f}_{L^{\frac{2d}{d-2}}(\bbH^{d})}^{\frac{2d}{d-2}}
= \frac{2d}{d-2} \int_{0}^{\infty} \alp^{\frac{2d}{d-2}} \mu(\set{\abs{f} > \alp}) \, \frac{\ud \alp}{\alp}.
\end{align*}
Choosing $\lmb_{0}(\alp) = c \alp^{\frac{2}{d-2}}$ so that $\nrm{f_{< \lmb_{0}(\alp)}}_{L^{\infty}} < \frac{\alp}{2}$, it follows that $\set{\abs{f} > \alp} \subseteq \set{\abs{f_{>\lmb_{0}(\alp)}} > \frac{\alp}{2}}$. Using also~\eqref{eq:impSob:hiFreq:1} and Chebyshev's inequality, the preceding expression can be estimated by
\begin{align*}
	& \aleq \int_{0}^{\infty} \alp^{\frac{2d}{d-2} - 2} \nrm{f_{> \lmb_{0}(\alp)}}_{L^{2}(\bbH^{d})}^{2} \, \frac{\ud \alp}{\alp} \\
	& \aleq \int_{0}^{\infty} \int_{c \alp^{\frac{2}{d-2}}}^{\infty} \alp^{\frac{2d}{d-2} - 2} c^{2}(\lmb) \, \frac{\ud \lmb}{\lmb} \frac{\ud \alp}{\alp} \\
	& \aleq \int_{0}^{\infty} \bb( \int_{0}^{(\frac{\lmb}{c})^{\frac{d-2}{2}}} \alp^{\frac{2d}{d-2} - 2} \, \frac{\ud \alp}{\alp} \bb) c^{2}(\lmb) \, \frac{\ud \lmb}{\lmb} \\
	& \aleq \int_{0}^{\infty} \lmb^{2} c^{2}(\lmb) \frac{\ud \lmb}{\lmb}.
\end{align*}
By~\eqref{eq:impSob:hiFreq:2}, the last line is bounded by $\aleq \nrm{\nb f}_{L^{2}(\bbH^{d})}$, which concludes the proof of~\eqref{eq:impSob}.

It remains to prove~\eqref{eq:impSob:lowFreq}--\eqref{eq:impSob:hiFreq:2}. For~\eqref{eq:impSob:lowFreq}, we begin with the obvious bound
\begin{equation*}
	\nrm{f_{\leq \lmb_{0}}}_{L^{\infty}(\bbH^{d})}
	\leq \int_{0}^{\lmb_{0}} \nrm{P_{\lmb} f}_{L^{\infty}(\bbH^{d})} \, \frac{\ud \lmb}{\lmb}
	\aleq_{d} \lmb_{0}^{\frac{d-2}{2}} \sup_{\lmb > 0} \lmb^{-\frac{d-2}{2}}\nrm{P_{\lmb} f}_{L^{\infty}(\bbH^{d})}.
\end{equation*}
Then by the long time estimate for the heat kernel, i.e., Lemma \ref{lem:p-est}(\ref{item:p-est:2}), we have
\begin{align*}
	\sup_{\lmb \in (0, 1)} \lmb^{-\frac{d-2}{2}} \nrm{P_{\lmb} f}_{L^{\infty}(\bbH^{d})}
	& \aleq \sup_{\lmb \in (0, 1)} \lmb^{-\frac{d-2}{2} - 4} \nrm{e^{(\lmb^{-2} - \frac{1}{4}) \lap_{\bbH^{d}}}P_{2} f}_{L^{\infty}(\bbH^{d})} \\
	& \aleq \nrm{P_{2} f}_{L^{\infty}(\bbH^{d})} \aleq 1
\end{align*}
where we have used~\eqref{eq:impSob:B=1} in the last inequality. Combined with the previous inequality,~\eqref{eq:impSob:lowFreq} follows. 

Next, in order to prove~\eqref{eq:impSob:hiFreq:1}, we begin by writing
\begin{align*}
	\nrm{P_{> \lmb_{0}} f}_{L^{2}}^{2}
= & \ang{\int_{\lmb_{0}}^{\infty} P_{\lmb_{1}} f \frac{\ud \lmb_{1}}{\lmb_{1}} \, \mid \,
		\int_{\lmb_{0}}^{\infty} P_{\lmb_{2}} f \frac{\ud \lmb_{2}}{\lmb_{2}}}_{L^{2}}.
\end{align*}
By symmetry, without loss of generality, we may restrict to the parameter range $\lmb_{0} < \lmb_{1} < \lmb_{2}$. Then the last line can be estimated by
\begin{align*}
\aleq & \int_{\lmb_{0}}^{\infty} \int_{\lmb_{1}}^{\infty}
\ang{	\lmb_{1}^{-4} \lap_{\bbH^{d}}^{2} e^{\lmb_{1}^{-2} \lap_{\bbH^{d}}} f 
	\, \mid \,
		\lmb_{2}^{-4} \lap_{\bbH^{d}}^{2} e^{\lmb_{2}^{-2} \lap_{\bbH^{d}}} f }_{L^{2}} \frac{\ud \lmb_{2}}{\lmb_{2}} \, \frac{\ud \lmb_{1}}{\lmb_{1}}  \\
\aleq & \int_{\lmb_{0}}^{\infty} \int_{\lmb_{1}}^{\infty}
\bb( \frac{\lmb_{1}}{\lmb_{2}} \bb)^{2} 
	\ang{	\lmb_{1}^{-6} \lap_{\bbH^{d}}^{3} e^{\lmb_{1}^{-2} \lap_{\bbH^{d}}} f 
	\, \mid \,
		\lmb_{2}^{-2} \lap_{\bbH^{d}} e^{\lmb_{2}^{-2} \lap_{\bbH^{d}}} f }_{L^{2}} \frac{\ud \lmb_{2}}{\lmb_{2}} \, \frac{\ud \lmb_{1}}{\lmb_{1}} \\
\aleq & \int_{\lmb_{0}}^{\infty} \int_{\lmb_{1}}^{\infty} 
\bb( \frac{\lmb_{1}}{\lmb_{2}} \bb)^{2} c(\lmb_{1}) c(\lmb_{2}) 
	\,\frac{\ud \lmb_{2}}{\lmb_{2}} \, \frac{\ud \lmb_{1}}{\lmb_{1}} 
	\aleq \int_{\lmb_{0}}^{\infty} c^{2} (\lmb) \frac{\ud \lmb}{\lmb}
\end{align*}
where $c(\lmb)$ is defined to be
\begin{equation} \label{eq:impSob:c}
	c^{2}(\lmb) := \nrm{\lmb^{-2} \lap_{\bbH^{d}} e^{\lmb^{-2} \lap_{\bbH^{d}}} f}_{L^{2}(\bbH^{d})}^{2}
			+ \nrm{\lmb^{-6} \lap_{\bbH^{d}}^{3} e^{\lmb^{-2} \lap_{\bbH^{d}}} f}_{L^{2}(\bbH^{d})}^{2}.
\end{equation}

Finally, we turn to the proof of~\eqref{eq:impSob:hiFreq:2}. Note the following identities for $f \in C^{\infty}_{0}(\bbH^{d})$, which are proved by simple integration by parts arguments:
\begin{align*}
	\rd_{s} \nrm{\nb e^{s \lap_{\bbH^{d}}} f}_{L^{2}(\bbH^{d})}^{2}
	=& -2 \nrm{\lap_{\bbH^{d}} e^{s \lap_{\bbH^{d}}} f}_{L^{2}(\bbH^{d})}^{2} \\
	\rd_{s} (s^{4} \nrm{\nb \lap_{\bbH^{d}}^{2} e^{s \lap_{\bbH^{d}}} f}_{L^{2}(\bbH^{d})}^{2})
	=& - 2 s^{4} \nrm{\lap_{\bbH^{d}}^{3} e^{s \lap_{\bbH^{d}}} f}_{L^{2}(\bbH^{d})}^{2} \\
		& + 4 s^{3} \brk{\lap_{\bbH^{d}}^2 e^{s \lap_{\bbH^{d}}} f \, \mid \, 
				\lap_{\bbH^{d}}^{3} e^{s\lap_{\bbH^{d}}}f}_{L^{2}}
\end{align*}
For a constant $\dlt_{E} > 0$ to be determined, consider the quantity 
\begin{equation*}
E(s) := \nrm{\nb e^{s \lap_{\bbH^{d}}}f}_{L^{2}(\bbH^{d})}^{2} + \dlt_{E} s^{4} \nrm{\nb \lap_{\bbH^{d}}^{2}e^{s \lap_{\bbH^{d}}}f}_{L^{2}(\bbH^{d})}^{2}.
\end{equation*}
Applying the preceding identities and Cauchy-Schwartz, we have for sufficiently small $\dlt_{E} > 0$
\begin{equation*}
	-\rd_{s} E(s)
	\ageq_{\dlt_{E}} \nrm{\lap_{\bbH^{d}} e^{s \lap_{\bbH^{d}}} f}_{L^{2}(\bbH^{d})}^{2}
				+s^{4} \nrm{\lap_{\bbH^{d}}^{3} e^{s \lap_{\bbH^{d}}} f}_{L^{2}(\bbH^{d})}^{2}.
\end{equation*}
On the other hand, note that $E(0) = \nrm{\nb f}_{L^{2}(\bbH^{d})}^{2}$ and $\lim_{s \to \infty} E(s) = 0$ as $f \in C^{\infty}_{0}(\bbH^{d})$. Therefore, 
\begin{equation*}
\int_{0}^{\infty} (s^{-\frac{1}{2}} \nrm{s \lap_{\bbH^{d}} e^{s \lap_{\bbH^{d}}} f}_{L^{2}(\bbH^{d})})^{2}
				+ (s^{-\frac{1}{2}} \nrm{s^{3} \lap_{\bbH^{d}}^{3} e^{s \lap_{\bbH^{d}}} f}_{L^{2}(\bbH^{d})})^{2} \, \frac{\ud s}{s} \aleq \nrm{\nb f}_{L^{2}(\bbH^{d})}^{2}.
\end{equation*}
Making the change of variable $\lmb := s^{-\frac{1}{2}}$ and recalling the definition~\eqref{eq:impSob:c} of $c(\lmb)$,~\eqref{eq:impSob:hiFreq:2} follows. \qedhere
\end{proof}

For $k = 0, 1, 2, \ldots$, define the radial function ${}^{(k)} \PP_{\lmb}$ on $\bbH^{d}$ by
\begin{equation*}
	{}^{(k)} \PP_{\lmb}(r) := 2 \lmb^{-2k} \lap_{\bbH^{d}}^{k} p_{\lmb^{-2}} (r).
\end{equation*}
Note that ${}^{(2)}\PP_{\lmb}$ is the convolution kernel of the operator $P_{\lmb}$. More precisely, we have
\begin{equation} \label{}
	  P_{\lmb} f(x) = {}^{(2)} \calP_{\la} \ast f(x).
\end{equation}
Since the heat kernel $p_{s}(r)$ solves the linear heat equation $\rd_{s} p_{s} = \lap_{\bbH^{d}} p_{s}$, we have the identity
\begin{equation} \label{eq:PP}
	{}^{(k)} \calP_{\lmb}(r) = 2 (s^{k} \rd_{s}^{k} p_{s}) \bb\vert_{s = \lmb^{-2}} (r)
\end{equation}
In particular, by Lemma \ref{lem:p-est}(\ref{item:p-est:1}), for $k = 0, 1, 2$ we have the pointwise estimate
\begin{equation} \label{eq:PPest:1}
	{}^{(k)} \calP_{\lmb}(r) \aleq \lmb^{d} \bb( 1 + \lmb^{2} r^{2} \bb)^{N_{k}} e^{-\frac{\lmb^{2} r^{2}}{4}}.
\end{equation}

An important role will be played by the functions ${}^{(1)} \PP_{\lmb}$ in the proof of the linear profile decomposition theorem (Theorem \ref{thm: BG}), in particular for identifying nontrivial profiles. The relevance of ${}^{(1)} \PP_{\lmb}$ arises from the following simple chain of identities: For $\lmb \in [1, \infty)$ and $h \in \GG$, we have
\begin{equation} \label{eq:PP1}
\begin{aligned}
	P_{\lmb} f(h \cdot \zero) = {}^{(2)} \PP_{\lmb} \ast (\tau_{h} f) (\zero) 
	& = \ang{\tau_{h} f \mid {}^{(2)} \PP_{\lmb}}_{L^{2}(\bbH^{d})} \\
	& = - \ang{\tau_{h} f \mid  \lmb^{-2} \, {}^{(1)} \PP_{\lmb}}_{H^{1}(\bbH^{d})} .
\end{aligned}
\end{equation}
In the next lemma, we collect properties of ${}^{(1)} \PP_{\lmb}$ that will be needed in the sequel.
\begin{lem} \label{lem:concPP}
Let $\set{\lmb_{n}} \subseteq [1, \infty)$ be a sequence such that $\lmb_{n} \to \infty$. Let $\eta \in C^{\infty}_{0}(\bbR)$ be an even function such that $\eta = 1$ on $(-1, 1)$ and $\eta = 0$ outside $(-2, 2)$. For any $R >0$, define $\eta_{R} := \eta(\cdot / R)$. Then the following statements hold:
\begin{enumerate}
\item \label{item:concPP:1}
	For all $n$ and $\gmm=0,1,2$, we have the uniform bound
\begin{equation} \label{eq:concPP:1}
	\lmb_{n}^{-\frac{d+2\gmm}{2}} \nrm{{}^{(1)} \PP_{\lmb_{n}}(r)}_{H^{\gmm}(\bbH^{d})} \aleq 1.
\end{equation}
\item \label{item:concPP:2} 
	Given any $\eps > 0$, there exists $R > 0$ such that for all $n$, we have
\begin{equation}
	 \lmb_{n}^{- \frac{d+2}{2}} \nrm{(1-\eta_{\lmb_{n}^{-1} R})(r) {}^{(1)} \PP_{\lmb_{n}}(r)}_{H^{1}(\bbH^{d})} < \eps.
\end{equation}

\item \label{item:concPP:3}
	Fix any $R \geq 1$. Let us view each
\begin{equation*}
	f_{n}(r) := \lmb_{n}^{-d} \eta_{R}(r) {}^{(1)}\PP_{\lmb_{n}}(r / \lmb_{n})
\end{equation*}
as a radial function on $\bbR^{d}$ with $r = \abs{x}$; equivalently, $f_{n}$ is the pullback of $\lmb_{n}^{-\frac{d}{2}} \eta_{\lmb_{n}^{-1} R} {}^{(1)}\PP_{\lmb_{n}}$ via $\Psi$, suitably rescaled. Then there exists $\PP_{\infty} \in \dot{H}^{1}(\bbR^{d})$ such that along a subsequence, which we still denote by $f_{n}$, we have the strong convergence $f_{n} \to  \PP_{\infty}(r)$ in $\dot{H}^{1}(\bbR^{d})$. Moreover, $\PP_{\infty}$ satisfies
\begin{equation} 
	\nrm{\PP_{\infty}}_{\dot{H}^{1}(\bbR^{d})} \leq C \mand \supp \, \PP_{\infty} \subseteq \set{r \leq 2R}
\end{equation}
where $C$ is an absolute constant independent of $\lmb_{n}$.
\end{enumerate}

\end{lem}

\begin{proof} 
We begin by proving part (\ref{item:concPP:1}). Note that ${}^{(2)} \PP_{\lmb_{n}} = \lmb_{n}^{-2} \lap_{\bbH^{d}} {}^{(1)} \PP_{\lmb_{n}}$. Therefore, taking the $L^{2}(\bbH^{d})$ norm of \eqref{eq:PPest:1} for $k = 1, 2$, we have the uniform bound
\begin{equation} \label{eq:concPP:pf:0}
	\lmb_{n}^{-\frac{d}{2}} \nrm{{}^{(1)} \PP_{\lmb_{n}}}_{L^{2}(\bbH^{d})} +\lmb_{n}^{-\frac{d+4}{2}} \nrm{\lap_{\bbH^{d}} {}^{(1)} \PP_{\lmb_{n}}}_{L^{2}(\bbH^{d})} \aleq 1.
\end{equation}
We claim that \eqref{eq:concPP:pf:0} implies the $H^{1}$ bound in \eqref{eq:concPP:1}. Indeed, by an integration by parts followed by Cauchy-Schwarz and \eqref{eq:concPP:pf:0}, we have
\begin{align*}
	\lmb_{n}^{-(d+2)}\nrm{{}^{(1)} \PP_{\lmb_{n}}}_{H^{1}(\bbH^{d})}^{2} 
	=& \lmb_{n}^{-(d+2)} \ang{{}^{(1)} \PP_{\lmb_{n}} \mid {}^{(1)} \PP_{\lmb_{n}}}_{H^{1}(\bbH^{d})}  \\
	=& - \ang{ \lmb_{n}^{-\frac{d}{2}}{}^{(1)} \PP_{\lmb_{n}} \mid \lmb_{n}^{-\frac{d+2}{4}}\lap_{\bbH^{d}} {}^{(1)} \PP_{\lmb_{n}}}_{L^{2}(\bbH^{d})} \aleq 1.
\end{align*}
Finally, by standard elliptic regularity on $\bbH^{d}$ (which may be proved by a simple integration by parts argument combined with the $H^{1}$ bound that we just proved), we obtain the desired bound for $\nrm{{}^{(1)} \PP_{\lmb_{n}}}_{H^{2}(\bbH^{d})}$. 

Next we establish part (\ref{item:concPP:2}). In what follows, the notation $o_{R}(1)$ will always refer to a quantity which vanishes \emph{uniformly} in $n$ as $R \to \infty$.
Using the pointwise bound~\eqref{eq:PPest:1} (in particular, the Gaussian decay as $\lmb_{n} r \to \infty$) for $k= 1, 2$ and taking the $L^{2}$ norm over the set $\set{r \geq R}$, it is not difficult to see that
\begin{equation} \label{eq:concPP:pf:1}
	\lmb_{n}^{- \frac{d}{2}}\nrm{{}^{(1)} \calP_{\lmb_{n}}(r)}_{L^{2}(\set{\lmb_{n} r \geq R})} 
	+ \lmb_{n}^{- \frac{d+4}{2}} \nrm{\lap_{\bbH^{d}} {}^{(1)} \calP_{\lmb_{n}}(r)}_{L^{2}(\set{\lmb_{n} r \geq R})} = o_{R}(1).
\end{equation}
We omit the routine details. We will show that~\eqref{eq:concPP:pf:1} imply part (\ref{item:concPP:2}) by an interpolation argument. We begin by writing
\begin{align*}
& \hskip-2em
	 \lmb_{n}^{- \frac{d+2}{2}} \nrm{(1-\eta_{\lmb_{n}^{-1} R})(r) {}^{(1)}\PP_{\lmb_{n}}(r)}_{H^{1}(\bbH^{d})} \\
	 = & \lmb_{n}^{- \frac{d}{2}} R^{-1} \nrm{\eta'_{\lmb_{n}^{-1} R}(r) {}^{(1)}\PP_{\lmb_{n}}(r)}_{L^{2}(\bbH^{d})}
	 + \lmb_{n}^{- \frac{d+2}{2}} \nrm{(1-\eta_{\lmb_{n}^{-1} R})(r) \nb {}^{(1)}\PP_{\lmb_{n}}(r)}_{L^{2}(\bbH^{d})}
\end{align*}
where the first term on the last line is $o_{R}(1)$, thanks to~\eqref{eq:concPP:pf:1}. To treat the second term, we take its square and perform an integration by parts as follows:
\begin{align}
& \hskip-2em
\lmb_{n}^{- (d+2) } \nrm{(1-\eta_{\lmb_{n}^{-1} R})(r) \nb {}^{(1)} \PP_{\lmb_{n}}(r)}_{L^{2}(\bbH^{d})}^{2} \notag \\
= &\lmb_{n}^{- (d+2)} \int_{\bbH^{d}} (1-\eta_{\lmb_{n}^{-1} R})^{2} \bfg^{ij} (\nb_{i} {}^{(1)} \PP_{\lmb_{n}}) (\nb_{j} {}^{(1)} \PP_{\lmb_{n}}) \, \mu(\ud x) \notag \\
= &  \lmb_{n}^{- (d+1)} R^{-1} \int_{\bbH^{d}} 2 \bfg^{ij} (1-\eta_{\lmb_{n}^{-1} R}) (\nb_{i} \eta)_{\lmb_{n}^{-1} R}  ({}^{(1)} \PP_{\lmb_{n}}) (\nb_{j} {}^{(1)}\PP_{\lmb_{n}}) \, \mu(\ud x) \label{eq:concPP:pf:3}  \\
& - \lmb_{n}^{- (d+2)} \int_{\bbH^{d}} (1-\eta_{\lmb_{n}^{-1} R})^{2} ({}^{(1)} \PP_{\lmb_{n}}) (\lap_{\bbH^{d}} {}^{(1)} \PP_{\lmb_{n}}) \, \mu(\ud x) \label{eq:concPP:pf:4}
\end{align}
where $(\nb_{i} \eta)_{\lmb_{n}^{-1} R}(x) := (\nb_{i} \eta)(\lmb_{n} x / R)$.
The term~\eqref{eq:concPP:pf:4} is $o_{R}(1)$ by Cauchy-Schwarz and~\eqref{eq:concPP:pf:1}. For~\eqref{eq:concPP:pf:3}, we apply Cauchy-Schwarz to estimate
\begin{equation*}
	\eqref{eq:concPP:pf:3} \aleq ( \lmb_{n}^{- \frac{d+2}{2}} \nrm{(1-\eta_{\lmb_{n}^{-1} R})(r) \nb {}^{(1)} \PP_{\lmb_{n}}(r)}_{L^{2}(\bbH^{d})} ) (\lmb_{n}^{-\frac{d}{2}} R^{-1} \nrm{{}^{(1)} \PP_{\lmb_{n}}}_{L^{2}(\set{\lmb_{n} r \geq R})})
\end{equation*}
Then the first factor can be absorbed into the left-hand side by Cauchy, and the second factor is $o_{R}(1)$ by~\eqref{eq:concPP:pf:1}. This completes the proof of part (\ref{item:concPP:2}).

We now turn to part (\ref{item:concPP:3}). Since every $f_{n}$ is supported in a common compact set $\set{r \leq 2 R} \subseteq \bbR^{d}$, the desired statement would follow, by Rellich-Kondrachov, from the uniform bound
\begin{equation} \label{}
	\nrm{f_{n}}_{L^{2}(\bbR^{d})} + \nrm{\lap_{\bbR^{d}} f_{n}}_{L^{2}(\bbR^{d})} \aleq 1 < \infty,
\end{equation}
where the implicit constant is independent of both $R \geq 1$ and $n = 1, 2, \cdots$. This bound can be derived from \eqref{eq:concPP:1}. We omit the details. \qedhere
\end{proof}

\part{Profiles for linear waves on hyperbolic space}\label{p:lin}
The grand aim of the first part of the paper is to formulate and prove an analogue of the Bahouri-G\'erard linear profile decomposition theorem from \cite{BG} in the setting of a linear wave equation with potential on $\bbR \times \bbH^{d}$ (Theorem~\ref{thm: BG}). Along the way, we also develop linear approximation theory for traveling and concentrating linear profiles, by comparing them with solutions to the potential-less linear wave equation on $\bbR \times \bbH^{d}$ and the linear wave equation on $\bbR^{1+d}$, respectively (see Section~\ref{subsec:linApprox}).  

In this part we have strived to develop general tools that could be applied to a variety of nonlinear problems. A specific nonlinear application is given in Part \ref{p:nonlin} of this paper, where we study the energy critical defocusing semilinear wave equation~\eqref{u eq}, possibly perturbed by a repulsive potential. 
Another application will be given in a forthcoming work, where we address the asymptotic behavior of large energy equivariant wave maps from the hyperbolic plane.

\section{Dispersion, Strichartz estimates, and linear approximations}\label{sec:lin}

\subsection{Dispersive theory for free waves on $\Hp^d$}
We begin by recalling  Strichartz estimates for the free wave equation on $\R \times \Hp^d$ from \cite{AP, MT11, MTay12}. Consider the inhomogeneous wave equation on $\R \times \Hp^d$:
 \EQ{ \label{linear wave}
 &\Box_{\bbH^{d}} u := u_{tt} - \Delta_{\bbH^{d}} u = F\\
& \vec u(0) = ( f, g)
 }
 We will say that a triple $(p, q, \ga)$, is \emph{hyperbolic-admissible} if
 \EQ{
 & p \ge 2, q >2\\
 & \ga = \begin{cases} 
\frac{d+1}{2} (\frac{1}{2} - \frac{1}{q}) \mif \frac{2}{p} + \frac{d-1}{q} \ge \frac{d-1}{2} \\ 
d (\frac{1}{2} - \frac{1}{q}) - \frac{1}{p}   \mif \frac{2}{p}+ \frac{d-1}{q} \le \frac{d-1}{2} 
 \end{cases}
 }
  We also include the energy estimates $(p, q, \gamma) = (\infty, 2, 0)$.  Note that we are not allowing the endpoint $q = \infty$. The following Strichartz estimates were proved   independently by Metcalfe and Taylor,~\cite{MT11, MTay12}, and by Anker and Pierfelice~\cite{AP}. In fact, these estimates hold for a larger class of admissible exponents than in the Euclidean case, a phenomena which can be attributed to  the exponential volume growth of concentric spheres in $\Hp^d$. However, our application of Strichartz estimates to the energy critical semi-linear wave equation  only requires the usual range of exponents which are also admissible in the case of free waves on Euclidean space.

 \begin{prop}\cite{MT11, AP, MTay12} \label{strich}
 Let $\vec u(t)$ be a solution to~\eqref{linear wave} with initial data $\vec u(0) = (f, g)$ and let $I \subseteq \R$ be any time interval. Then for $(p, q, \gamma)$ and $(a, b, \s)$ hyperbolic-admissible we have the estimate
 \EQ{
 \| \nabla_{t, x} u\|_{L^p( I; W^{-\gamma, q}(\Hp^d))} \lesssim \| (f, g)\|_{H^1 \times L^2( \Hp^d)} + \|F\|_{L^{a'}(I; W^{ \s, b'}(\Hp^d))}
 }
 \end{prop}
 
The Strichartz estimates proved above are a consequence of the fundamental dispersive estimate proved in~\cite{AP, MTay12}, which we will also use in Section~\ref{sec:nonlinApprox}. Note the improved rate of decay for large times below. 
\begin{prop} \label{prop: disp} \cite{AP, MTay12}
Let $d \geq 3$. 
Let $q \in (2, \infty)$ and $\sgm = (d+1)( \frac{1}{2} - \frac{1}{q} )$. 
\begin{enumerate}
\item For $0 < \abs{t} \leq 2$, the following \emph{short time dispersive estimate} holds:
\begin{equation} \label{eq:shortDisp}
	\nrm{e^{\pm i t \sqrt{-\De_{\bbH^{d}}}}h}_{L^{q}(\bbH^{d})} \aleq_{q} \abs{t}^{-(d-1) (\frac{1}{2} - \frac{1}{q})} \nrm{h}_{W^{\sgm, q'}(\bbH^{d})}
	\end{equation}
\item For $\abs{t} \geq 2$, the following \emph{long time dispersive estimate} holds: 
\begin{equation} \label{eq:longDisp}
	\nrm{e^{\pm i t \sqrt{-\De_{\bbH^{d}}}}h}_{L^{q}(\bbH^{d})} \aleq_{q} \abs{t}^{-\frac{3}{2}} \nrm{h}_{W^{\sgm, q'}(\bbH^{d})}
\end{equation}
\end{enumerate}
\end{prop}

\begin{rem} 
The short time estimate~\eqref{eq:shortDisp} is exactly the same as what one can obtain on the Minkowski space $\bbR^{1+d}$ by interpolating between the $L^{2} \to L^{2}$ energy estimate and the $L^{1} \to L^{\infty}$ dispersive estimate. In particular, this estimate is consistent with dimensional analysis. On the other hand, the long time estimate~\eqref{eq:longDisp} is \emph{not} consistent with dimensional analysis, and exhibits a \emph{better} decay rate as $q \to 2$.
\end{rem}

\subsection{Dispersive theory for linear waves with potential}


Next, we consider the linear wave equation with potential on $\bbR \times \bbH^{d}$:
\begin{equation} \label{eq:Vwave}
	\Box_{\bbH^{d}} u + V u =0.
\end{equation}

Our aim will be to establish a Bahouri-G\'erard type profile decomposition for solutions to the equation~\eqref{eq:Vwave}. We first introduce some notation to be used below. Let
\begin{equation*}
	D_{V} := \sqrt{- \De_{\bbH^{d}} + V}
\end{equation*}
and define the $H^{1}_{V}$ norm for $f \in C^{\infty}_{0}(\bbH^{d})$ as follows:
\begin{equation*}
	\nrm{f}_{H^{1}_{V}} := \nrm{D_{V} f}_{L^{2}(\bbH^{d})}.
\end{equation*}
As usual, we define the space $H^{1}_{V}(\bbH^{d})$ by taking the completion of $C^{\infty}_{0}(\bbH^{d})$ under the $H^{1}_{V}$ norm.

We also define the propagator $S_{V}$ for the linear wave equation with potential~\eqref{eq:Vwave} as follows: For $\vec{u}(0) = (u_{0}, u_{1})$, let 
\begin{equation*}
S_{V}(t) \vec{u}(0) := (S_{V, 0}(t) \vec{u}(0), S_{V, 1}(t) \vec{u}(0)), 
\end{equation*}
where
\begin{equation} \label{eq:propagator4VWave}
\begin{aligned}
	S_{V, 0}(t) \vec{u}(0) := & \cos(t D_{V}) u_{0} + \frac{\sin (t D_{V})}{D_{V}} u_{1}, \\
	S_{V, 1}(t) \vec{u}(0) := & \rd_{t} S_{V, 0}(t) = - \sin (t D_{V}) D_{V} u_{0} + \cos (t D_{V}) u_{1}.
\end{aligned}
\end{equation}

When $V\equiv0$ we replace $S_V$ (respectively $S_{V,0}$ and $S_{V,1}$) by $S_\hyp$ (respectively $S_{\hyp, 0}$ and $S_{\hyp,1}$).

We make the following assumptions on the potential $V$:

\begin{enumerate}
\item {\it (Decay of $V$)} We assume that $V$ decays exponentially towards spatial infinity, i.e., for some $A_{1}, \alp_{1} > 0$, we have
\begin{equation} \label{eq:Vdecay}
	\abs{V (x)} \leq A_{1} e^{-\alp_{1} r}.
\end{equation}

\item {\it (Positivity of energy)} We assume that the operator $-\De_{\bbH^{d}} + V$ is self-adjoint and positive on $H^{1}(\bbH^{d})$. More precisely, there exists $\mu_{V} > 0$ such that for every $f \in H^{1}(\bbH^{d})$, we have
\begin{equation} \label{eq:Vpos}
	\brk{(-\De_{\bbH^{d}} + V) f, f}_{L^{2}} \geq \mu^{2}_{V} \nrm{f}^{2}_{L^{2}}.
\end{equation}

\item {\it (Integrated local energy decay)} We assume that the following \emph{integrated local energy decay estimate} holds: There exist constants $A_{2} > 0$ and $\alp_{2} > 0$ such that for every $f \in C^{\infty}_{0}(\bbH^{d})$, we have
\begin{equation} \label{eq:VLED}
	\nrm{e^{-\alp_{2} r} e^{\pm i t D_{V}} f}_{L^{2}_{t,x}} \leq A_{2} \nrm{f}_{L^{2}}.
\end{equation}
\item {\it (Strichartz estimates)} We assume that non-sharp Strichartz estimates hold for the half-wave propagators $e^{\pm i t D_{V}}$. More precisely, if $(p, q, \gmm)$ is a hyperbolic admissible triple, then we assume 
\begin{equation} \label{eq:Str4halfVWave:1}
	 \nrm{e^{\pm i t D_{V}} f}_{L^{q}_{t} W^{-\gmm, r}_{x}(\bbR \times \bbH^{d})} \leq A_{3} \nrm{f}_{L^{2}}
\end{equation}
and
\begin{equation} \label{eq:Str4halfVWave:2}
	\nrm{\sqrt{-\lap_{\bbH^{d}}} \, e^{\pm i t D_{V}} f}_{L^{q}_{t} W^{-\gmm, r}_{x}(\bbR \times \bbH^{d})}
	\leq A_{3} \nrm{D_{V} f}_{L^{2}}
\end{equation}
hold for some constant $A_{3} = A_{3}(p, q) > 0$.
\end{enumerate}

\begin{rem} 
The assumptions above on $V$ are by no means minimal for the theorems in this paper to hold. With an eye towards further applications, they are designed to be merely general enough to include the class of potentials that arise in \cite{LOS1} as a result of linearization of the equivariant wave maps equation about a family of stationary solutions.
\end{rem}

\begin{rem}
By the assumptions~\eqref{eq:Vdecay} and~\eqref{eq:Vpos}, the operator $-\De_{\bbH^{d}} + V$ is self-adjoint on its form domain $H^{1}(\bbH^{d})$ and has non-negative spectrum; see~\cite[Section~X.2]{RS2}. These facts justify the use of functional calculus for $-\De_{\bbH^{d}} + V$, e.g., in the definition of $D_{V} = \sqrt{-\De_{\bbH^{d}} + V}$ and $S_{V}(t)$.
It is possible to derive further spectral properties of $-\De_{\bbH^{d}} + V$ from the assumptions, but we will refrain from doing so since they will not be used in the sequel. 

On the other hand, we note that under a further condition\footnote{The condition required for our proof of Lemma~\ref{lem: Strichartz proof} is $\alp_{1} \geq 1 + \alp_{2}$, but we do not claim that this is optimal.} on the exponents $\alp_{1}$ and $\alp_{2}$, the Strichartz estimates~\eqref{eq:Str4halfVWave:1},~\eqref{eq:Str4halfVWave:2} can be deduced from the first three assumptions~\eqref{eq:Vdecay}--\eqref{eq:VLED}. We refer to the argument in the proof of Lemma~\ref{lem: Strichartz proof} below, which is based on earlier arguments in \cite{LOS1, LS, RodS}.
\end{rem}


A simple example of a class of potentials that satisfy the assumptions \eqref{eq:Vdecay}--\eqref{eq:Str4halfVWave:2} above is that of \emph{repulsive potentials with compact support}, i.e., potentials that obey the assumptions of Theorem~\ref{main}. We record this fact in the following lemma:
\begin{lem} \label{lem:defocusingV}
Let $V$ be a smooth, compactly supported potential which is \emph{repulsive} in the sense that
\begin{equation} \label{eq:defocusingV}
	V \geq 0 \mand \rd_{r} V \leq 0.
\end{equation}
where $\rd_{r}$ is the radial directional derivative in the polar coordinates $(r, \omg)$ on $\bbH^{d}$. Then $V$ obeys the assumptions \eqref{eq:Vdecay}--\eqref{eq:Str4halfVWave:2}.
\end{lem}
That such a potential $V$ satisfies \eqref{eq:Vdecay} and \eqref{eq:Vpos} is obvious; the only nontrivial parts are that the integrated local energy decay \eqref{eq:VLED} and the Strichartz estimates \eqref{eq:Str4halfVWave:1} and \eqref{eq:Str4halfVWave:2} hold. We defer the proof of these facts, hence also that of Lemma~\ref{lem:defocusingV}, until Section~\ref{sec:morawetz}.
 
 In the following lemma, we record few immediate consequences of the assumptions \eqref{eq:Vdecay}--\eqref{eq:Str4halfVWave:2}.
\begin{lem} \label{lem:basicVWave}
Let $V$ satisfy \eqref{eq:Vdecay}--\eqref{eq:Str4halfVWave:2}. Then the following statements hold.
\begin{enumerate}
\item \label{item:basicVWave:1}
	The $\nrm{\cdot}_{H^{1}}$ norm is equivalent to $\nrm{\cdot}_{H^{1}_{V}}$, i.e., for every $f \in C^{\infty}_{0}(\bbH^{d})$ we have
\begin{equation} \label{eq:H1equiv}
	\nrm{f}_{H^{1}} \aleq_{A_{1}, \mu_{V}} \nrm{f}_{H^{1}_{V}} \aleq_{A_{1}} \nrm{f}_{H^{1}}.
\end{equation}

\item \label{item:basicVWave:2}
	Suppose $u$ is a solution of $\Box_{\Hp^d}u+Vu=F$ with initial data $\vec u(0)=(f, g).$ Let $I \subseteq \bbR$ be an interval, and $(q, r, \gmm)$ be a hyperbolic admissible triple. Then we have
\begin{equation}
\begin{aligned}
& \hskip-2em
	\nrm{u}_{L^{q}_{t} (I; W^{1-\gmm, r}(\bbH^{d}))} + \nrm{\rd_{t} u}_{L^{q}_{t} (I; W^{-\gmm, r}(\bbH^{d}))} \\
	& \leq C A_{3} (\nrm{(f, g)}_{H^{1} \times L^{2}(\bbH^{d})} + \nrm{F}_{L^{1}_{t} (I; L^{2}(\bbH^{d}))})
\end{aligned}
\end{equation}
\end{enumerate}
\end{lem}

\begin{proof} 
The second inequality in~\eqref{eq:H1equiv} follows from the embedding $H^{1}(\bbH^{d}) \subseteq L^{2}(\bbH^{d})$ and the fact that $V$ is uniformly bounded. The first inequality follows by the same proof, this time using~\eqref{eq:Vpos}. For the Strichartz inequalities note that by Duhamel's principle, the unique solution $u$ can be represented as
\begin{equation*}
	\vec{u}(t) = S_{V}(t)(f,g) + \int_{0}^{t} S_{V}(t-s)(0, F(s)) \, \ud s
\end{equation*}
Since the propagator $S_{V}$ can be related to the half-wave propagators $e^{\pm i t D_{V}}$ via the Euler identity $e^{\pm i t D_{V}} = \cos t D_{V} \pm i \sin t D_{V},$ Strichartz estimates are a corollary of \eqref{eq:Str4halfVWave:1} and \eqref{eq:Str4halfVWave:2}. \qedhere
\end{proof}
%
%

In the rest of this paper, we will suppress the dependence of constants on $A_{1}, A_{2}, A_{3}(p,q), \alp_{1}, \alp_{2}$ and $\mu_{V}$.

\subsection{Linear approximation for traveling and concentrating profiles} \label{subsec:linApprox}
Here we develop linear approximation theory for solutions to \eqref{eq:Vwave} with initial data that either travels out to infinity or concentrates to a point. More precisely, we show that such solutions can be well-approximated by a single solution (suitably translated and/or scaled) to the potential-free or Euclidean equation, respectively, which are the underlying symmetry-invariant equations.

\subsubsection{Approximation of traveling profiles by free hyperbolic waves}
Given a sequence $\set{h_{n}}$ of elements in $\bbG$, we say that $h_{n}$ \emph{escapes to infinity}, and write $\abs{h_{n}} \to \infty$, if for every compact subset $K \subset \bbG$, $h_{n} \not \in K$ for all sufficiently large $n$. Equivalently, there exists a Cartan decomposition~\eqref{eq:cartanDecomp} of $h_{n}$ of the form
\begin{equation*}
	h_{n} = k_{n} \circ a_{r_{n}} \circ \widetilde{k_{n}}
\end{equation*}
with $r_{n} \to +\infty$. From the last statement, we see that $\abs{h_{n}} \to \infty$ is furthermore equivalent to 
\begin{equation*}
	\bfd_{\bbH^{d}}(h_{n} \cdot \zero, \zero) \to +\infty.
\end{equation*}
Motivated by these considerations, henceforth we will use the notation $\abs{h} := \bfd_{\bbH^{d}}(h \cdot \zero, \zero)$ for $h \in \GG$; note that this is consistent with the shorthand $\abs{h_{n}} \to \infty$ introduced above.

The following lemma is our main linear approximation result for traveling linear profiles, i.e., linear waves with initial data $\tau_{h_{n}^{-1}} (f, g)$ where $h_{n}$ escapes to infinity.

\begin{lem} \label{lem:pert4traveling}
Let $(f, g) \in \HH(\bbH^{d})$ be an initial data set and $\set{h_{n}}$ be a sequence in $\bbG$ such that $\abs{h_{n}} \to \infty$.
Then for every hyperbolic admissible triple $(p, q, \gmm)$, we have
\begin{equation*}
	\nrm{\tau_{h_{n}} S_{V}(t) \tau_{h_{n}^{-1}} (f,g) 
	-  S_{\hyp}(t) (f,g)}_{L^{p}_{t} (\bbR ; W^{1-\gmm, q} \times W^{-\gmm, q}(\bbH^{d}))}
	\to 0
	\mas n \to \infty.
\end{equation*}
\end{lem}
\begin{rem} 
Given any $t \in \bbR$ and $h \in \bbG$, observe that we have the identity
\begin{equation*}
	S_{V}(t) \tau_{h^{-1}} = \tau_{h^{-1}} S_{\tau_{h} V}(t).
\end{equation*}
Thus, the preceding lemma aims to make precise the following statement: Fix an initial data set $(f,g)$ and consider a sequence of wave equations with a potential $\tau_{h_{n}} V$ that escapes to infinity. Then the corresponding solutions approach the solution to the free wave equation with the same initial data as $n \to \infty$.
\end{rem}

\begin{proof} 
By an approximation argument, we may assume that $(f, g) \in C^{\infty}_{0} \times C^{\infty}_{0}(\bbH^{d})$. 
Using the translation invariance of the $W^{1-\gmm, q} \times W^{-\gmm, q}(\bbH^{d})$ norm and the time reversibility of the equation~\eqref{eq:Vwave}, the lemma would follow once we show that
\begin{equation} \label{eq:pert4traveling:pf:1}
	\nrm{S_{V}(t) \tau_{h_{n}^{-1}} (f,g) -  \tau_{h_{n}^{-1}} S_{\hyp}(t) (f,g)}_{L^{p}_{t} ([0, \infty) ; W^{1-\gmm, q} \times W^{-\gmm, q}(\bbH^{d}))}
	\to 0,
\end{equation}
as $n \to \infty$.

In what follows, we will write $\vec{u}_{\hyp}(t) = (u_{\hyp}, \rd_{t} u_{\hyp})(t):= S_{\hyp} (t) (f,g)$. Note that
\begin{equation*}
	(\rd_{t}^{2} - \De_{\bbH^{d}} + V) \tau_{h_{n}^{-1}} u_{\hyp}  = V \, \tau_{h_{n}^{-1}} u_{\hyp}.
\end{equation*}
Observe also that $\tau_{h_{n}^{-1}} \vec{u}_{\hyp}(0) = \tau_{h_{n}^{-1}}(f,g)$. Therefore, by Duhamel's principle, we obtain the representation
\begin{equation*}
	\tau_{h_{n}^{-1}} \vec{u}_{\hyp}(t) 
	= S_{V}(t) \tau_{h_{n}^{-1}} (f,g) 
		+ \int_{0}^{t} S_{V}(t-s)(0, V \tau_{h_{n}^{-1}} u_{\hyp}(s)) \, \ud s.
\end{equation*}
Therefore, the left-hand side of~\eqref{eq:pert4traveling:pf:1} can be estimated using Minkowski's inequality as follows:
\begin{align*}
& \hskip-2em
\nrm{S_{V}(t) \tau_{h_{n}^{-1}} (f,g) -  \tau_{h_{n}^{-1}} S_{\hyp}(t) (f,g)}_{L^{p}_{t} ([0, \infty) ; W^{1-\gmm, q} \times W^{-\gmm, q}(\bbH^{d}))} \\
=& \nrm{\int_{0}^{t} S_{V}(t-s)(0, V \tau_{h_{n}^{-1}} u_{\hyp}(s)) \, \ud s}_{L^{p}_{t} ([0, \infty) ; W^{1-\gmm, q} \times W^{-\gmm, q})(\bbH^{d})} \\
\leq & \int_{0}^{\infty} \nrm{S_{V}(t-s)(0, V \tau_{h_{n}^{-1}} u_{\hyp}(s)) }_{L^{p}_{t} ([0, \infty) ; W^{1-\gmm, q} \times W^{-\gmm, q}(\bbH^{d}))} \, \ud s
\end{align*}
Writing out $S_{V}(t-s)$ using the Euler identity $e^{\pm i (t-s) D_{V}} = \cos (t-s) D_{V} \pm i \sin(t-s) D_{V}$, and applying the Strichartz estimates~\eqref{eq:Str4halfVWave:1},~\eqref{eq:Str4halfVWave:2} for the half-wave propagators $e^{\pm i t D_{V}}$, the preceding line can be bounded by
\begin{equation*}
	\aleq_{p, q,\gmm} \int_{0}^{\infty} \nrm{e^{\pm i s D_{V}} V \tau_{h_{n}^{-1}} u_{\hyp}(s)}_{L^{2}(\bbH^{d})} \, \ud s.
\end{equation*}
Since $e^{\pm i s D_{V}}$ is an isometry on $L^{2}(\bbH^{d})$, we see that proving~\eqref{eq:pert4traveling:pf:1} has been reduced to establishing
\begin{equation} \label{eq:pert4traveling:pf:2}
	\int_{0}^{\infty} \nrm{V \tau_{h_{n}^{-1}} u_{\hyp}(s)}_{L^{2}(\bbH^{d})} \, \ud s \to 0 \quad \hbox{ as } n \to \infty.
\end{equation}
Let $\eps > 0$; we will show that the left-hand side of~\eqref{eq:pert4traveling:pf:2} is $< \eps$ for all sufficiently large $n$. We begin by splitting the $s$-integral into $\int_{0}^{T} + \int_{T}^{\infty}$, where $T > 2$ is to be determined. The second integral is estimated using the long time dispersive estimate as follows. For $r \in (2, \infty)$, define $r^{\ast}$ by $\frac{1}{r} + \frac{1}{r^{\ast}} = \frac{1}{2}$.  Thanks to the exponential decay of $V$~\eqref{eq:Vdecay}, observe that $\nrm{V}_{L^{r^{\ast}(\bbH^{d})}_{x}} < \infty$ for some $r^{\ast} < \infty$. Therefore,
\begin{align*}
\int_{T}^{\infty} \nrm{V \tau_{h_{n}^{-1}} u_{\hyp}(s)}_{L^{2}(\bbH^{d})} \, \ud s 
& \aleq \int_{T}^{\infty} \nrm{V}_{L^{r^{\ast}}(\bbH^{d})} \nrm{\tau_{h_{n}^{-1}} u_{\hyp}(s)}_{L^{r}(\bbH^{d})} \, \ud s \\
 &\aleq_{f,g, r} \nrm{V}_{L^{r^{\ast}}_{x}} \int_{T}^{\infty} s^{-\frac{3}{2}} \, \ud s.
\end{align*}
Hence taking $T$ sufficiently large, we have
\begin{equation} \label{eq:pert4traveling:pf:3}
	\int_{T}^{\infty} \nrm{V \tau_{h_{n}^{-1}} u_{\hyp}(s)}_{L^{2}(\bbH^{d})} \, \ud s  < \frac{\eps}{2}.
\end{equation}
We remark that this inequality holds for all $n$.

Now it remains to treat the $s$-integral on $[0,T]$. For each fixed $s \in [0, \infty)$, we have
\begin{equation*}
	\nrm{V \tau_{h_{n}^{-1}} u_{\hyp}(s)}_{L^{2}} \to 0
	\quad \hbox{ as } n \to \infty,
\end{equation*}
since $u_{\hyp}(s) \in L^{2}(\bbH^{d})$, $V$ decays exponentially according to the assumption~\eqref{eq:Vdecay} and $\abs{h_{n}} \to \infty$. 
By the dominated convergence theorem, it follows that
\begin{equation*}
	\int_{0}^{T} \nrm{V \tau_{h_{n}^{-1}} u_{\hyp}(s)}_{L^{2}(\bbH^{d})} \, \ud s < \frac{\eps}{2}
\end{equation*}
for all sufficiently large $n$, as desired. \qedhere
\end{proof}


We also record a simple corollary of Lemma \ref{lem:pert4traveling}, which will be useful later in the definition of nonlinear profiles. We omit the obvious proof.
\begin{cor} \label{cor: h h:lin:smallScat}
Let $\set{h_{n}}$ be a sequence in $\bbG$ such that $\abs{h_{n}} \to \infty$. Let $(f, g) \in \HH (\bbH^{d})$ be an initial data set on $\bbH^{d}$, and $(p, q, \gmm)$ a hyperbolic admissible triple. 
Suppose furthermore that
\begin{equation*}
	\nrm{S_{\hyp}(t) (f,g)}_{L^{p}_{t} ([0, \infty); W^{1-\gmm, q} \times W^{-\gmm, q}(\bbH^{d}) )} < \eps
\end{equation*}
for some $\eps > 0$. Then we have
\begin{equation*}
	\limsup_{n \to \infty} \nrm{S_{V}(t) \tau_{h_{n}^{-1}} (f, g)}_{L^{p}_{t} ([0, \infty) ; W^{1-\gmm, q} \times W^{-\gmm, q}(\bbH^{d}))} < 2 \eps.
\end{equation*}
An analogous statement holds in the negative time direction as well.
\end{cor}

\subsubsection{Approximation of concentrating profiles}
Here we consider the case of initial data that concentrate to smaller and smaller scales. As discussed earlier, the relevant equation in this case is the underlying scale-invariant Euclidean equation $\Box_{\bbR^{d}} v = 0$. To make this notion precise, we need a means to pass back and forth functions on $\bbH^{d}$ and $\bbR^{d}$. We will achieve this by using the map 
\begin{align*}
\Psi : (0, \infty) \times \bbS^{d-1} \to (\bbH^{d} \setminus \set{\zero}) \subset \bbR^{d+1}, \quad (r, \omg) \mapsto (\sinh r \omg, \cosh r),
\end{align*}
where we identify the domain with $\bbR^{d} \setminus \set{0}$, to pull back functions on $\bbH^{d}$ to $\bbR^{d}$ and vice versa. In practical terms, we identify functions on $\bbH^{d}$ and $\bbR^{d}$ by using polar coordinates on both spaces\footnote{There is the well-known issue of singularity at $r= 0$, but it will only be a minor inconvenience.}.

Let $\chi \in C^{\infty}_0(\R^d)$ be a radial function, with $\chi(r) = 1$ for $r \le 1$ and $\supp (\chi)\, \subseteq \,\{ r \le 2\}$, and set $\chi_{R}(r):= \chi(r/ R)$ for $R > 0$. For $M > 0$, we define the mapping 
\EQ{
&\QQ_{M}: \HH_{\euc}(\R^d) \to  \HH_{\euc}(\R^d) \cap C^{\infty}_0 \times C^{\infty}_0(\R^d)\\
&\QQ_{M} (f, g) :=  (\QQ_{M} f, \QQ_{M} g) :=(\chi_{\sqrt{M}} e^{M^{-1} \De } f, \, \, \chi_{\sqrt{M}} e^{M^{-1}\De } g)
}
where $e^{M^{-1} \De}$ is defined by the Euclidean Fourier multiplier, 
\EQ{
\ha{ e^{M^{-1} \De} f}( \xi) = e^{- \abs{\xi}^2/M} \ha{f}(\xi)
}
Thus, $\QQ_{M}$ regularizes and truncates the data $(f, g) \in \HH_{\euc}$. Note that we have 
\EQ{
\| \QQ_{M}(f, g) - (f, g)\|_{\HH_{\euc}(\R^d)} \to 0 \mas M \to \infty
}
Let $\lmb_{n} \in [1, \infty)$ be any sequence. We define the sequence of maps $\T_{\lmb_{n}}(f, g) : \HH_{\euc}(\bbR^{d}) \to \HH(\bbH^{d})$ by rescaling $\QQ_{\lmb_{n}}$ by $\lmb_{n}$ and then pulling back by $\Psi^{-1}$. We will simply write
\begin{align*}
	\T_{\lmb_{n}}(f, g) (r, \omg) =&  (\T_{\lmb_{n}}^{0} f, \T_{\lmb_{n}}^{1} g) (r, \omg) \\
	:=& \big(\lmb_{n}^{\frac{d-2}{2}} (\QQ_{\lmb_{n}} f)( \lmb_{n} r, \omg), 
		\lmb_{n}^{\frac{d}{2}} (\QQ_{\lmb_{n}} g)(\lmb_{n} r, \omg) \big).
\end{align*}

For any $p \in [1, \infty]$ and $\gmm \geq 0$, we have
\begin{equation} \label{eq:Tlmbn}
\begin{aligned}
	& \nrm{\T_{\lmb_{n}}^{0} f}_{W^{\gmm, p}(\bbH^{d})} 
	\aleq \lmb_{n}^{\frac{d-2}{2}} \nrm{\QQ_{\lmb_{n}} f(\lmb_{n} \cdot)}_{W^{\gmm, p}(\bbR^{d})}
	\aleq \lmb_{n}^{\frac{d-2}{2} - \frac{d}{p} + \gmm} \nrm{f}_{W^{\gmm, p}(\bbR^{d})}, \\
	& \nrm{\T_{\lmb_{n}}^{1} g}_{W^{\gmm, p}(\bbH^{d})} 
	\aleq \lmb_{n}^{\frac{d}{2}} \nrm{\QQ_{\lmb_{n}} g(\lmb_{n} \cdot)}_{W^{\gmm, p}(\bbR^{d})}
	\aleq \lmb_{n}^{\frac{d}{2} - \frac{d}{p} + \gmm} \nrm{g}_{W^{\gmm, p}(\bbR^{d})}.
\end{aligned}
\end{equation}
Moreover, for any sequence $\set{\lmb_{n}} \subseteq [1, \infty)$ such that $\lmb_{n} \to \infty$ as $n \to \infty$, we have
\begin{equation} \label{eq:Tlmbn:HH}
	\nrm{\T_{\lmb_{n}} (f, g)}_{\HH} = \nrm{(f, g)}_{\HH_{\euc}} + o_{n}(1),
\end{equation}
under the normalization 
\begin{align*}
	\brk{\vec{u} \mid \vec{v} \, }_{\HH} =& \brk{u_{0} \mid (-\lap_{\bbH^{d}}) v_{0}}_{L^{2}(\bbH^{d})} + \brk{u_{1} \mid v_{1}}_{L^{2}(\bbH^{d})} \\
	\brk{\vec{u} \mid \vec{v} \, }_{\HH_{\euc}} =& \brk{u_{0} \mid (-\lap_{\bbR^{d}}) v_{0}}_{L^{2}(\bbR^{d})} + \brk{u_{1} \mid v_{1}}_{L^{2}(\bbR^{d})} 
\end{align*}
for smooth and compactly supported $\vec{u}, \vec{v}$. We leave the routine verification of the above statements to the reader.
In the sequel, we will use the convention that if $\T_{\lmb_{n}}$ is applied to a single function $f$, then $\T_{\lmb_{n}} f := \T^{0}_{\lmb_{n}} f$.

The first approximation lemma we prove says that for a concentrating profile, the evolution $S_{V}(t)$ can be replaced by the potential-free evolution $S_{\hyp}(t)$, asymptotically as $n \to \infty$.
\begin{lem} [Approximation by potential-free evolution] \label{lem:pert4conc}
Let $(f, g) \in \HH_{\euc}(\bbR^{d})$ be an initial data set and $\set{\lmb_{n}} \subseteq [1, \infty)$, $\set{h_{n}} \subseteq \GG$ be sequences such that $\lmb_{n} \to \infty$ as $n \to \infty$. Then for every hyperbolic admissible triple $(p,q, \gmm)$, we have
\begin{equation*}
	\nrm{S_{V}(t) \tau_{h_{n}^{-1}} \T_{\lmb_{n}}(f, g) 
		- S_{\hyp}(t) \tau_{h_{n}^{-1}} \T_{\lmb_{n}}(f, g)}_{L^{p}_{t} (\bbR; W^{1-\gmm, q} \times W^{-\gmm, q}(\bbH^{d}))} \to 0
\end{equation*}
as $n \to \infty$.
\end{lem}

\begin{proof} 
As before, by an approximation argument, we may assume that $(f, g) \in C^{\infty}_{0} \times C^{\infty}_{0}(\bbR^{d})$.
Define $\vec{u}_{\hyp, n}(t) = (u_{\hyp, n}, \rd_{t} u_{\hyp, n})(t) := S_{\hyp}(t) \tau_{h_{n}^{-1}} \T_{\lmb_{n}}(f, g)$. Then applying Duhamel's principle to $(\rd_{t}^{2} - \Dlt_{\bbH^{d}} + V) u_{\hyp, n} = V u_{\hyp, n}$, we see that
\begin{equation*}
	\vec{u}_{\hyp, n} = S_{V}(t) \tau_{h_{n}^{-1}} \T_{\lmb_{n}}(f, g)
				 + \int_{0}^{t} S_{V}(t-s) (0, V u_{\hyp, n}(s)) \, \ud s 
\end{equation*}
Therefore, as in the proof of Lemma \ref{lem:pert4traveling}, it suffices to prove
\begin{equation}
	\int_{0}^{\infty} \nrm{V u_{\hyp, n}(s)}_{L^{2}(\bbH^{d})} \, \ud s \to 0 \mas n \to \infty.
\end{equation}

As before, given $r \in (2, \infty)$ we define $r^{\ast}$ by $\frac{1}{r} + \frac{1}{r^{\ast}}  = \frac{1}{2}$. Thanks to the assumption~\eqref{eq:Vdecay} on $V$, $V \in L^{r^{\ast}}(\bbH^{d})$ for a sufficiently large $r^{\ast}$, or equivalently, an exponent $r$ sufficiently close to $2$. Taking $r$ closer to $2$ if necessary, we can also ensure that $(d-1)(\frac{1}{2} - \frac{1}{r}) < 1$. By H\"older's inequality and Proposition \ref{prop: disp} (the dispersive estimate for free waves), we have
\begin{align*}
\int_{0}^{\infty} \nrm{V u_{\hyp, n}(s)}_{L^{2}(\bbH^{d})} \, \ud s
& \aleq \int_{0}^{\infty} \nrm{V}_{L^{r^{\ast}}(\bbH^{d})} \nrm{u_{\hyp, n}(s)}_{L^{r}(\bbH^{d})} \, \ud s \\
& \aleq  \nrm{V}_{L^{r^{\ast}}(\bbH^{d})} 
	\bb( \int_{0}^{2} \abs{s}^{-(d-1)(\frac{1}{2} - \frac{1}{r})} \, \ud s
		+ \int_{2}^{\infty} \abs{s}^{-\frac{3}{2}} \, \ud s \bb) \\
&\phantom{\aleq} 
	\times \nrm{\vec{u}_{\hyp, n}(0)}_{W^{\sgm, r'} \times W^{\sgm-1, r'} (\bbH^{d})} \\
& \aleq_{d, r}  \nrm{V}_{L^{r^{\ast}}(\bbH^{d})} \nrm{\T_{\lmb_{n}} (f, g)}_{W^{\sgm, r'} \times W^{\sgm-1, r'} (\bbH^{d})},
\end{align*}
where $\sgm = (d+1)(\frac{1}{2} - \frac{1}{r})$ and on the last line, we used the fact that $(d-1)(\frac{1}{2}-\frac{1}{r}) < 1$ to carry out the integral from $0$ to $2$. Now observe that
\begin{align*}
	\nrm{\T_{\lmb_{n}}^{0} f}_{W^{\sgm, r'}(\bbH^{d})} \aleq_{f} \lmb_{n}^{\frac{d-2}{2} - d(1 - \frac{1}{r}) + (d+1)(\frac{1}{2} - \frac{1}{r}) }
\end{align*}
and the right-hand side equals $\lmb_{n}^{- \frac{1}{2} - \frac{1}{r}}$, which goes to $0$ as $n \to \infty$. To handle $\T_{\lmb_{n}}^{1} g$, note first that by taking $r$ closer to $2$ if necessary, we can guarantee that $\sigma-1<0$. Then by the dual Sobolev inequality, we have $L^{b} \hookrightarrow W^{\sgm-1, r'}$, where $b \in (1, \infty)$ is given by $\frac{d}{b} = \frac{d}{2} + \frac{1}{2} + \frac{1}{r}$. Therefore,
\begin{align*}
	\nrm{\T_{\lmb_{n}}^{1} g}_{W^{\sgm-1, r'}(\bbH^{d})} 
	\aleq \nrm{\T_{\lmb_{n}}^{1} g}_{L^{b}(\bbH^{d})} 
	\aleq_{g} \lmb_{n}^{-\frac{1}{2} - \frac{1}{r}} 
	\to 0 \mas n \to \infty,
\end{align*}
which establishes the claim. \qedhere
\end{proof}

Our next approximation lemma makes precise the idea that the linear hyperbolic evolutions $S_{V}(t) \tau_{h_{n}^{-1}} \T_{\lmb_{n}}(f,g)$ are well-approximated by a concentrating sequence of Euclidean evolutions $\tau_{h_{n}^{-1}} \T_{\lmb_{n}} S_{\euc} (\lmb_{n} t) (f, g)$ as $n \to \infty$. Unlike Lemma \ref{lem:pert4conc}, however, the approximation is valid only on short intervals of the form $(-T_{0} / \lmb_{n}, T_{0} / \lmb_{n})$.
\begin{lem}[Approximation by Euclidean evolution] \label{lem: e h:lin}
Let $(f, g) \in \HH_{\euc}(\bbR^{d})$ be an initial data set and $\set{\lmb_{n}} \subseteq [1, \infty)$, $\set{h_{n}} \subseteq \GG$ be sequences such that $\lmb_{n} \to \infty$ as $n \to \infty$. Fix an interval $(-T_{0}, T_{0})$ and define $I_{n} := (-T_{0} / \lmb_{n}, T_{0} / \lmb_{n})$ for every $n$. 
Then we have
\begin{equation} \label{eq: e h:lin:H}
	\nrm{S_{V}(t) \tau_{h_{n}^{-1}} \T_{\lmb_{n}} (f, g) - \tau_{h_{n}^{-1}} \T_{\lmb_{n}} S_{\euc}(\lmb_{n} t) (f, g)}_{L^{\infty}_{t}(I_{n}; \HH)} \to 0 \mas n \to \infty.
\end{equation}
Moreover, for every hyperbolic admissible triple $(p, q, \gmm)$ with $\gmm = 1$ that is also Euclidean admissible, i.e.,
\begin{equation*}
	\frac{2}{p} + \frac{d-1}{q} \leq \frac{d-1}{2},
\end{equation*}
we have
\begin{equation} \label{eq: e h:lin:S}
	\nrm{S_{V, 0}(t) \tau_{h_{n}^{-1}} \T_{\lmb_{n}} (f, g) - \tau_{h_{n}^{-1}} \T_{\lmb_{n}}^{0} S_{\euc, 0}(\lmb_{n} t) (f, g)}_{L^{p}_{t}(I_{n}; L^{q} (\bbH^{d}))} \to 0
\end{equation}
as $n \to \infty$.
\end{lem}
\begin{proof} 
In this proof we abuse notation by using $S_V(t)\vec u$ to mean both $S_{V,0}(t)\vec u$ and $(S_{V,0}(t)\vec u,S_{V,1}(t)\vec u),$ for $\vec u\in\HH$,  depending on the context. We claim that it suffices to prove~\eqref{eq: e h:lin:H} and~\eqref{eq: e h:lin:S} for $(f, g) \in C^{\infty}_{0} \times C^{\infty}_{0}(\bbR^{d})$. To verify this claim, it is enough, by the density of $C^{\infty}_{0} \times C^{\infty}_{0}(\bbR^{d})$ in $\HH_{\euc}$, to show that the left-hand side of~\eqref{eq: e h:lin:S} can be made arbitrarily small uniformly in $n$ by taking $(f, g)$ to be small enough in $\HH_{\euc}$. Indeed, by Strichartz estimates for $S_{V}$ and~\eqref{eq:Tlmbn}, we have
\begin{align*}
	\nrm{S_{V}(t) \tau_{h_{n}^{-1}} \T_{\lmb_{n}} (f,g)}_{L^{\infty}_{t}(I_{n}; \HH) \cap L^{p}_{t}(I_{n}; L^{q}(\bbH^{d}))}
	\aleq & \nrm{\T_{\lmb_{n}}(f,g)}_{\HH} \aleq \nrm{(f,g)}_{\HH_{\euc}}.
\end{align*}
Moreover, by~\eqref{eq:Tlmbn} and Strichartz estimates for $S_{\euc}$, we have
\begin{align*}
&\hskip-2em
	\nrm{\tau_{h_{n}^{-1}} \T_{\lmb_{n}} S_{\euc} (t)(f,g)}_{L^{\infty}_{t}(I_{n}; \HH) \cap L^{p}_{t}(I_{n}; L^{q}(\bbH^{d}))} \\
	\aleq & \nrm{S_{\euc}(t)(f,g)}_{L^{\infty}_{t}((-T_{0}, T_{0}); \HH) \cap L^{p}_{t}((-T_{0},T_{0}); L^{q} (\bbR^{d}))} 
	\aleq \nrm{(f,g)}_{\HH_{\euc}}.
\end{align*}
Hence, by the triangle inequality, the desired conclusion follows. Henceforth, we assume that $(f, g) \in C^{\infty}_{0} \times C^{\infty}_{0}(\bbR^{d})$.

Since 
\begin{equation*}
S_{V}(0) \tau_{h_{n}^{-1}} \T_{\lmb_{n}} (f, g) = \tau_{h_{n}^{-1}} \T_{\lmb_{n}} S_{\euc}(0) (f, g) = \tau_{h_{n}^{-1}} \T_{\lmb_{n}} (f, g),
\end{equation*}
the estimates~\eqref{eq: e h:lin:H} and~\eqref{eq: e h:lin:S} would follow, by Duhamel's principle and the Strichartz estimate for $S_{V}$, once we prove
\begin{equation}
	\nrm{\Box_{V} (\tau_{h_{n}^{-1}} \T_{\lmb_{n}} S_{\euc}(\lmb_{n} t) (f, g))}_{L^{1}_{t}(I_{n}; L^{2}(\bbH^{d}))} \to 0 \mas n \to \infty.
\end{equation}
Using the shorthand $v(\cdot) := S_{\euc}(\cdot) (f, g)$, we decompose
\begin{equation} \label{eq: e h:lin:BoxV}
\Box_{V} (\tau_{h_{n}^{-1}} \T_{\lmb_{n}} v(\lmb_{n} t))
= \tau_{h_{n}^{-1}} \Box_{\bbH^{d}} (\T_{\lmb_{n}} v (\lmb_{n} t)) + V \tau_{h_{n}^{-1}} \T_{\lmb_{n}} v(\lmb_{n} t).
\end{equation}
Since $V$ is bounded, the contribution of the second term can be treated easily using Poincar\'e's inequality as follows:
\begin{align*}
	& \hskip-2em
	\nrm{V \tau_{h_{n}^{-1}} \T_{\lmb_{n}} v(\lmb_{n} t)}_{L^{1}_{t} (I_{n}; L^{2}(\bbH^{d}))} \\
	\aleq & (T_{0} / \lmb_{n}) \nrm{V}_{L^{\infty}(\bbH^{d})} \nrm{\T_{\lmb_{n}} v(\lmb_{n} t)}_{L^{\infty}_{t} (I_{n}; L^{2}(\bbH^{d}))} \\
	\aleq &  (T_{0} / \lmb_{n}) \nrm{V}_{L^{\infty}(\bbH^{d})} \nrm{v(t)}_{L^{\infty}_{t} (I; \dot{H}^{1}(\bbR^{d}))} \to 0 \mas n \to \infty.
\end{align*}
It remains to handle the contribution of the first term in~\eqref{eq: e h:lin:BoxV}. By the translation invariance of the $L^{2}(\bbH^{d})$ norm, we may remove $\tau_{h_{n}^{-1}}$. Then using the fact that $\Box_{\bbR^{d}} v = 0$, it suffices to prove
\begin{equation*}
	\nrm{\Box_{\bbH^{d}} (\T_{\lmb_{n}} v (\lmb_{n} t)) - \lmb_{n}^{2} \T_{\lmb_{n}} (\Box_{\bbR^{d}} v)(\lmb_{n} t)}_{L^{1}_{t}(I_{n}; L^{2}(\bbH^{d}))} \to 0 \mas n \to \infty.
\end{equation*}
This estimate is an immediate consequence of the fact that the $\dot{H}^{1}(\bbR^{d})$ and $\dot{H}^{2}(\bbR^{d})$ norm of $v$ is uniformly bounded, which is obvious since $v(t) = S_{\euc}(t) (f, g)$ with $(f, g) \in C^{\infty}_{0} \times C^{\infty}_{0}(\bbR^{d})$, and Claim \ref{claim: e h:comm} below. \qedhere
\end{proof}
\begin{claim}[Key commutator estimate] \label{claim: e h:comm}
Let $\set{\lmb_{n}} \subseteq [1, \infty)$ be a sequence such that $\lmb_{n} \to \infty$. Let $v(t) = v(t, x)$ be a smooth function on $\bbR^{1+d}$ that satisfies
\begin{equation} \label{eq: e h:comm:hyp}
	\nrm{v}_{L^{\infty}_{t}(\bbR; \dot{H}^{1}(\bbR^{d}))}
	+ \nrm{v}_{L^{\infty}_{t}(\bbR; \dot{H}^{2}(\bbR^{d}))} < \infty.
\end{equation}
Given a finite interval $(- T_{0}, T_{0})$, define $I_{n} := (-T_{0} / \lmb_{n}, T_{0} / \lmb_{n})$. Then we have
\begin{equation} \label{eq: e h:comm}
	\nrm{\Box_{\bbH^{d}} (\T_{\lmb_{n}} v(\lmb_{n} t))  - \lmb_{n}^{2} \T_{\lmb_{n}} (\Box_{\bbR^{d}} v)(\lmb_{n} t)}_{L^{1}_{t}(I_{n}; L^{2}(\bbH^{d}))} \to 0
\end{equation}
as $n \to \infty$.
\end{claim}
\begin{proof} 
Recall that in the definition of $\T_{\lmb_{n}}$, we passed from functions on $\bbR^{d}$ to functions on $\bbH^{d}$ by using the geodesic polar coordinates on both spaces. 
In what follows, we will view both $\T_{\lmb_{n}} \vec{v}(\cdot)$ and $\vec{v}(\cdot)$ as functions of $(r, \omg)$ using these coordinates.
Comparing the volume form of $\bbH^{d}$ and $\bbR^{d}$ in polar coordinates, we have the identity
\begin{equation} \label{e h:vol}
\nrm{\cdot}_{L^{2}(\bbH^{d})}=\nrm{(\cdot) \bb( \frac{\sinh r}{r} \bb)^{\frac{d-1}{2}}}_{L^{2}(\bbR^{d})}. 
\end{equation}
Moreover, the d'Alembertians on $\bbR \times \bbH^{d}$ and $\bbR^{1+d}$ take the form
\begin{align*}
	\Box_{\bbH^{d}} =& \rd_{t}^{2} - \rd_{r}^{2} - (d-1) \coth r \rd_{r} - \frac{1}{\sinh^{2} r} \lap_{\bbS^{d-1}}, \\
	\Box_{\bbR^{d}} =& \rd_{t}^{2} - \rd_{r}^{2} - (d-1) \frac{1}{r} \rd_{r} - \frac{1}{r^{2}} \lap_{\bbS^{d-1}},
\end{align*}
respectively. 
Following the conventions just described and using the shorthand $\vec{v}_{n} (t) := \T_{\lmb_{n}} \vec{v}(\lmb_{n} t)$, we split the commutator on the left-hand side of~\eqref{eq: e h:comm} as follows:
\begin{align*}
	\Box_{\bbH^{d}} (\T_{\lmb_{n}} v(\lmb_{n} t))  - \lmb_{n}^{2} \T_{\lmb_{n}} (\Box_{\bbR^{d}} v)(\lmb_{n} t)
	= & (\Box_{\bbH^{d}} - \Box_{\bbR^{d}}) v_{n}(t) \\
	& + \Box_{\bbR^{d}} v_{n}(t) - \lmb_{n}^{2} \T_{\lmb_{n}} (\Box_{\bbR^{d}} v)(\lmb_{n} t)
\end{align*}

We first handle the contribution of $(\Box_{\bbH^{d}} - \Box_{\bbR^{d}}) (\T_{\lmb_{n}} v(\lmb_{n} t))$. We have
\begin{align*}
	(\Box_{\Hp^{d}} - \Box_{\R^{d}}) v_{n} 
	=&	- (d-1) \bb( \coth r - \frac{1}{r} \bb) \rd_{r} v_{n} 
		- \bb( \frac{1}{\sinh^{2} r} - \frac{1}{r^{2}} \bb) \lap_{\bbS^{d-1}} v_{n}  \\
	=&	- (d-1) \bb(\coth r - \frac{1}{r} \bb) 
			\lmb_{n}^{\frac{d-1}{2}} \chi'(\lmb_{n}^{\frac{1}{2}} r) e^{\lmb_{n}^{-1} \lap} v(\lmb_{n} t, \lmb_{n} r, \omg) \\
	&	- (d-1) \bb(\coth r - \frac{1}{r} \bb) \lmb_{n}^{\frac{d}{2}} \chi(\lmb_{n}^{\frac{1}{2}} r)  (\rd_{r} e^{\lmb_{n}^{-1}\lap} v) (\lmb_{n} t, \lmb_{n} r, \omg) \\
	&	- \bb(\frac{1}{\sinh^{2} r} - \frac{1}{r^{2}} \bb) \lmb_{n}^{\frac{d-2}{2}} \chi(\lmb_{n}^{\frac{1}{2}} r) \lap_{\bbS^{d-1}} e^{\lmb_{n}^{-1} \lap} v(\lmb_{n} t, \lmb_{n} r, \omg)  
\end{align*}
Note that $\chi(\lmb_{n}^{\frac{1}{2}} r)$ is supported on $\set{r \aleq \lmb_{n}^{-\frac{1}{2}}}$. As $\lmb_{n} \geq 1$, there exists an absolute constant $c > 0$ such that
\begin{equation*}
	\frac{\sinh r}{r} \leq c, \quad \abs{\coth r - \frac{1}{r}} \leq c r, \quad \abs{\frac{1}{\sinh^{2} r} - \frac{1}{r^{2}}} \leq c \quad \hbox{on} \,\, \supp \chi(\lmb_{n}^{\frac{1}{2}} \cdot)
\end{equation*}
Putting the above ingredients together and applying H\"older's inequality, we have
\begin{align*}
& \hskip-2em
	\nrm{(\Box_{\bbH^{d}} - \Box_{\bbR^{d}}) v_{n}}_{L^{1}_{t}(I_{n}; L^{2}(\bbH^{d}))} \\
\aleq & (T_{0} / \lmb_{n}) \lmb_{n}^{-\frac{1}{2}} \nrm{\lmb_{n} r^{2} \chi'(\lmb_{n}^{\frac{1}{2}} r)}_{L^{\infty}(\bbR^{d})} \nrm{\frac{\lmb_{n}^{\frac{d}{2}}}{\lmb_{n} r} e^{\lmb_{n}^{-1} \lap} v(\lmb_{n} t, \lmb_{n} r, \omg)}_{L^{\infty}_{t}(I_{n}; L^{2}(\bbR^{d}))} \\
& + (T_{0} / \lmb_{n}) \lmb_{n}^{-\frac{1}{2}} \nrm{\lmb_{n}^{\frac{1}{2}} r \chi(\lmb_{n}^{\frac{1}{2}} r)}_{L^{\infty}(\bbR^{d})} 
				\nrm{\lmb_{n}^{\frac{d}{2}} (\rd_{r} e^{\lmb_{n}^{-1} \lap} v) (\lmb_{n} r, \lmb_{n} r, \omg)}_{L^{\infty}_{t}(I_{n}; L^{2}(\bbR^{d}))} \\
& + (T_{0} / \lmb_{n}) \nrm{\lmb_{n} r^{2} \chi(\lmb_{n}^{\frac{1}{2}}r)}_{L^{\infty}(\bbR^{d})} \nrm{\frac{\lmb_{n}^{\frac{d}{2}}}{\lmb_{n}^{2} r^{2}} \lap_{\bbS^{d-1}} e^{\lmb_{n}^{-1} \lap} v(\lmb_{n} t, \lmb_{n} r, \omg)}_{L^{\infty}_{t} (I_{n}; L^{2}(\bbR^{d}))} \\
\aleq & T_{0} \lmb_{n}^{-\frac{3}{2}} \nrm{v}_{L^{\infty}_{t}(I; \dot{H}^{1}(\bbR^{3}))} + T_{0} \lmb_{n}^{-1} \nrm{v}_{L^{\infty}_{t}(I; \dot{H}^{2}(\bbR^{d}))}
\to 0 \mas n \to \infty,
\end{align*}
where we used the rescaling $(t, r, \omg) \mapsto (t/\lmb_{n}, r/\lmb_{n}, \omg)$ and Hardy's inequality for the last inequality. 

Next, we treat the term $\Box_{\bbR^{d}} v_{n}(t) - \lmb_{n}^{2} \T_{\lmb_{n}} (\Box_{\bbR^{d}} v)(\lmb_{n} t)$. We begin by computing
\begin{align*}
	\Box_{\bbR^{d}} v_{n}(t) - \lmb_{n}^{2} \T_{\lmb_{n}} (\Box_{\bbR^{d}} v)(\lmb_{n} t) 
	= & - 2 \lmb_{n}^{\frac{d+1}{2}} \chi'(\lmb_{n}^{\frac{1}{2}} r) (\rd_{r} e^{\lmb_{n}^{-1} \lap} v)(\lmb_{n} t, \lmb_{n} r, \omg) \\
	& - \lmb_{n}^{\frac{d}{2}} (\lap_{\bbR^{d}} \chi)(\lmb_{n}^{\frac{1}{2}} r) (e^{\lmb_{n}^{-1} \lap} v)(\lmb_{n} t, \lmb_{n} r, \omg).
\end{align*}
Then proceeding as before, we estimate
\begin{align*}
& \hskip-2em
	\nrm{\Box_{\bbR^{d}} v_{n}(t) - \lmb_{n}^{2} \T_{\lmb_{n}} (\Box_{\bbR^{d}} v)(\lmb_{n} t) }_{L^{1}_{t}(I_{n}; L^{2}(\bbH^{d}))} \\
\aleq & (T_{0} / \lmb_{n}) \lmb_{n}^{\frac{1}{2}} \nrm{\chi' (\lmb_{n}^{\frac{1}{2}} r)}_{L^{\infty}(\bbR^{d})} 
			\nrm{\lmb_{n}^{\frac{d}{2}} (\rd_{r} e^{\lmb_{n}^{-1} \lap} v)(\lmb_{n} t, \lmb_{n} r, \omg)}_{L^{\infty}_{t}(I_{n}; L^{2}(\bbR^{d}))} \\
	& +(T_{0} / \lmb_{n}) \lmb_{n}^{\frac{1}{2}} \nrm{\lmb_{n}^{\frac{1}{2}} r (\lap_{\bbR^{d}} \chi) (\lmb_{n}^{\frac{1}{2}} r)}_{L^{\infty}(\bbR^{d})} 
			\nrm{\frac{\lmb_{n}^{\frac{d}{2}}}{\lmb_{n} r} e^{\lmb_{n}^{-1} \lap} v(\lmb_{n} t, \lmb_{n} r, \omg)}_{L^{\infty}_{t}(I_{n}; L^{2}(\bbR^{d}))} \\
\aleq & T_{0} \lmb_{n}^{-\frac{1}{2}} \nrm{v}_{L^{\infty}_{t}(I; \dot{H}^{1}(\bbR^{d}))} \to 0 \mas n \to \infty,
\end{align*}
where we used the rescaling $(t, r, \omg) \mapsto (t/\lmb_{n}, r/\lmb_{n}, \omg)$ and Hardy's inequality for the last inequality. This completes the proof of the claim. \qedhere
\end{proof}

Our next goal is to prove an analogoue of Corollary \ref{cor: h h:lin:smallScat}, which will be used in the definition of nonlinear profiles. In the present case, we need an extra ingredient, whose proof depends on Lemma \ref{lem:pert4conc} and the dispersive estimate for the free wave equation on $\bbH^{d}$.
\begin{lem} \label{lem:uniformScat4conc}
Let $(f, g) \in \HH_{\euc}(\bbR^{d})$ be an initial data set and $\set{\lmb_{n}} \subseteq [1, \infty)$, $\set{h_{n}} \subseteq \GG$ be sequences such that $\lmb_{n} \to \infty$ as $n \to \infty$. Let $(p, q, \gmm)$ a hyperbolic admissible triple such that 
\begin{equation*} 
	\frac{2}{p} +  \frac{d-1}{q}  \leq \frac{d-1}{2}, \quad 2 < p < \infty.
\end{equation*}
Then we have
\begin{equation} \label{eq:uniformScat4conc}
	\lim_{T \to \infty} \bb( \sup_{n} \nrm{S_{V} (t)\tau_{h_{n}^{-1}} \T_{\lmb_{n}} (f, g)}_{L^{p}_{t} ((T / \lmb_{n}, \infty) ; W^{1-\gmm, q} \times W^{-\gmm, q}(\bbH^{d}))} \bb) = 0.
\end{equation}
An analogous statement holds in the negative time direction as well.
\end{lem}
\begin{proof} 
We begin by making a few reductions. 
By Lemma \ref{lem:pert4conc}, it will suffice to prove~\eqref{eq:uniformScat4conc} for $S_{V} = S_{\hyp}$. Then as $S_{\hyp}$ and the norm on the left-hand side of~\eqref{eq:uniformScat4conc} are translation invariant, we may remove $\tau_{h_{n}^{-1}}$. We may futhermore restrict ourselves to $t > 0$, as the argument for $t < 0$ will be the same by the time reversal symmetry, and also assume that $(f, g) \in C^{\infty}_{0} \times C^{\infty}_{0} (\bbR^{d})$ by approximation. Thus it suffices to show that for any $\eps > 0$, there exists a $T > 0$ such that
\begin{equation} \label{eq:uniformScat4conc:reduce}
	\limsup_{n \to \infty} \nrm{S_{\hyp}(t) \T_{\lmb_{n}}(f, g)}_{L^{p}_{t}((T/\lmb_{n}, \infty);W^{1-\gmm, q} \times W^{-\gmm, q}(\bbH^{d}))} < \eps,
\end{equation}
where we remark that the $\limsup_{n \to \infty}$ can be replaced with $\sup_{n}$ by choosing $T$ larger if ncecessary.

Let $T > 0$ be a positive number to be determined. Our main tool is the dispersive estimate (Proposition~\ref{prop: disp}) for the free hyperbolic evolution $S_{\hyp}$. We first handle the short time contribution. As we are taking $\limsup_{n \to \infty}$ and $\lmb_{n} \to \infty$, we have $T/ \lmb_{n} < 2$ for $n$ large enough. By the first part of Proposition~\ref{prop: disp} and~\eqref{eq:Tlmbn}, we estimate
\begin{equation} \label{eq:uniformScat4conc:pf:1}
\begin{aligned}
	& \hskip-2em
	\nrm{S_{\hyp}(t) \T_{\lmb_{n}}(f, g)}_{L^{p}_{t}((T/\lmb_{n}, 2]; W^{1-\gmm, q} \times W^{-\gmm, q}(\bbH^{d}))}  \\
	\aleq & \nrm{\abs{t}^{-(d-1)(\frac{1}{2} - \frac{1}{q})}}_{L^{p}_{t}((T/\lmb_{n}, 2])} \nrm{\T_{\lmb_{n}} (f, g)}_{W^{1-\gmm+\sgm, q'} \times W^{-\gmm+\sgm, q'}(\bbH^{d})} \\
	\aleq & (T/\lmb_{n})^{-(d-1)(\frac{1}{2} - \frac{1}{q})+\frac{1}{p}} \nrm{\T_{\lmb_{n}} (f, g)}_{W^{1-\gmm+\sgm, q'} \times W^{-\gmm+\sgm, q'}(\bbH^{d})} \\
	\aleq & T^{-(d-1)(\frac{1}{2} - \frac{1}{q}) + \frac{1}{p}} \nrm{(f, g)}_{W^{1-\gmm+\sgm, q'} \times W^{-\gmm+\sgm, q'}(\bbR^{d})},
\end{aligned}
\end{equation}
which can be made arbitrarily small by taking $T > 0$ sufficiently large. Note that $-(d-1)(\frac{1}{2} - \frac{1}{q}) + \frac{1}{p} \leq -\frac{1}{p} < 0$ by hypothesis. To see the cancellation of $\lmb_{n}$, we recall that $\sgm = (d+1) (\frac{1}{2} - \frac{1}{q})$ and $\gmm = \frac{d}{2} - \frac{1}{p} - \frac{d}{q}$ and compute the power of $\lmb_{n}$ we obtain from~\eqref{eq:Tlmbn} to be
\begin{align*}
	\tfrac{d}{2} - \tfrac{d}{q'} - \gmm + \sgm 
	=& -(d-1)(\tfrac{1}{2} - \tfrac{1}{q}) + \tfrac{1}{p}.
\end{align*}

Next, we deal with the long time contribution. This time we use the second part of Proposition \ref{prop: disp} for $\abs{t} > 2$ and~\eqref{eq:Tlmbn}. Proceeding as before, we obtain the following chain of estimates:
\begin{equation} \label{eq:uniformScat4conc:pf:2}
\begin{aligned}
& \hskip-2em
	\nrm{S_{\hyp}(t) \T_{\lmb_{n}}(f, g)}_{L^{p}_{t}((2, \infty); W^{1-\gmm, q} \times W^{-\gmm, q}(\bbH^{d}))}  \\
	\aleq & \nrm{\abs{t}^{-\frac{3}{2}}}_{L^{p}_{t}((2, \infty))} \nrm{\T_{\lmb_{n}} (f,g)}_{W^{1-\gmm+\sgm, q'} \times W^{-\gmm+\sgm, q'}(\bbH^{d})} \\
	\aleq & \lmb_{n}^{-(d-1)(\frac{1}{2} - \frac{1}{q})+\frac{1}{p}} \nrm{(f, g)}_{W^{1-\gmm+\sgm, q'} \times W^{-\gmm+\sgm, q'}(\bbR^{d})}.
\end{aligned}
\end{equation}
As $-(d-1)(\frac{1}{2} - \frac{1}{q}) + \frac{1}{p} \leq -\frac{1}{p}< 0$ by hypothesis, it follows that the last line goes to $0$ as we take $\limsup_{n \to \infty}$. Combining~\eqref{eq:uniformScat4conc:pf:1} and~\eqref{eq:uniformScat4conc:pf:2}, the desired inequality~\eqref{eq:uniformScat4conc:reduce} and hence the conclusion of the lemma follows. \qedhere
\end{proof}

We are now ready to state and prove the analogue of Corollary \ref{cor: h h:lin:smallScat} in the case of a concentrating profile.
\begin{cor} \label{cor: e h:lin:smallScat}
Let $(f, g) \in \HH_{\euc}(\bbR^{d})$ be an initial data set and $\set{\lmb_{n}} \subseteq [1, \infty)$, $\set{h_{n}} \subseteq \GG$ be sequences such that $\lmb_{n} \to \infty$ as $n \to \infty$. Let $(p, q, \gmm)$ be a hyperbolic admissible triple such that $\gmm = 1$ and
\begin{equation*} 
	\frac{2}{p} +  \frac{d-1}{q}  \leq \frac{d-1}{2}, \quad 2 < p < \infty.
\end{equation*}
Then there exists an absolute constant $C > 0$ such that the following holds: If
\begin{equation*}
	\nrm{S_{\euc, 0}(t)(f,g)}_{L^{p}_{t} ([0, \infty) ; L^{q}(\bbR^{d}))} < \eps,
\end{equation*}
for some $\eps > 0$, then we have
\begin{equation} \label{eq: e h:lin:smallScat}
	\limsup_{n \to \infty }\nrm{S_{V, 0} (t)\tau_{h_{n}^{-1}} \T_{\lmb_{n}} (f, g)}_{L^{p}_{t} ([0, \infty) ; L^{q}(\bbH^{d}))} < C \eps.
\end{equation}
An analogous statement holds in the negative time direction as well.
\end{cor}
\begin{proof} 
By Lemma \ref{lem:uniformScat4conc}, there exists a $T > 0$ such that
\begin{equation*}
	\sup_{n} \nrm{S_{V, 0}(t) \tau_{h_{n}^{-1}} \T_{\lmb_{n}}(f, g)}_{L^{p}([T/\lmb_{n}, \infty); L^{q}(\bbH^{d}) )} < \eps. 
\end{equation*} 
On the other hand, by Lemma \ref{lem: e h:lin}, it follows that
\begin{equation*}
	\limsup_{n \to \infty } \nrm{S_{V, 0} (t) \tau_{h_{n}^{-1}} \T_{\lmb_{n}} (f, g) - \tau_{h_{n}^{-1}} \T_{\lmb_{n}}^{0} S_{\euc, 0}(\lmb_{n} t) (f, g)}_{L^{p}_{t} ([0, T / \lmb_{n}) ; L^{q}(\bbH^{d}))} = 0,
\end{equation*}
where we have
\begin{equation*}
	\nrm{\tau_{h_{n}^{-1}} \T_{\lmb_{n}}^{0} S_{\euc,0}(\lmb_{n} t) (f, g)}_{L^{p}_{t}([0, T/\lmb_{n}) ; L^{q}(\bbH^{d}))}
	\aleq \nrm{S_{\euc, 0} (f, g)}_{L^{p}_{t}([0, T), L^{q}(\bbR^{d}))} < \eps.
\end{equation*}
Combining the preceding statements, the corollary follows. \qedhere
\end{proof}


\section{Bahouri-G\'erard linear profile decomposition}\label{sec: bg}

The goal of this section is to prove an analog of the linear Bahouri-G\'erard profile decomposition described in~\eqref{bg euc0} for sequences of linear waves $\vec u_n$ on $\R \times \Hp^d$ with bounded hyperbolic free energy, i.e., $\| \vec u_n\|_{\HH} \le C$. With further applications in mind, we will consider a more general class of linear wave equations than simply the free wave equation on $\R \times \Hp^d$. In particular our results will apply to bounded energy sequences of solutions $\vec u_n \in \HH$ to equations of the form 
\EQ{ \label{eq: u V free}
&u_{tt} - \De_{\Hp^d} u + V u = 0\\
}
where $V: \Hp^d \to  \R$ is a potential satisfying the assumptions in~\eqref{eq:Vdecay}--\eqref{eq:Str4halfVWave:2}. 



\subsection{Precise statement of the linear profile decomposition} 
 We begin by outlining a procedure for extracting limiting profiles from the sequence $\vec u_n$, which is originally due to Bahouri and G\'erard~\cite{BG} in the case of $\bbR^{1+3}$, and due to Ionescu, Pausader, and Staffilani~\cite{IPS} in this hyperbolic space incarnation.

 To identify a nonzero profile  that carries a nonzero amount of energy, one must first identify locations in space, time as well as the \emph{scale} at which the wave is concentrated. For waves that live at a fixed scale, the space and time translations as well as the limiting profiles are identified by examining the nonzero weak limits of
\EQ{
\tau_{h_n}S_{V}( t_n) \vec u_n(0) \rightharpoonup \vec V^j(0) \ \hbox{in} \ \HH
} 
for arbitrary sequences $\{t_n\} \subseteq \R$ and $\{ h_n\} \subseteq \GG$. However, such limits would fail to identify a nonzero limiting profile which arises due to concentration into smaller and smaller scales. One of the the key insights in~\cite{IPS} is that suitably defined Littlewood-Paley type projections $P_{\la}$ can be used to capture this particular failure of compactness. 

%
%

Indeed, with the goal of proving a decomposition as in Theorem~\ref{thm:mainbg}, we seek to control the errors generated after extracting profiles in a suitable Strichartz norm $S(\bbR)$. To see directly how the the projections $P_\la$ are used, we make the following reduction, which is a simple consequence of Strichartz estimates and the refined Sobolev embedding, Lemma~\ref{lem: se}. 
\begin{lem}\label{lem: BS} 
Let $(p, q, \gmm)$ be a hyperbolic-admissible triple that is \emph{not} an endpoint in the sense that
\begin{equation} \label{eq:BS:non-endStr}
	\frac{2}{p} +  \frac{d-1}{q}  \leq \frac{d-1}{2}, \quad p > 2, \quad (p, q, \gmm) \neq (\infty, 2, 0).
\end{equation}
Let $\{\vec w^k(t)\}_{k \in \N}$ be a sequence of solutions to~\eqref{eq: u V free} with uniformly bounded energy, i.e., $E_{V}(\vec{w}^{k}) \le C$. Suppose that 
\EQ{  \label{BS vanish}
\sup_{\la \ge 1, \, t \in \R, \, x \in \Hp^d} \abs{ \la^{-\frac{d-2}{2}}(P_\la w^k(t))( x)} \to 0 \mas k \to \infty
}
Then we have 
\begin{equation*}
\| w^k\|_{L^{p}_{t} (\R; W^{1-\gmm, q}_{x})} \to 0 \mas k  \to \infty.
\end{equation*}
\end{lem} 
\begin{proof}
The non-endpoint assumption \eqref{eq:BS:non-endStr} allows us to interpolate the $L^{\infty}_{t} L^{\frac{2d}{d-2}}(\bbH^{d})$ norm, which goes to $0$ thanks to the refined Sobolev inequality (Lemma \ref{lem: se}), with another Strichartz norm, which is uniformly bounded thanks to the Strichartz inequality (Lemma~\ref{lem:basicVWave} part (\ref{item:basicVWave:2})), to estimate $\| w^k\|_{L^{p}_{t} (\R; W^{1-\gmm, q}_{x})}$. We omit the details.
\end{proof} 

The point of Lemma~\ref{lem: BS} is that it motivates the following procedure for identifying nonzero limiting profiles -- in particular, it will suffice to look for profiles which carry a nonzero amount of the norm on the left-hand side of~\eqref{BS vanish}. 

We now begin an extended definition of  the limiting profiles associated to a sequence of solutions to~\eqref{eq: u V free} with bounded energy.
\begin{defn} \label{def:linprof} Let $\vec  u_n(t,  \cdot) \in \HH$ be a sequence of solutions to~\eqref{eq: u V free}, with uniformly bounded energy, i.e., $E_{V}(\vec{u}_{n}) \leq C$. Let $\{t_n, h_n, \la_n\} \subseteq \R \times \GG \times [1, \infty)$ be a sequence so that  
\ant{
0< \limsup_{n \to \infty} \abs{\la_n^{-\frac{d-2}{2}}P_{\la_n} S_{V, 0}(t_n)\vec u_n(0)}( h_n  \cdot \ori).
}

\vskip.5em
\noindent\textit{Case 1: A perturbed hyperbolic (or stationary) profile.} Suppose that 
\begin{equation*}
	\limsup_{n \to \infty}( \la_{n} + \abs{h_{n}}) < \infty,
\end{equation*}
where $\abs{h_{n}}$ is a shorthand for $\bfd_{\bbH^{d}}(h_{n} \cdot \zero, \zero)$. Then, up to passing to a subsequence, we can assume $\la_n \to \la_{\infty} \in [1, \infty)$ and $h_{n} \to h_{\infty} \in \GG$. Denote by $ \vec U_{V}(0)$ a weak limit in $\HH_{V}( \Hp^d)$ of the sequence 
\EQ{
\left(S_{V, 0}(t_n)\vec u_n(0), \, S_{V, 1}(t_n)\vec u_n(0) \right) \rightharpoonup \vec U_{V}(0) \ \hbox{in} \ \HH( \Hp^{d}).
}
We refer to $\vec U_{V}(0)$  as a \emph{limiting perturbed hyperbolic} (or \emph{stationary}) \emph{profile} associated to the sequence $\vec u_n(t)$. 
The profile $\vec U_{V}(0)$ will be related back to the original sequence $\vec u_n(0)$ by evolving back for the time $-t_n$ using the perturbed hyperbolic evolution 
\EQ{ \label{V prof}
\vec U_{V, n, L}(t):= \vec U_{V, L}(t-t_n, \cdot) := S_{V}(t-t_n) \vec U_{V}(0).
}
and we refer to $\vec U_{V, n, L}(0)$ as \emph{perturbed hyperbolic} (or \emph{stationary}) \emph{profiles} associated to the sequence $\vec u_n(0)$.

\begin{rem}
In practice we will take $h_n$ to be the sequence consisting of only the identity element in the case described above. That this can be done is justified in the process of extraction of profiles in the proof of Theorem~\ref{thm:  BG}.
\end{rem}

\vskip.5em
\noindent \textit{Case 2: A free hyperbolic (or traveling) profile.} Suppose that 
\begin{equation*}
\limsup_{n \to \infty} \la_n < \infty \quad \hbox{ but } \limsup_{n \to \infty} \abs{h_{n}} = \infty.
\end{equation*}
Up to passing to a subsequence, we can assume $\la_n \to \la_{\infty} \in [1, \infty)$. 
In this case, the solution lives at a fixed scale but is traveling out to the spatial infinity. Thanks to the assumption~\eqref{eq:Vdecay} on the spatial decay of the potential $V$, in this situation we expect the \emph{free} hyperbolic evolution to be a good approximation for the full evolution $S_{V}$.

With these heuristics in mind, we now define precisely the free hyperbolic (or traveling) profiles. Denote by $ \vec U_{\hyp}(0)$ a weak limit in $\HH( \Hp^d)$ of the sequence 
\EQ{
\left(\tau_{h_n}  S_{V, 0}(t_n)\vec u_n(0), \, \tau_{h_n}  S_{V, 1}(t_n)\vec u_n(0) \right) \rightharpoonup \vec U_{\hyp}(0) \ \hbox{in} \ \HH( \Hp^{d}).
}
We refer to $\vec U_{\hyp}(0)$  as a \emph{limiting free hyperbolic} (or \emph{traveling}) \emph{profile} associated to the sequence $\vec u_n(t)$. 
We relate the profile $\vec U_{\hyp}(0)$ back to the original sequence $\vec u_n(0)$ by the evolution
\EQ{ \label{hyp prof}
\vec U_{\hyp, n, L}(t) = \vec U_{\hyp, L}(t-t_n, h_n^{-1} \cdot) := S_{V}(t-t_n) \tau_{h_n^{-1}} \vec U_{\hyp}(0).
}
We will refer to $\vec U_{\hyp, n, L}(0)$ as \emph{free hyperbolic} (or \emph{traveling}) \emph{profiles} associated to the sequence $\vec u_n(0)$.  

\begin{rem}
As discussed above, by Lemma \ref{lem:pert4traveling}, we see that the wave $\vec{U}_{\hyp, n, L}$ is well-approximated by the \emph{free} hyperbolic wave, i.e.,
\begin{equation} \label{hyp prof is free}
	\nrm{\vec{U}_{\hyp, n, L}(t) - S_{\hyp}(t-t_{n}) \tau_{h_{n}^{-1}} \vec{U}(0)}_{L^{\infty}_{t} (\R; \HH)} \to 0 \mas n \to \infty.
\end{equation}
\end{rem}

\vskip.5em
\noindent \textit{Case 3: A Euclidean (or concentrating) profile.} Suppose that 
\begin{equation*}
\limsup_{n \to \infty} \la_n =  \infty.
\end{equation*} 
Then, up to passing to a subsequence, we can assume $\la_n \to \infty $. In this case, the solution is concentrating at a very small scale in physical space. This is precisely the situation in which the underlying scale-invariant Euclidean free evolution serves as a good approximation for the hyperbolic evolution, and this motivates the following lengthy definition of what we refer to as a \emph{Euclidean} (or \emph{concentrating}) \emph{profile}. The following definition is necessary to properly state  Theorem~\ref{thm: BG} below.

Begin by evolving the original sequence $\vec u_n(0)$ up to time $t_{n}$, 
 \EQ{
  \vec u_n (0) \longmapsto S_{V}(t_{n}) \vec u_{n}(0) =: ( u_{n, L}\left(t_{n},r , \om \right), \p_t u_{n, L}\left(t_{n}, r, \om \right))
  }
  We next take the pair of functions on the right-hand side  above, which we are viewing as functions at the level of the  coordinates $(r, \om)$, upstairs to $\Hp^d \subset \R^{d+1}$ with the coordinate map  via  the identification 
\begin{multline}
( u_{n, L}\left(t_{n},r , \om \right), \p_t u_{n, L}\left(t_{n}, r, \om \right)) \\
 \longleftrightarrow   (u_{n, L}\left(t_{n}, \Psi(r, \om) \right), \, \p_t u_{n, L}\left(t_{n}, \Psi(r, \om) \right))
\end{multline}
Next, we translate by $h_{n} \in \GG$, 
\EQ{
(u_{n, L}\left(t_{n}, \Psi(r, \om) \right), \, &\p_t u_{n, L}\left(t_{n}, \Psi(r, \om) \right)) \longmapsto  \tau_{h_{n}} S_V(t_{n}) \vec u_n(0)( \Psi(r, \om)) \\
& = \left((S_{V}(t_{n}) \vec u_n(0)\right)(  h_{n} \cdot \Psi(r, \om))
}
 Now, truncate at the level of coordinates by  $\la_{n}^{-\frac{1}{2}}$ and rescale by $\la_{n}$, 
\begin{multline}
  \left((S_{V, 0}(t_{n}) \vec u_n(0)\right)(  h_{n} \cdot \Psi(r, \om)) \longmapsto \\
   \la_{n}^{-\frac{d-2}{2}}  \chi \left( \frac{r}{ \sqrt{\la_{n}}}\right) \,  u_{n, L}\left(t_{n} , \,  h_{n} \cdot \Psi\left(\frac{ r }{ \la_{n}}, \om\right) \right) 
   = :v_{n, 0}(r, \om)
  \end{multline}
  \begin{multline}
  \left((S_{V, 1}(t_{n}) \vec u_n(0)\right)(  h_{n} \cdot \Psi(r, \om)) \longmapsto \\
   \la_{n}^{-\frac{d}{2}}  \chi \left(\frac{r}{ \sqrt{\la_n}}\right) \,  \p_t u_{n, L}\left(t_{n},  \,h_{n} \cdot \Psi\left(\frac{ r }{ \la_{n}}, \om\right) \right) =:
   v_{n}(r, \om)
  \end{multline}
where $\chi$ is radial and supported in $\set{ r \leq 2}$, with $\chi(r) = 1$ for $r \leq 1$.
Note that the sequence $ \vec v_n := (v_{n, 0}, v_{n, 1})$ is uniformly bounded in $\HH_{\euc} := \dot{H}^1 \times L^2(\R^d)$.  We can then extract a weak limit, $\vec V_{\euc}(0) \in \HH_{\euc}$, i.e., 
\EQ{
 \vec v_n \rightharpoonup \vec V_{\euc}(0) \ \hbox{in} \ \HH_{\euc}. 
 }
We will refer to $\vec V_{\euc}(0)$ as a \emph{limiting Euclidean} (or \emph{concentrating}) \emph{profile} associated to the sequence $\vec u_n(0)$.  In order to relate this limiting profile back to the original sequence, we first pass back to the hyperbolic energy space $\HH$ via the map $\T_{\la_{n}}:  \HH_{\euc} \to \HH$ and set 
\EQ{
\vec U_{\euc, n}(0) := (\T_{\la_{n}} \vec V_{\euc})(0)
}
Finally, we define the \emph{Euclidean} (or \emph{concentrating}) \emph{profiles},  $\vec U_{\euc, n, L}(0)$, associated to the sequence $\vec u_n(0)$ and the parameters $\{t_{n}, h_n, \la_n\}$, by defining
\EQ{ \label{euc prof}
\vec U_{\euc, n, L}(t) &: =  S_{V}(t-t_{n}) \tau_{h_{n}^{-1}}\vec U_{\euc, n}(0)
}
and setting $t = 0$.

\end{defn}

With the above definitions we can now state the main result of this section, namely a Bahouri-G\'erard profile decomposition for linear waves on $\R \times \Hp^d$.

\begin{thm}[Bahouri-G\'erard profile decomposition]\label{thm: BG} 
Let $\vec u_n := (u_{0, n}, u_{1,n})$ be a sequence in $\HH(\Hp^d)$ with $E_{V}(\vec u_n) \le C$. Then, up to passing to a subsequence, there exist sequences 
$\{ t_{n, j}, h_{n, j}, \la_{n, j}\} \in \R \times \GG \times [1, \infty)$, and corresponding profiles $\vec U^j_{n, L}(0)$ defined by
\begin{equation*}
	\vec U^{j}_{n, L}(0) =
	\left\{
	\begin{array}{lll}
	\vec{U}^{j}_{V, n, L}(0) & \hbox{ as in~\eqref{V prof}} & \hbox{ if $\la_{n, j} \to \la_{\infty, j} \geq 1$ and $h_{n, j} \to h_{\infty, j}$}, \\
	\vec{U}^{j}_{\hyp, n, L}(0) & \hbox{ as in~\eqref{hyp prof}} & \hbox{ if $\la_{n, j} \to \la_{\infty, j} \geq 1$ and $\abs{h_{n, j}} \to \infty$}, \\
	\vec{U}^{j}_{\euc, n, L}(0) & \hbox{ as in~\eqref{euc prof}} & \hbox{ if $\la_{n, j} \to \infty$},
	\end{array}
	\right.
\end{equation*}
and  errors $\vec w_{n}^J$ defined by 
\EQ{
\vec u_n (0) =  \sum_{0<j<J  }\vec U^j_{n, L}(0)+  \vec w_{n}^J , 
}
so that the following statements hold.

\begin{enumerate}
\item  Let $\vec w_n^J(t):= S_{V}(t) \vec w^J_n$. Then for any $(p, q, \gmm)$ as in~\eqref{eq:BS:non-endStr}, we have 
 \begin{align} \label{b norm to 0}
  & \limsup_{n \to \infty}  \sup_{ \la  \ge 1, t \in \R, x \in \Hp^d} \abs{  \la^{- \frac{d-2}{2}} (P_{\la} w_{n}^J(t)) (x) }  \to 0 \mas J \to \infty \\
   \mand & \limsup_{n \to \infty}\| w_n^J\|_{L^{p}_{t}(\R; W^{1-\gmm, q}_{x}) } \to 0 \mas J \to \infty \label{s norm to 0}
\end{align}

\item For each $J \in \N$ and for every $j \neq k  <J$, we have the following orthogonality of parameters, 
\EQ{ \label{o pa}
&\textrm{either}  \quad \frac{\la_{n, j}}{ \la_{n, k}} + \frac{\la_{n, k}}{\la_{n, j}}  \to \infty, \\
 & \textrm{or} \quad  \la_{n, j} \simeq \la_{n, k}  \\
&\phantom{\textrm{or} \quad} \mand   \la_{n, j}\abs{t_{n, j} - t_{n, k}} +  \la_{n, j}\bfd_{\Hp^d}(h_{n, j} \cdot \zero,  h_{n, k} \cdot \zero) \to \infty.
  }

\item For each $J \in \N$  and  $j<J$, if $\vec U^j_{n, L} =  \vec U_{V, n, L}$ arises as a perturbed hyperbolic profile, then we have the weak convergence 
\EQ{ \label{w weak V}
  S_{V}( t_{n,j}) \vec w_n^J(0) \rightharpoonup 0 \ \hbox{in} \  \HH.
  }
If $\vec U^j_{n, L} =  \vec U_{\hyp, n, L}$ arises as a free hyperbolic profile, then we have the weak convergence 
\EQ{ \label{w weak hyp}
  \tau_{h_{n, j}} S_{V}( t_{n,j}) \vec w_n^J(0) \rightharpoonup 0 \ \hbox{in} \ \HH.
  }
  If $\vec U^j_{n, L} = \vec U^j_{\euc, n, L}$ arises as a  Euclidean profile then we have the weak convergence, 
\begin{equation}  \label{w weak euc}
\begin{aligned}
  \bigg(\la_{n, j}^{-\frac{d-2}{2}}  & \chi( \cdot/ \la_{n, j}^{\frac{1}{2}}) ( \tau_{h_{n, j}}S_{V, 0}(t_{n, j}) \vec w_{n}^J)(\cdot/   \la_{n,j} )   , \,  \\
   &\la_{n, j}^{-\frac{d}{2}}  \chi( \cdot/  \la_{n, 1}^{\frac{1}{2}}) ( \tau_{h_{n, j}}S_{V, 1}(t_{n, j}) \vec w_{n}^J)( \cdot/  \la_{n,j} ) \bigg)  \rightharpoonup 0 \ \hbox{in} \ \HH_{\euc}.
\end{aligned}  
\end{equation}

\item Finally, we have the following Pythagorean decomposition of the energy: 
For each fixed $J \in \N$ we have 
\EQ{ \label{o en}
E_{V}( \vec u_n) = \sum_{j<J} E_{V} (\vec U^j_{n, L}) + E_{V}( \vec w^J_n) + o_n(1) \mas n \to \infty
}

\end{enumerate}
\end{thm}

\begin{proof} We begin by observing that in light of Lemma~\ref{lem: BS},~\eqref{s norm to 0} follows from~\eqref{b norm to 0}. Our goal will be to prove~\eqref{b norm to 0} with the remaining conclusions  to be established along the way. We proceed via a stopping time argument. Define 
\EQ{
 \nu_1 :=  \liminf_{n \to \infty}  \sup_{ \la  \ge 1, t \in \R, x \in \Hp^d} \abs{  \la^{- \frac{d-2}{2}} (P_{\la} u_{n, L}(t)) (x) }
}
where $\vec u_{n, L} (t):= S_{V}(t) \vec u_n$. If $\nu_1 = 0$ we stop here as we are done with $\vec w_n^k = \vec u_n$ for all $k$. If $\nu_1 >0$, we set $$\vec w_n^1:= \vec u_n, $$ and we can find  a sequence $\{t_{n, 1}, h_{n, 1}, \la_{n, 1}\}  \subseteq \R \times \GG \times [1, \infty)$ so that 
\EQ{
 \frac{\nu_1}{ 2} \le \abs{ \la_{n,1}^{- \frac{d-2}{2}} (P_{\la_{n, 1}} u_{n, L}(t_{n, 1}) )(h_{n, 1} \cdot  \mathbf{0})}
 }
 
 \vskip.5em
\textbf{Extraction of the first limiting profile}:
We divide the  extraction of the first profile into three separate cases. Note that after passing to  a subsequence  we can assume, without loss of generality that one of the following three scenarios hold:
\begin{enumerate}
\item $\la_{n, 1} \to \lmb_{\infty, 1} \geq 1$ and $h_{n, 1} \to h_{\infty, 1} \in \GG$;
\item $\la_{n, 1} \to \lmb_{\infty, 1} \geq 1$ and $\abs{h_{n, 1}} \to \infty$; or
\item $\la_{n, 1} \to \infty$.
\end{enumerate}

\vskip.5em
\noindent \textit{Case 1: A perturbed hyperbolic profile.} 
In the case $\la_{n, 1} \to \la_{ \infty, 1} \ge 1$ and $h_{n, 1} \to h_{\infty, 1} \in \GG$, consider the   sequence $\left( u_{n, L}(t_{n, 1}),   \p_t u_{n, L} (t_{n, 1}) \right)$, which is bounded in $\HH$. Therefore, we can extract a weak limit in $\HH$, i.e.,
\EQ{
 \left( u_{n, L}(t_{n, 1}),   \p_t u_{n, L} (t_{n, 1}) \right) \rightharpoonup \vec U^1_{V}(0) := (U_{V, 0}^1, U_{V, 1}^1) \in \HH.
  }
  We claim that 
  \EQ{ \label{U-V nu1}
	\nu_{1} \aleq
	 \| \vec U^1_{V}(0) \|_{\HH} \aleq 
	 \nrm{\vec{U}^{1}_{V}(0)}_{\HH_{V}(\bbH^{d})}
  }
The second inequality follows from Lemma \ref{lem:basicVWave}. To prove the first inequality, we use \eqref{eq:PP1} and proceed as follows:
\ant{
\nu_1 &\lesssim \la_{ \infty, 1}^{-\frac{d-2}{2}} \abs{ P_{\la_{\infty, 1}} u_{n, L}(t_{n, 1}) ( \zero)} \\
& =   \abs{ \ang{ u_{n, L}( t_{n, 1}) \mid  \la_{\infty, 1}^{-\frac{d+2}{2}} {}^{(1)} \PP_{\lmb_{\infty, 1}}}_{H^{1}( \Hp^d)}} \\
& =  \abs{ \ang{ U^{1}_{V, 0} \mid  \la_{\infty, 1}^{-\frac{d+2}{2}} {}^{(1)}\PP_{\lmb_{\infty, 1}}}_{H^{1}( \Hp^d)}} 
		+ o_{n}(1) \quad \hbox{ as } n \to \infty,
}
where we used the fact that $u_{n, L}( t_n, 1) \rightharpoonup U^1_{V, 0}$ weakly in $H^1(\Hp^d)$ on the last line. Then by Lemma \ref{lem:concPP}(\ref{item:concPP:1}), the claim~\eqref{U-V nu1} follows. 
  Now, set 
\EQ{
\vec U_{V,n, L}^1(t) &:=  S_{V}(t - t_{n, 1}) \vec U_{V}^1(0)  
 = \left( U_{V, L}^1 (t -t_{n, 1}, \cdot ), \, \p_t U_{V, L}^1 (t -t_{n, 1}, \cdot ) \right)
}
We define the second error term $\vec w_{n}^2$ as 
\EQ{
\vec w_n^2:=  \vec u_n -  \vec U^1_{V, n, L}(0)
}
Note that by construction we have 
\begin{align} 
S_{V}( t_{n, 1})\vec w_n^2 \rightharpoonup & \, 0 \ \hbox{in} \ \HH,  \label{V w^2 to 0} \\
E_{V}(\vec{U}^{1}_{V, n, L}) \ageq & \, \nu_{1}^{2}		.
\end{align}

\vskip.5em
\noindent \textit{Case 2: A free hyperbolic profile.} 
In the case $\la_{n, 1} \to \la_{ \infty, 1} \ge 1$ but $\abs{h_{n, 1}} \to \infty$, consider the sequence $\vec u_n^1:=  \left( \tau_{h_{n, 1}} u_{n, L}(t_{n, 1}),   \tau_{h_{n, 1}} \p_t u_{n, L} (t_{n, 1}) \right)$. Since this sequence is bounded in $\HH$, we can extract a weak limit, i.e., 
\EQ{
 \left( \tau_{h_{n, 1}} u_{n, L}(t_{n, 1}),   \tau_{h_{n, 1}} \p_t u_{n, L} (t_{n, 1}) \right)
  \rightharpoonup \vec U^1_{\hyp}(0) := (U_{\hyp, 0}^1, U_{\hyp, 1}^1) \in \HH.
  }
As in the previous case, we claim that 
  \EQ{ \label{U-hyp nu1}
 \| \vec U^1_{\hyp}(0) \|_{\HH} \gtrsim  \nu_1
  }
Proceeding as in Case 1, we have
\ant{
\nu_1 &\lesssim \la_{ \infty, 1}^{-\frac{d-2}{2}} \abs{ P_{\la_{\infty, 1}} u_{n, L}(t_{n, 1}) ( h_{n, 1} \cdot \zero)} \\
& =  \abs{ \ang{ U^{1}_{\hyp, 0} \mid  \la_{\infty, 1}^{-\frac{d+2}{2}} {}^{(1)} \PP_{\lmb_{\infty, 1}}}_{H^{1}( \Hp^d)}} 
		+ o_{n}(1) \quad \hbox{ as } n \to \infty.
}
Then by Lemma~\ref{lem:concPP}(\ref{item:concPP:1}), the desired claim \eqref{U-hyp nu1} follows. 
We set
\EQ{
\vec U_{\hyp,n, L}^1 (t)
&:=  S_{V}(t - t_{n, 1}) \tau_{h_{n, 1}^{-1}} \vec U_{\hyp}^1(0)  ,
}
and define the second error term $\vec w_{n}^2$ as 
\EQ{
\vec w_n^2:=  \vec u_n -  \vec U^1_{\hyp, n, L}(0).
}
By construction, we see that
\EQ{ \label{w^2 to 0}
\tau_{h_{n,1}} S_{V}( t_{n, 1})\vec w_n^2 \rightharpoonup 0   \in \HH.
}
Moreover, by Lemma \ref{lem:basicVWave} and~\eqref{U-hyp nu1}, we have
\begin{equation} 
	E_{V}(\vec{U}^{1}_{\hyp, n, L}) \ageq \nu_{1}^{2}.
\end{equation}
Finally, by Lemma \ref{lem:pert4traveling} (linear approximation for traveling profiles), it follows that $\vec{U}_{\hyp, n, L}^{1}$ is well-approximated by the a free hyperbolic evolution, more precisely,
\begin{equation} \label{}
	\nrm{\vec{U}^{1}_{\hyp, n, L}(t) - S_{\hyp}(t -t_{n, 1}) \tau_{h_{n,1}^{-1}} \vec{U}^{1}_{\hyp}(0)}_{L^{\infty}_{t} (\bbR; \HH)} \to 0.
\end{equation}

\vskip.5em
\noindent \textit{Case 3: A Euclidean profile.} Consider the case in which $\la_{n, 1} \to \infty$ as $n \to \infty$.  
To construct the limiting Euclidean profile, we begin by evolving the original sequence $\vec u_n(0)$ up to time $t_{n,1}$, 
 \EQ{
  \vec u_n (0) \longmapsto S_{V}(t_{n, 1}) \vec u_{n}(0) =: ( u_{n, L}\left(t_{n, 1},r , \om \right), \p_t u_{n, L}\left(t_{n, 1}, r, \om \right))
  }
  We next take the pair of functions on the right-hand side  above, which we are viewing as functions at the level of the  coordinates $(r, \om)$, upstairs to $\Hp^d \subset \R^{d+1}$ with the coordinate map  via  the identification 
\EQ{
\vec{u}_{n, L}(t_{n, 1}, r, \om) 
\longleftrightarrow \vec{u}_{n, L} (t_{n, 1}, \Psi(r, \omg)).
}
Next, we translate by $h_{n, 1}$:
\EQ{
\vec{u}_{n, L} (t_{n, 1}, & \Psi(r, \omg))
\longmapsto  \\
& \tau_{h_{n, 1}} S_{V}(t_{n, 1}) \vec u_n(0)( \Psi(r, \om))  
			= S_{V}(t_{n, 1}) \vec u_n(0)( h_{n, 1} \cdot \Psi(r, \om)) 
}
Then truncate at the level of coordinates by  $R_n := \la_{n,1}^{-\frac{1}{2}}$ and rescale by $\la_{n, 1}$: 
\begin{multline}
  \left((S_{V, 0}(t_{n, 1}) \vec u_n(0)\right)(  h_{n, 1} \cdot \Psi(r, \om)) \longmapsto \\
   \la_{n, 1}^{-\frac{d-2}{2}}  \chi \left( \frac{r}{ \la_{n, 1}R_n}\right) \,  u_{n, L}\left(t_{n, 1} , \,  h_{n, 1} \cdot \Psi\left(\frac{ r }{ \la_{n, 1}}, \om\right) \right) 
   = :v_{n, 0}^1(r, \om)
  \end{multline}
  \begin{multline}
  \left((S_{V, 1}(t_{n, 1}) \vec u_n(0)\right)(  h_{n, 1} \cdot \Psi(r, \om)) \longmapsto \\
   \la_{n, 1}^{-\frac{d}{2}}  \chi \left(\frac{r}{ \la_{n, 1}R_n}\right) \,  \p_t u_{n, L}\left(t_{n, 1},  \,h_{n, 1} \cdot \Psi\left(\frac{ r }{ \la_{n, 1}}, \om\right) \right) =:
   v_{n, 1}^1(r, \om)
  \end{multline}

Note that the sequence $ \vec v_n^1 := (v_{n, 0}^1, v_{n, 1}^1)$ is uniformly bounded in $\HH_{\euc} := \dot{H}^1 \times L^2(\R^d)$ (in fact this is so for all choices of $R_n$). By passing to a subsequence, we can then extract a weak limit, $\vec V_{\euc}^1(0) \in \HH_{\euc}$, i.e., 
\EQ{
 \vec v_n^1 \rightharpoonup \vec V^1_{\euc}(0) \in \HH_{\euc}. 
 }
We claim that
\EQ{ \label{V euc en}
\nu_1 \lesssim  \| \vec V^1_{\euc}(0)\|_{ \dot{H}^1 \times L^2 (\R^d)} .
}
Using \eqref{eq:PP1} and Lemma~\ref{lem:concPP}(\ref{item:concPP:2}), we have
\begin{align*}
\frac{\nu_{1}}{2}
\leq & \abs{\lmb_{n,1}^{-\frac{d-2}{2}} (P_{\lmb_{n, 1}} u_{n, L}(t_{n,1}))(h_{n,1} \cdot \zero)} \\
= & \ang{\tau_{h_{n,1}} u_{n, L}(t_{n,1}) \, \mid \, \lmb_{n,1}^{-\frac{d+2}{2}}  {}^{(1)} \PP_{\lmb_{n,1}}}_{H^{1}(\bbH^{d})} \\
= & \ang{\tau_{h_{n,1}} u_{n, L}(t_{n,1}) \, \mid \, \lmb_{n,1}^{-\frac{d+2}{2}} \eta_{\lmb_{n}^{-1} R} {}^{(1)} \PP_{\lmb_{n,1}}}_{H^{1}(\bbH^{d})} + o_{R}(1) 
\end{align*}
where $o_{R}(1)$ refers to a quantity that vanishes uniformly in $n$ as $R \to \infty$. We fix an $R \geq 1$ so that $o_{R}(1) \leq \frac{\nu_{1}}{4}$, and thus 
\begin{align*}
\frac{\nu_{1}}{4} \leq \ang{\tau_{h_{n,1}} u_{n, L}(t_{n,1}) \, \mid \, \lmb_{n,1}^{-\frac{d+2}{2}} \eta_{\lmb_{n}^{-1} R} {}^{(1)} \PP_{\lmb_{n,1}}}_{H^{1}(\bbH^{d})}
\end{align*}
for all $n$. We claim that there exists a function $\PP_{\infty} \in \dot{H}^{1}(\bbR^{d})$ such that, up to passing to a subsequence,
\begin{equation} \label{V euc en:pf:1}
\begin{aligned}
& \hskip-2em
\ang{\tau_{h_{n,1}} u_{n, L}(t_{n,1}) \, \mid \, \lmb_{n,1}^{-\frac{d+2}{2}} \eta_{\lmb_{n}^{-1} R} {}^{(1)} \PP_{\lmb_{n,1}}}_{H^{1}(\bbH^{d})} \\
=& \ang{V_{\euc, 0}^{1} (0) \mid \PP_{\infty}}_{\dot{H}^{1}(\bbR^{d})} + o_{n}(1) \mas n \to \infty.
\end{aligned}
\end{equation}
Moreover, $\PP_{\infty}$ satisfies
\begin{equation} \label{V euc en:pf:2}
	\nrm{\PP_{\infty}}_{\dot{H}^{1}(\bbR^{d})} \leq C \mand \supp \, \PP_{\infty} \subseteq \set{r \leq 2R},
\end{equation}
where $C$ is the absolute constant from Lemma~\ref{lem:concPP}(\ref{item:concPP:3}). Once we establish \eqref{V euc en:pf:1} and \eqref{V euc en:pf:2}, the desired conclusion \eqref{V euc en} would follow immediately.

Arguing as in Section~\ref{subsec:linApprox}, we may exploit the bound \eqref{eq:concPP:1} and the fact that the support of $\eta_{\lmb_{n}^{-1} R}$ shrinks to $\set{\zero}$ to relate $\lap_{\bbH^{d}}$, $\brk{\cdot \mid \cdot}_{L^{2}(\bbH^{d})}$ to their Euclidean counterparts $\lap_{\bbR^{d}}$, $\brk{\cdot \mid \cdot}_{L^{2}(\bbR^{d})}$, respectively, as $n \to \infty$ as follows:
\begin{align*}
& \hskip-2em
\ang{\tau_{h_{n,1}} u_{n, L}(t_{n,1}) \mid \lmb_{n,1}^{-\frac{d+2}{2}} \eta_{\lmb_{n}^{-1} R} {}^{(1)} \PP_{\lmb_{n,1}}}_{H^{1}(\bbH^{d})} \\
=& \ang{\tau_{h_{n,1}} u_{n, L}(t_{n,1})
		\mid (-\lap_{\bbH^{d}}) \lmb_{n,1}^{-\frac{d+2}{2}} \eta_{\lmb_{n}^{-1} R} {}^{(1)} \PP_{\lmb_{n,1}} }_{L^{2}(\bbH^{d})} \\
=& \ang{\tau_{h_{n,1}} u_{n, L}(t_{n,1}) (\Psi (r, \omg)) 
		\mid (-\lap_{\bbR^{d}}) \lmb_{n,1}^{-\frac{d+2}{2}} \eta_{\lmb_{n}^{-1} R} {}^{(1)} \PP_{\lmb_{n,1}} (\Psi (r, \omg)) }_{L^{2}(\bbR^{d})} + o_{n}(1)
\end{align*}
%
%
We omit the routine details. Then rescaling by $\lmb_{n}$ and noting that $\chi_{\lmb_{n}^{\frac{1}{2}}}(r) = 1$ on the support of $\eta_{R}(r)$ for sufficiently large $n$, the first term on the last line equals
\begin{align*}
= & \ang{\lmb_{n,1}^{-\frac{d-2}{2}} \tau_{h_{n,1}} u_{n, L}(t_{n,1}) (\Psi (\frac{r}{\lmb_{n}}, \omg)) 
		\mid \lmb_{n,1}^{-d} \eta_{R} {}^{(1)} \PP_{\lmb_{n,1}} (\Psi (\frac{r}{\lmb_{n}}, \omg)) }_{\dot{H}^{1}(\bbR^{d})} \\
= & \ang{v^{1}_{n, 0}(r, \omg) \mid f_{n}(r)}_{\dot{H}^{1}(\bbR^{d})} \mfor n \hbox{ sufficiently large},
\end{align*}
where $f_{n}$ is defined as in Lemma~\ref{lem:concPP}(\ref{item:concPP:3}). By the weak and strong $\dot{H}^{1}(\bbR^{d})$ convergences of $v^{1}_{n, 0}$ and $f_{n}$ up to passing to subsequences, respectively, \eqref{V euc en:pf:1} follows. Moreover, from Lemma~\ref{lem:concPP}(\ref{item:concPP:3}), we see that \eqref{V euc en:pf:2} holds as well.

Now we define the corresponding linear profiles. Recall the regularization-truncation-rescaling operator $\T_{\lmb_{n}}$ from Section \ref{subsec:linApprox}. We set
\EQ{ \label{U 1 euc}
\vec U^1_{\euc, n}(r, \om)(0) &:= (\T_{\la_{n, 1}} \vec V^1_{\euc})(r, \om)(0) \\
\vec U^1_{\euc, n, L}(t) &: =  S_{V}(t-t_{n, 1}) \tau_{h_{n, 1}^{-1}} \vec U^1_{\euc, n}(0)
}
We claim the following lower bound on the energy of $\vec U^1_{\euc, n, L}(t)$: 
\EQ{ \label{U euc en}
E_{V} (\vec U^1_{\euc, n, L}) \gtrsim  \nu_1^{2}
}
Indeed, using the facts that the energy $E_{V}$ is conserved under $S_{V}$ and that $E_{V}(\cdot)$ is equivalent to $\nrm{\cdot}_{\HH}^{2}$ by Lemma~\ref{lem:basicVWave}, we have
\begin{equation*}
E_{V}(\vec{U}^{1}_{\euc, n, L})
\gtrsim \|\vec U^1_{\euc, n}(0)\|_{\HH}^{2} = \nrm{\T_{\la_{n, 1}} \vec{V}^{1}_{\euc}}_{\HH}^{2}.
\end{equation*}
Then by \eqref{eq:Tlmbn:HH} and \eqref{V euc en}, the desired claim follows.

Next, define the error term 
\EQ{
\vec w_{n}^2:= \vec u_n - \vec U^1_{\euc, n, L}(0)
}
We claim that we have the weak convergence, 
\begin{equation} \label{w^2 euc to 0}
\begin{aligned}
\bigg(\la_{n, 1}^{-\frac{d-2}{2}}  & \chi ( \cdot/ \la_{n, 1}^{\frac{1}{2}}) ( \tau_{h_{n, 1}}S_{V, 0}(t_{n, 1}) \vec w_{n}^2)(\cdot/   \la_{n,1} )   , \,  \\
  &  \la_{n, 1}^{-\frac{d}{2}}  \chi( \cdot/ \la_{n, 1}^{\frac{1}{2}}) ( \tau_{h_{n, 1}}S_{V, 1}(t_{n, 1}) \vec w_{n}^2)( \cdot/  \la_{n,1} ) \bigg)  \rightharpoonup 0 \in  \dot{H}^1 \times L^2 (\R^d)
\end{aligned}
  \end{equation}
Indeed, first observe that we have
\EQ{
 \tau_{h_{n, 1}} S_{V}( t_{n, 1}) \vec w_n^2  =  \tau_{h_{n, 1}} S_{V}( t_{n, 1}) \vec u_n - \vec U^1_{\euc, n}(0)
 }
 Now truncating by $\chi_{R_n}$ (recall that $R_{n} = \lmb_{n}^{-\frac{1}{2}}$)  and rescaling by $\la_{n, 1}$ yields 
\begin{align*}
&\vec v_{n}^1 -  \left(\chi_{\la_{n, 1} R_n}  (\QQ_{\lmb_{n, 1}} V_{\euc, 0})( r, \om), \, \, \chi_{\la_{n,1} R_n}  (\QQ_{\lmb_{n, 1}} V_{\euc, 1})( r, \om) \right) \\
& = \vec v_{n}^1 -  \big(\chi_{\la_{n, 1} R_n} \chi_{ \lmb_{n, 1}^{\frac{1}{2}}}  (e^{\lmb_{n, 1}^{-1} \De} V_{\euc, 0})( r, \om), \, \, 
	\chi_{\la_{n,1} R_n}  \chi_{\lmb_{n, 1}^{\frac{1}{2}}}  (e^{\lmb_{n, 1}^{-1} \De}  V_{\euc, 1})( r, \om) \big)\\
& = \vec v_{n}^1 -  \big(\chi^2_{\la_{n, 1}^{\frac{1}{2}}}   (e^{\la_{n, 1}^{-1} \De} V_{\euc, 0})( r, \om), \, \, \chi^2_{\la_{n,1}^{\frac{1}{2}}}   (e^{\la_{n, 1}^{-1} \De}  V_{\euc, 1})( r, \om) \big)\\
& =  \vec v_n^1- \vec V_{\euc}(0) + o_{n}(1)
\end{align*}
on the right-hand side above, where the last line is in the sense of a strong limit in $\dot{H}^{1} \times L^2 (\R^d)$.


 \vskip.5em
\textbf{The second profile and orthogonality of the parameters}:  We have now successfully extracted the first profile and formed the second error, $\vec w_n^2$. Next we extract the second profile from the  sequence $\vec w^2_n$ which is uniformly bounded in $\HH$ regardless of the type of the first profile. Set 
\EQ{
\nu_2 :=  \sup_{\la \ge1,  \, h \in \GG, \, t \in \R} \abs{ \la^{-\frac{d-2}{2}} (P_{\la} S_{V, 0}(t) \vec w_n^2)( h \cdot  \mathbf{0})}
}
If $\nu_{2} = 0$, we stop here and we are done as there is only one nonzero profile $\vec U^1$.  If not, we proceed as before, finding a sequence $ \{(t_{n, 2}, h_{n, 2}, \la_{n, 2})\} \subseteq \R \times \GG \times [1, \infty)$ so that 
\EQ{
\frac{ \nu_2}{2}  \le \liminf_{n \to \infty} \abs{ \la_{n, 2}^{-\frac{d-2}{2}} (P_{\la_{n, 2}} S_{V, 0}(t_{n, 2}) \vec w_n^2)( h_{n, 2} \cdot  \mathbf{0})}
}
We divide into three cases, namely $(1)$ $\la_{n,2} \to \la_{\infty,2} \geq 1$ and $h_{n, 2} \to h_{\infty, 2} \in \GG$, $(2)$ $\la_{n, 2} \to \la_{\infty, 2} \ge 1$ but $\abs{h_{n, 2}} \to \infty$, or $(3)$ $\la_{n,2} \to \infty$. Arguing as before, we extract from $\vec{w}_{n}^{2}$ a perturbed hyperbolic profile $\vec{U}_{V, n}^{2}(0)$ in the first case, a free hyperbolic profile $\vec U_{\hyp, n}^2(0)$ in the second case, or a Euclidean profile, $ \vec U_{\hyp, n}^2(0)$ in the third case. 
Let
\EQ{
&\vec U_{V, n, L}^2(t) := S_{V}( t- t_{n, 2}) \vec U^2_{\hyp, n}(0)\\
&\vec U_{\hyp, n, L}^2(t) := S_{V}( t- t_{n, 2}) \tau_{h_{n, 2}^{-1}}  \vec U^2_{\hyp, n}(0)\\
&\vec U_{\euc, n, L}^2(t) := S_{V} ( t- t_{n, 2}) \tau_{h_{n, 2}^{-1}} \vec U^2_{\euc, n}(0)
}
Then we define the third error term as
\begin{equation*}
	\vec{w}_{n}^{3} 
	:= \vec{w}_{n}^{2} - \vec{U}^{2}_{n, L}(0)
	= \vec{u}_{n} - \vec{U}^{1}_{n, L}(0) - \vec{U}^{2}_{n, L}(0)
\end{equation*}
where we have suppressed the subscripts $V, \euc, \hyp$.  Note that in all three cases, we have the estimate
\EQ{
E_{V} (\vec U_{n, L}^2) \gtrsim  \nu_2^{2}.
}
Moreover, depending on whether $\vec{U}_{n, L}^{j} = \vec{U}_{V, n, L}^{j}$, $\vec{U}_{\hyp, n, L}^{j}$ or $\vec{U}_{\euc, n, L}^{j}$, the weak convergence statement~\eqref{w weak V},~\eqref{w weak hyp} or~\eqref{w weak euc} holds, respectively, for $J = 3$ and $j = 1, 2$.

Our next order of business is to show that a necessary consequence of this procedure is that the parameters $ \{(t_{n, 1}, h_{n, 1}, \la_{n, 1})\}$ and  $\{(t_{n, 2}, h_{n, 2}, \la_{n, 2})\}$ are orthogonal to each other as in~\eqref{o pa}. We make the following claim: 
\begin{claim}\label{claim: good pa} Let $\vec U^1$ and $\vec U^2$ be the first two profiles, both nonzero and as defined above, with parameters $ \{(t_{n, 1}, h_{n, 1}, \la_{n, 1})\}$ and  $\{(t_{n, 2}, h_{n, 2}, \la_{n, 2})\}$. Then, 
\EQ{\label{good pa}
 &\textrm{either}  \quad \frac{\la_{n, 1}}{ \la_{n, 2}} + \frac{\la_{n, 2}}{\la_{n, 1}}  \to \infty\\
 & \textrm{or} \quad  \la_{n, 1} \simeq \la_{n, 2}  \mand   \la_{n, 1}\abs{t_{n, 1} - t_{n, 2}} +  \la_{n, 1}\bfd_{\Hp^d}(h_{n, 1} \cdot \zero,  h_{n, 2} \cdot \zero) \to \infty
  }
  \end{claim}
  
  \begin{proof}[Proof of Claim~\ref{claim: good pa}]
First, observe that~\eqref{good pa} is obvious when the two profiles $\vec U^{1}$ and $\vec U^{2}$ are of different types.
We may thus assume that $\vec U^{1}$ and $\vec U^{2}$ are of the same type. We divide the argument into three cases based the type of the profiles.

\vskip.5em
\noindent \textit{Case $1$: two perturbed hyperbolic profiles.}
Here, without loss of generality we can assume that $\la_{n,1} = \la_{n,2} = 1$. If~\eqref{good pa} fails then, up to passing to a further subsequence, we may assume that
\begin{equation} 
	t_{n,1} - t_{n,2} \to t_{0} \in \bbR \mas n \to \infty.
\end{equation}
Under these assumptions, we claim that
\begin{equation} \label{eq: good pa:V}
	\vec{u}_{n}^{2} := S_{V}(t_{n,2}) \vec{w}_{n}^{2} \rightharpoonup 0 \ \hbox{in} \ \HH,
\end{equation}
which contradicts the fact that the profile $\vec{U}_{V}^{2}(0)$ is nonzero. To prove this claim, first observe that it suffices to prove that $\vec{u}_{n}^{2} \rightharpoonup 0$ in $\HH_{V}$, as $\HH$ and $\HH_{V}$ are equivalent by Lemma~\ref{lem:basicVWave}. Now for any $\vec{\phi} \in C^{\infty}_{0} \times C^{\infty}_{0}(\Hp^{d})$, we have
\begin{align*}
\ang{\vec{u}^{2}_{n} \mid \vec{\phi}}_{\HH_{V}}
=& \ang{S_{V}(t_{n, 1}) \vec{w}_{n}^{2} \mid S_{V}(t_{n,1} - t_{n, 2}) \vec{\phi}}_{\HH_{V}} \\
=& \ang{S_{V}(t_{n,1}) \vec{w}_{n}^{2} \mid S_{V}(t_{0})\vec{\phi}}_{\HH} + o_{n}(1) \mas n \to \infty.
\end{align*}
From~\eqref{V w^2 to 0}, and again using the fact that $\HH$ and $\HH_{V}$ are equivalent, the desired contradiction~\eqref{eq: good pa:V} follows.

\vskip.5em
\noindent \textit{Case $2$: two free hyperbolic profiles.}
As in the previous case, we may assume that $\la_{n, 1} = \la_{n, 2} =1$. Suppose that~\eqref{good pa} fails. This means that, up to passing to further subsequences we can assume without loss of generality that 
  \EQ{ \label{eq: good pa:hyphyp}
  (t_{n, 1} - t_{n, 2}) \to t_0  \in \R   \mand  h_{n, 2}  \circ h_{n, 1}^{-1}  \to h_0 \in \GG \mas n \to \infty
  }
To see the latter point, note that we are assuming that $\bfd_{\Hp^d}( h_{n, 1} \cdot \zero,  h_{n, 2} \cdot \zero)$ is bounded, and hence up to extracting a subsequence we can extract a limit $\bfd_{\Hp^d}( h_{n, 1} \cdot \zero, h_{n, 2}\cdot \zero) \to c_0 >0$. Then using the Cartan decomposition $h_{n, j} = m_{n, j} \circ a_{s_{n, j}} \circ m'_{n,j}$ and the compactness of $\K$ we can extract a subsequence so that $h_{n, 1}  \circ h_{n, 2}^{-1}  \to h_0 \in \GG$. 

We claim that under these assumptions, we must have 
\EQ{
 \vec u_n^2:=  \tau_{h_{n, 2}} S_{V}( t_{n, 2}) \vec w_n^2(0) \rightharpoonup 0 \ \hbox{ in } \ \HH
 }
 which contradicts the fact that the profile $ \vec U_{\hyp}^2$ is nonzero. 

Indeed, let $\vec \phi = (\phi_{0}, \phi_{1}) \in C_0^{\infty} \times C^\infty_0( \Hp^d)$. We need to be slightly careful, because $S_{V}$ is anti-self-adjoint under $\HH_{V}$ whereas $\tau_{h}$ is anti-self-adjoint under $\HH$. To pass from $\HH$ to $\HH_{V}$ and vice versa, we will rely on the following elementary observation: For any sequence $\set{h_{n}} \subseteq \bbG$ such that $\abs{h_{n, 1}} \to \infty$, we have
\begin{equation} \label{eq:travelingV}
	\nrm{V \tau_{h_{n}} \phi_{0}}_{L^{2}(\bbH^{d})} = o_{n}(1) \mas n \to \infty.
\end{equation}
To prove \eqref{eq:travelingV}, first note that $V \tau_{h_{n}} \phi_{0} \to 0$ pointwisely, thanks to the compact support assumption on $\phi_{0}$ and the decay assumption \eqref{eq:Vdecay} on $V$. As $\phi_{0} \in L^{2}(\bbH^{d})$ as well, \eqref{eq:travelingV} follows by the dominated convergence theorem.

We begin by writing
\begin{align*}
  \ang{  \vec u_n^2 \mid \vec \phi}_{\HH} 
  &=  \ang{ S_{V}( t_{n, 2}) \vec w^2_n (0) \mid \tau_{h_{n, 2}^{-1}} \vec \phi}_{\HH} \\
  &=   \ang{ S_{V}( t_{n, 2}) \vec w^2_n (0) \mid \tau_{h_{n, 2}^{-1}} \vec \phi}_{\HH_{V}} + o_{n}(1) \\
   &=  \ang{ S_{V}( t_{n, 1}) \vec w^2_n (0) \mid S_{V}(t_{n,1} - t_{n,2}) \tau_{h_{n, 2}^{-1}} \vec \phi}_{\HH_{V}}  + o_{n}(1) \mas n \to \infty,
\end{align*}
where we used \eqref{eq:travelingV} on the second line. Then by Lemma~\ref{lem:pert4traveling}, \eqref{eq: good pa:hyphyp} and \eqref{eq:travelingV} again (note that $\tau_{h_{0}^{-1}} S_{\hyp, 0}(t_{0})  \vec \phi \in C^{\infty}_{0}(\bbH^{d})$), the last line equals
\begin{align*}
    &=  \ang{ S_{V}( t_{n, 1}) \vec w^2_n (0) \mid \tau_{h_{n, 2}^{-1}} S_{\hyp}(t_{n,1} - t_{n,2})  \vec \phi}_{\HH_{V}}  + o_{n}(1) \\
    &=  \ang{ S_{V}( t_{n, 1}) \vec w^2_n (0) \mid \tau_{h_{n, 1}^{-1}} \tau_{h_{0}^{-1}} S_{\hyp}(t_{0})  \vec \phi}_{\HH}  + o_{n}(1) \\
    &=  \ang{ \tau_{h_{n, 1}} S_{V}( t_{n, 1}) \vec w^2_n (0) \mid \tau_{h_{0}^{-1}} S_{\hyp}(t_{0})  \vec \phi}_{\HH}  + o_{n}(1)
\end{align*}
and the last line goes to zero as $n \to \infty$ by~\eqref{w^2 to 0}.

\vskip.5em 
\noindent \textit{Case $3$: two Euclidean profiles.}
Finally we consider the case of two Euclidean profiles, which means that $\la_{n, 1}, \la_{n, 2} \to \infty$. If~\eqref{good pa} fails, we can assume that $\la_{n, 1} = \la_{n, 2}  \to \infty$ and it follows that 
\EQ{
\abs{t_{n, 1} - t_{n, 2}} \to 0 \mand \bfd_{\Hp^d}( h_{n, 1} \cdot \zero, h_{n, 2} \cdot \zero)  \to 0 \mas n \to \infty
}
We can then further reduce to the case that $t_{n, 1} = t_{n, 2}$ and $h_{n, 1} = h_{n, 2}$. Recall that in the case of a Euclidean profile, the limiting profile $\vec V^2_{\euc}(0)$ is extracted as a weak limit in $\dot{H}^1 \times L^2 (\R^d)$ of the sequence  $\vec v_{n}^2 = ( v_{n, 0}^2, v_{n, 1}^2)$  defined by 
\EQ{
 &v^2_{n, 0}:=  \la_{n, 2}^{-\frac{d-2}{2}} \chi_{\la_{n, 2}^{\frac{1}{2}}} w^2_{n, L} (t_{n, 2} , \, h_{n, 2} \cdot \Psi( \frac{\cdot}{ \la_{n, 2}} , \om)) \\
 &v^2_{n, 1}:= \la_{n, 2}^{-\frac{d}{2}} \chi_{\la_{n, 2}^{\frac{1}{2}}} \p_t w^2_{n, L} (t_{n, 2} , \, h_{n, 2} \cdot \Psi( \frac{\cdot}{ \la_{n, 2}} , \om))
 }
 But then, since $t_{n, 2} = t_{n, 1}$, $h_{n, 2} = h_{n, 1}$ and $\la_{n, 2}  = \la_{n, 1}$ we have for $\vec \phi =(\phi_{0}, \phi_{1}) \in C^{\infty}_0 \times C^{\infty}_0(\R^{d})$,
 \begin{align*}
\ang{  \vec v_n^2 \mid \vec \phi}_{\HH_{\euc}} 
 = &  \ang {\la_{n, 1}^{-\frac{d-2}{2}} \chi_{\la_{n, 1}^{\frac{1}{2}}} w^2_{n, L} (t_{n, 1} , \, h_{n, 1} \cdot \Psi( \frac{\cdot}{ \la_{n, 1}} , \om))\, \mid \phi_{0}}_{\dot{H}^{1}(\R^{d})}  
 	\\
 &   + \ang{\la_{n, 1}^{-\frac{d}{2}} \chi_{\la_{n, 1}^{\frac{1}{2}}} \p_t w^2_{n, L} (t_{n, 1} , \, h_{n, 1} \cdot \Psi( \frac{\cdot}{ \la_{n, 1}} , \om))  \, \mid \phi_{1}}_{L^{2}(\R^{d})} 
\end{align*} 
goes to $0$ as $n \to \infty$  by~\eqref{w^2 euc to 0}. This would mean that $\vec v_{n}^2$ tends to zero weakly in $\HH_{\euc} = \dot{H}^1 \times L^2 (\R^d)$ which contradicts the fact that limiting profile, $V^2_{\euc}$, is assumed to be nonzero. This completes the proof of Claim~\ref{claim: good pa}. \qedhere 
\end{proof}

\vskip.5em
\textbf{Extraction of remaining profiles}: We now proceed by induction, finding numbers $\nu_j$, parameters $\{t_{n, j}, h_{n, j}, \la_{n, j}\} \subseteq  \R \times  \GG \times [1, \infty)$ and profiles $\vec U^{j}$, which are perturbed hyperbolic, free hyperbolic or Euclidean depending on the asymptotic behavior of the parameters $h_{n, j}, \la_{n, j}$, satisfying $E_{V}(\vec{U}^{j}_{n, L}) \ageq \nu_{j}^{2}$. The $J$-th error term, $\vec w^{J}_n$ is defined as 
\EQ{
\vec w^J_n: =  \vec u_n - \sum_{j<J} \vec U^j_{n, L}(0)
}
One can show using the exact same argument used to prove Claim~\ref{claim: good pa} that for each $J$ and for every $j \neq k  <J$ we have 
\EQ{
&\textrm{either}  \quad \frac{\la_{n, j}}{ \la_{n, k}} + \frac{\la_{n, k}}{\la_{n, j}}  \to \infty\\
 & \textrm{or} \quad  \la_{n, j} \simeq \la_{n, k}  \mand   \la_{n, j}\abs{t_{n, j} - t_{n, k}} +  \la_{n, j}\bfd_{\Hp^d}(h_{n, j} \cdot \zero,  h_{n, k} \cdot \zero) \to \infty
  }
Moreover, for each $j<J$, one of the following weak convergence statements hold, depending on the type of $\vec{U}^{j}_{n, L}$:
\begin{enumerate}
\item If $\vec U^j_{n, L} = \vec U^j_{V, n, L}$ arises as a perturbed hyperbolic profile, then we have
\EQ{
  S_{V}( t_{n,j}) \vec w_n^J(0) \rightharpoonup 0 \ \hbox{ in } \ \HH
  }
\item If $\vec U^j_{n, L} = \vec U^j_{\hyp, n, L}$ arises as a free hyperbolic profile, then we have
\EQ{
  \tau_{h_{n, j}} S_{V}( t_{n,j}) \vec w_n^J(0) \rightharpoonup 0 \ \hbox{ in } \  \HH
  }
\item If $\vec U^j_{n, L} = \vec U^j_{\euc, n, L}$ arises as a  Euclidean profile, then we have 
  \begin{multline} \label{w^J euc to 0}
  \bigg(\la_{n, j}^{-\frac{d-2}{2}}  \chi( \cdot/ \la_{n, j}^{\frac{1}{2}}) ( \tau_{h_{n, j}}S_{V, 0}(t_{n, j}) \vec w_{n}^J)(\cdot/   \la_{n,j} )   , \,  \\
   \, \la_{n, j}^{-\frac{d}{2}}  \chi( \cdot/  \la_{n, 1}^{\frac{1}{2}}) ( \tau_{h_{n, j}}S_{V, 1}(t_{n, j}) \vec w_{n}^J)( \cdot/  \la_{n,j} ) \bigg)  \rightharpoonup 0 \ \hbox{ in } \   \HH_{\euc}
  \end{multline}
  \end{enumerate}
  
\vskip.5em
\textbf{Orthogonality of the free energy:} Next, we prove the Pythagorean decomposition of the energy $E_{V}$. 
\begin{claim} \label{claim: o en}
For each fixed $k \in \N$ we have 
\EQ{
E_{V}( \vec u_n) = \sum_{j<k} E_{V} (\vec U^j_{n, L}(0)) + E_{V} (\vec w^k_n) + o_n(1) \mas n \to \infty.
}
\end{claim}

\begin{proof}[Proof of Claim~\ref{claim: o en}] 
It suffices to check that for  fixed $k \in \N$ and $j, \ell <k$, $ j \neq \ell$ we have 
\begin{align}
& \ang{ \vec U^{j}_{n, L}(0) \mid  \vec U^{\ell}_{n, L}(0) }_{\HH_{V}}  \to 0 \mas n \to \infty   \label{Uj Uell}\\
 & \ang{ \vec U^j_{n, L}(0) \mid \vec w^k_n }_{\HH_{V}} \to 0 \mas n \to \infty \label{Uj w}
 \end{align}

We begin by proving~\eqref{Uj Uell}. We divide the proof into several cases, depending on the types of the profiles $\vec{U}^{j}_{n, L}$ and $\vec{U}^{\ell}_{n, L}$.

\vskip.5em
\noindent \textit{Case $1$: one Euclidean and one perturbed or free hyperbolic profile}: 
In this case, say $\vec U^j_{n, L} = \vec U^{j}_{\euc, n, L}$ is Euclidean and $\vec U^\ell_{n, L}$ is perturbed or free hyperbolic. We will write $ \vec{U}^{\ell} = \vec{U}_{\Box, n, L}^{\ell}$ with $\Box = V$ or $\hyp$, and when $\Box = V$, we may take $h_{n, \ell} \equiv 0$. 
Without loss of generality we can assume that $\la_{n, \ell}  \equiv 1$, and we of course have $\la_{n, j} \to \infty$ as $n \to \infty$.  Recall that we have
\begin{equation} 
\begin{aligned}
\vec U_{\euc, n, L}^j(0) =&  S_{V}(-t_{n, j}) \tau_{h_{n, j}^{-1}} \T_{\la_{n, j}} \vec V^j_{\euc}(0) \\
 \vec U^\ell_{\Box, n, L}(0)=& S_{V}(-t_{n, \ell}) \tau_{h_{n, \ell}^{-1}}  \vec U^\ell_{\Box}(0).
\end{aligned}
\end{equation}
By approximation, it suffices to consider the case that the limiting profiles satisfy $\vec V^j_{\euc}(0) \in C^{\infty}_0 \times C^{\infty}_0(\R^d)$ and $\vec U^\ell_{\Box}(0) \in C^{\infty}_0 \times C^{\infty}_0(\Hp^d)$ and $\supp \, \vec{V}^{j}_{\euc}(0)$, $\supp \, \vec{U}^{\ell}_{\Box}(0)$ are contained in a same compact set $\{r \le R\}$. We have 
\begin{align*}
&\abs{\ang{ \vec U^{j}_{\euc, n, L}(0) \mid  \vec U^{\ell}_{\Box, n, L}(0) }_{\HH_{V}} }  \\ 
& = \abs{\ang{ S_{V}(-t_{n, j}) \tau_{h_{n, j}^{-1}} \T_{\la_{n, j}} \vec{V}^j_{\euc}(0) \mid  S_{V}(-t_{n, \ell}) \tau_{h_{n, \ell}^{-1}} \vec{U}^{\ell}_{\Box}(0) }_{\HH_{V}} }  \\
& = \abs{\ang{\tau_{h_{n, j}^{-1}} \T_{\la_{n, j}} \vec{V}^j_{\euc}(0) \mid S_{V}(t_{n, j}-t_{n, \ell}) \tau_{h_{n, \ell}^{-1}} \vec{U}^{\ell}_{\Box}(0) }_{\HH_{V} }   } 
\end{align*}
Then recalling the definition of $\HH_{V} = H^{1}_{V} \times L^{2}$, the last line can be estimated by
\begin{align*}
& \leq \nrm{\T_{\lmb_{n,j}} V_{\euc, 0}^{j}(0)}_{L^{2}} \nrm{D_{V}^{2} S_{V, 0}(t_{n, j}-t_{n, \ell}) \tau_{h_{n, \ell}^{-1}} \vec{U}^{\ell}_{\Box}(0)}_{L^{2}} \\
& \phantom{\leq}
		+ \nrm{\T_{\lmb_{n,j}} V_{\euc, 1}^{j}(0)}_{L^{\frac{2d}{d+2}}} \nrm{S_{V, 1}(t_{n, j}-t_{n, \ell}) \tau_{h_{n, \ell}^{-1}} \vec{U}^{\ell}_{\Box}(0) }_{L^{\frac{2d}{d-2}}},
\end{align*}
which goes to $0$ as $n \to \infty$, since we have  
\EQ{
& \|\T_{\la_{n, j}}^0V^j_{\euc, 0}\|_{L^2(\Hp^d)} \lesssim  \la_{n, j}^{-1} \| V^j_{\euc, 0}\|_{L^2(\R^d)} \to 0  \mas n \to \infty\\
 &\|\T_{\la_{n, j}}^1 V^j_{\euc, 1}\|_{L^{\frac{2d}{d+2}}(\Hp^d)} \lesssim \la_{n, j}^{-1} \| V^j_{\euc, 1}\|_{L^\frac{2d}{d+2}(\R^d)} \to 0 \mas n \to \infty
 }
by~\eqref{eq:Tlmbn},  and the factors with the evolutions $S_{V, 0}, S_{V, 1}$ are bounded by a constant thanks to the $H^{1} \hookrightarrow L^{\frac{2d}{d-2}}$ Sobolev inequality and the assumption that $\vec U^\ell_{\Box}(0) \in C^{\infty}_{0} \times C^{\infty}_{0}(\bbH^{d})$. 

\vskip.5em 
\noindent \textit{Case $2$: one free hyperbolic profile and one perturbed or free hyperbolic profile}: 
We may assume that $\vec{U}_{n, L}^{j} = \vec{U}_{\hyp, n, L}^{j}$ is free hyperbolic and $\vec{U}_{n, L}^{\ell}$ is perturbed or free hyperbolic. As before, we will write $ \vec{U}^{\ell} = \vec{U}_{\Box, n, L}^{\ell}$ with $\Box = V$ or $\hyp$, and when $\Box = V$, we take $h_{n, \ell} \equiv 0$. Then we have
\begin{equation} \label{eq:o en:one free hyp:1}
	\vec{U}_{\hyp, n, L}^{j}(0) = S_{V}(-t_{n, j}) \tau_{h_{n, j}^{-1}} \vec{U}_{\hyp}^{j}(0), \quad
	\vec{U}_{\Box, n, L}^{\ell}(0) = S_{V}(-t_{n, \ell}) \tau_{h_{n, \ell}^{-1}} \vec{U}_{\Box}^{\ell}(0).
\end{equation}

Again we can assume that both limiting profiles $\vec U_{\hyp}^j(0) , \, \vec U_{\Box}^\ell(0) \in C^{\infty}_0 \times C^{\infty}_0(\Hp^d)$  are smooth and supported in the same compact set $\{r \le R\}$. Without loss of generality we can assume that both $\la_{n, j} = \la_{n, \ell} = 1 $ for all $n$. 
 
 First consider the case 
\begin{equation}
 \abs{ t_{n, j} - t_{n, \ell}} \to \infty \mas n \to \infty.
\end{equation}
By~\eqref{eq:o en:one free hyp:1} and Lemma~\ref{lem:pert4traveling}, we have
\begin{align*}
& \hskip-2em
\abs{\ang{\vec{U}^{j}_{\hyp, n, L}(0) \mid \vec{U}^{\ell}_{\Box, n, L}(0)}_{\HH_{V}} } \\
& = \abs{\ang{ S_{V}(t_{n, \ell} - t_{n, j}) \tau_{h_{n, j}^{-1}} \vec{U}_{\hyp}^{j}(0) \mid \tau_{h_{n, \ell}^{-1}} \vec{U}^{\ell}_{\Box}(0)}_{\HH_{V}} } \\
& = \abs{\ang{ \tau_{h_{n, j}^{-1}} S_{\hyp}(t_{n, \ell} - t_{n, j})  \vec{U}_{\hyp}^{j}(0) \mid \tau_{h_{n, \ell}^{-1}} \vec{U}^{\ell}_{\Box}(0)}_{\HH_{V}} } + o_{n}(1) \mas n \to \infty.
\end{align*}
Then by the long time dispersive estimate in Proposition~\ref{prop: disp}, the assumption that $\vec U^j_{\hyp}(0), \vec U^\ell_{\Box}(0) \in C^{\infty}_{0} \times C^{\infty}_{0}(\bbH^{d})$ and the fact that $\abs{t_{n, j} - t_{n, \ell}} \to \infty$ as $n \to \infty$, it follows that the first term on the last line goes to $0$ as $n \to \infty$. We omit the details.

Next assume that we have 
\EQ{ \label{eq:o en:one free hyp:2}
 t_{n, j}- t_{n, \ell}\to t_0 \in \R, \mand \bfd_{\Hp^d}( h_{n, j} \cdot \zero , h_{n, \ell} \cdot \zero) \to \infty \mas n \to \infty
 }
Then by Lemma~\ref{lem:pert4traveling}, we have
\begin{align*}
&  \hskip-2em
 \abs{ \ang{ \vec U^{j}_{\hyp, n, L}(0), \vec U^{\ell}_{\Box, n, L}(0) }_{\HH_{V}} } \\
  = &  \abs{ \ang{ \tau_{h_{n, j}} S_{\hyp}(t_{n, \ell} - t_{n, j}) \vec U^{j}_{\hyp}(0), \tau_{h_{n, \ell}^{-1}} \vec U^{\ell}_{\Box, n, L}(0) }_{\HH_{V}} } + o_{n}(1)\\
 =&  \abs{\ang{\tau_{h_{n, \ell} \circ h_{n, j}^{-1}}  S_{\hyp}(-t_{0}) \vec U^j_{\hyp}(0)  \mid \vec U^\ell_{\Box}(0) }_{\HH} } \\
	& +\abs{\ang{\tau_{h_{n, \ell}} V \, \tau_{h_{n, \ell} \circ h_{n, j}^{-1}}  S_{\hyp, 0}(-t_{0}) U^j_{\hyp}(0)  \mid U^\ell_{\Box, 0}(0) }_{L^{2}} } + o_{n}(1) \mas n \to \infty,
\end{align*}
and both terms on the last line equal $0$ for $n$ sufficiently large, thanks to the support assumptions on $\vec U^j_{\hyp}(0), \vec U^\ell_{\Box}(0)$ and since $\bfd_{\Hp^d}( h_{n, j} \cdot \zero , h_{n, \ell} \cdot \zero) \to \infty$ as $n \to \infty$.

\vskip.5em
\noindent \textit{Case $3$: two perturbed hyperbolic profiles}: Since the dispersive estimate is not available for the perturbed wave equation~\eqref{eq: u V free}, we need to devise a different argument. The idea is to use instead the integrated local energy decay estimate~\eqref{eq:VLED}.

As in the preceding cases, by approximation, we may assume that $\vec{U}^{j}_{V}(0), \vec{U}^{\ell}_{V}(0) \in C^{\infty}_{0} \times C^{\infty}_{0}(\bbH^{d})$. Then for any smooth compactly supported cutoff $\chi_{R}$, from~\eqref{eq:VLED} it follows that
\begin{equation} \label{eq: o en:pf:VV:1}
	\nrm{(\chi_{R} D_{V} S_{V, 0}(t) \vec{U}^{j}_{V}(0), \chi_{R} S_{V, 1}(t) \vec{U}^{j}_{V}(0))}_{L^{2}_{t}(\bbR; L^{2} \times L^{2}(\bbH^{d}))} < \infty
\end{equation}
Thanks to the assumption that $\vec{U}^{j}_{V}(0) \in C^{\infty}_{0} \times C^{\infty}_{0}(\bbH^{d})$ and the fact that $S_{V}(t) \vec{U}^{j}_{V}(0)$ solves the equation $(\Box_{\bbH^{d}} + V) u = 0$, we see that
\begin{equation*}
	\nrm{(\chi_{R} D_{V} S_{V, 0}(t) \vec{U}^{j}_{V}(0), \chi_{R} S_{V, 1}(t) \vec{U}^{j}_{V}(0))}_{L^{2} \times L^{2}(\bbH^{d})}
\end{equation*}
is uniformly Lipschitz as a function of $t$. Then in order for the $L^{2}_{t}$ integral in~\eqref{eq: o en:pf:VV:1} to be finite, we must have
\begin{equation} \label{eq: o en:pf:VV:2}
	\nrm{(\chi_{R} D_{V} S_{V, 0}(t) \vec{U}^{j}_{V}(0), \chi_{R} S_{V, 1}(t) \vec{U}^{j}_{V}(0))}_{L^{2} \times L^{2}(\bbH^{d})} \to 0 \mas t \to \pm \infty.
\end{equation}

Now we turn to the proof of~\eqref{Uj Uell}. We begin by writing
\begin{align*}
& \hskip-2em
\abs{\ang{\vec{U}^{j}_{V, n, L}(0) \mid \vec{U}^{\ell}_{V, n, L}(0)}_{\HH_{V}} } \\
=&	\abs{\ang{S_{V}(t_{n, \ell}-t_{n, j}) \vec{U}^{j}_{V}(0) \mid \vec{U}^{\ell}_{V}(0)}_{\HH_{V}} } \\
=&	\abs{\ang{D_{V} S_{V, 0}(t_{n, \ell}-t_{n, j}) \vec{U}^{j}_{V}(0) \mid \chi_{R} D_{V} {U}^{\ell}_{V, 0}(0)}_{L^{2}}}  \\
& + \abs{\ang{S_{V, 1}(t_{n, \ell}-t_{n, j}) \vec{U}^{j}_{V}(0) \mid \chi_{R} {U}^{\ell}_{V, 1}(0)}_{L^{2}}}  + o_{R}(1) \mas R \to \infty.
\end{align*}
where $o_{R}(1)$ refers to a quantity that goes to $0$ uniformly in $n$ as $R \to \infty$. Now it suffices to show that for any fixed $R > 0$, the two $L^{2}$ inner products on the last line goes to $0$ as $n \to \infty$. Moving $\chi_{R}(r)$ to the first factors in these $L^{2}$ inner products and applying \eqref{eq: o en:pf:VV:2}, the desired conclusion follows as $\abs{t_{n, j} - t_{n, \ell}} \to \infty$. 

\vskip.5em
\noindent \textit{Case $4$: two Euclidean profiles}: We can again assume that the limiting profiles $\vec V^{j}_{\euc}(0)$, $\vec V_{\euc}^\ell(0)$ belong to the class $C^{\infty}_0 \times C^{\infty}_0(\R^d)$ and are compactly supported in, say $\{r \le R\}$. First, assume that 
\EQ{ \label{la 0 t inf}
\frac{ \la_{n, j}}{\la_{n, \ell}} \to 0 \mand \la_{n, j} \abs{ t_{n, j} - t_{n, \ell}} \to c_0  \ge 0 \mas n \to \infty
}
In this case, we can then assume without loss of generality, that $t_{n, j} = t_{n, \ell}$ for all $n \in \N$ since~\eqref{la 0 t inf} implies that $\abs{ t_{n, j} - t_{n, \ell}} \to 0$. We then have 
\begin{align*}
& \hskip-2em
	\abs{\ang {\vec{U}^{j}_{\euc, n, L}(0) \mid \vec{U}^{\ell}_{\euc, n, L}(0)}_{\HH_{V}} } \\
	= & \abs{\ang { \tau_{h_{n, j}^{-1}} \T_{\lmb_{n, j}} \vec{V}^{j}_{\euc}(0) \mid 
					S_{V}(t_{n, j} -t_{n, \ell}) \tau_{h_{n, \ell}^{-1}} \T_{\lmb_{n, \ell}} \vec{V}^{\ell}_{\euc}(0) }_{\HH_{V}} }  \\
	\leq & \abs{\ang { D_{V}^{2} \tau_{h_{n, j}^{-1}} \T_{\lmb_{n, j}}^{0} {V}^{j}_{\euc, 0}(0) \mid 
					 \tau_{h_{n, \ell}^{-1}} \T_{\lmb_{n, \ell}}^{0} {V}^{\ell}_{\euc, 0}(0) }_{L^{2}} }  \\
		& +  \abs{\ang { \tau_{h_{n, \ell} \circ h_{n, j}^{-1}} \T_{\lmb_{n, j}}^{1} {V}^{j}_{\euc, 1}(0) \mid 
					\T_{\lmb_{n, \ell}}^{1} {V}^{\ell}_{\euc, 1}(0) }_{L^{2}} }  \\
	\aleq & \nrm{\T^{0}_{\lmb_{n, j}} {V}^{j}_{\euc, 0}(0)}_{H^{2}(\bbH^{d})} \nrm{\T_{\lmb_{n, \ell}}^{0} {V}^{\ell}_{\euc, 0}(0)}_{L^{2}(\bbH^{d})} \\
		&	+ \nrm{\T_{\lmb_{n, j}}^{1} {V}^{j}_{\euc, 1}(0)}_{L^{\frac{2d}{d-2}}(\bbH^{d})} \nrm{\T_{\lmb_{n, \ell}}^{1} {V}^{\ell}_{\euc, 1}(0)}_{L^{\frac{2d}{d+2}}(\bbH^{d})}
\end{align*}
where the last line satisfies $\aleq \frac{\lmb_{n, j}}{\lmb_{n, \ell}} \to 0$ as $n \to \infty$, by~\eqref{eq:Tlmbn} and the assumption that $\vec{V}^{j}_{\euc}, \vec{V}^{\ell}_{\euc} \in C^{\infty}_{0} \times C^{\infty}_{0}(\bbR^{d})$.
 
Next, assume that either 
\EQ{
&\frac{ \la_{n, j}}{\la_{n, \ell}} \to 0 \mand \la_{n, j} \abs{ t_{n, j} - t_{n, \ell}} \to \infty \mas n \to \infty, \\
& \mor \qquad  \frac{ \la_{n, j}}{\la_{n, \ell}} \to C_0>0 \mand \la_{n, j} \abs{ t_{n, j} - t_{n, \ell}} \to \infty \mas n \to \infty.
}
Here the strategy is to first replace $S_{V}$ by $S_{\hyp}$ using Lemma \ref{lem:pert4conc}, and then to apply the dispersive estimates from Proposition~\ref{prop: disp}. Proceeding as before, we begin by writing
\begin{align}
& \hskip-2em
	\abs{\ang {\vec{U}^{j}_{\euc, n, L}(0) \mid \vec{U}^{\ell}_{\euc, n, L}(0)}_{\HH_{V}} } \notag \\
	\leq & \abs{\ang {  S_{V, 0}(t_{n, \ell} - t_{n, j}) \tau_{h_{n, j}^{-1}} \T_{\lmb_{n, j}} \vec{V}^{j}_{\euc}(0) \mid 
					D_{V}^{2} \tau_{h_{n, \ell}^{-1}} \T_{\lmb_{n, \ell}}^{0} V^{\ell}_{\euc, 0}(0) }_{L^{2}} }  \label{eq: o en:pf:EucEuc:1} \\
		& +  \abs{\ang { S_{V, 1}(t_{n, \ell} - t_{n, j}) \tau_{h_{n, \ell} \circ h_{n, j}^{-1}} \T_{\lmb_{n, j}} \vec{V}^{j}_{\euc}(0) \mid 
					 \T_{\lmb_{n, \ell}}^{1} V^{\ell}_{\euc, 1}(0) }_{L^{2}} }  	\label{eq: o en:pf:EucEuc:2}. 
\end{align}
Using Lemma \ref{lem:pert4conc} with $(p, q, \gmm) = (\infty, \frac{2d}{d-2}, 1)$, we estimate~\eqref{eq: o en:pf:EucEuc:1} by
\begin{align*}
\eqref{eq: o en:pf:EucEuc:1}
\aleq & \nrm{S_{V, 0}(t_{n, \ell} - t_{n, j}) \tau_{h_{n}^{-1}} \T_{\lmb_{n, j}} \vec{V}^{j}_{\euc}(0)}_{L^{\frac{2d}{d-2}}} 
	\nrm{\T_{\lmb_{n, \ell}}^{0} \vec{V}^{\ell}_{\euc, 0}(0)}_{W^{2, \frac{2d}{d+2}}} \\
= & \nrm{S_{\hyp, 0}(t_{n, \ell} - t_{n, j}) \tau_{h_{n}^{-1}} \T_{\lmb_{n, j}} \vec{V}^{j}_{\euc}(0)}_{L^{\frac{2d}{d-2}}} 
	\nrm{\T_{\lmb_{n, \ell}}^{0} V^{\ell}_{\euc, 0}(0)}_{W^{2, \frac{2d}{d+2}}}  + o_{n}(1)
\end{align*}
as $n \to \infty$. Then by Proposition \ref{prop: disp}, where we ignore the long-time $\abs{t}^{-\frac{3}{2}}$-gain and only use the weaker estimate~\eqref{eq:shortDisp} with the $\abs{t}^{-\frac{d-1}{d}}$-gain and $\sgm = \frac{d+1}{d}$, we see that the first term on the last line is bounded by
\begin{align*}
	\aleq & \abs{t_{n, j} - t_{n, \ell}}^{-\frac{d-1}{d}} 
	\nrm{\T_{\lmb_{n, j}} \vec{V}^{j}_{\euc}(0)}_{W^{\frac{d+1}{d}, \frac{2d}{d+2}} \times W^{\frac{1}{d}, \frac{2d}{d+2}}} 
	\nrm{\T_{\lmb_{n, \ell}}^{0} {V}^{\ell}_{\euc, 0}(0)}_{W^{2, \frac{2d}{d+2}}} \\	
	\aleq & (\lmb_{n, j} \abs{t_{n, j} - t_{n, \ell}})^{-\frac{d-1}{d}} \to 0 \mas n \to \infty,
\end{align*}
where we used~\eqref{eq:Tlmbn} on the last line. 
Next, applying Lemma \ref{lem:pert4conc} again with $(p, q, \gmm) =(\infty, \frac{2d}{d-2}, 1)$ and using the assumption that $\vec{V}^{j}_{\euc}(0) \in H^{2} \times H^{1}(\bbR^{d}) \subseteq C^{\infty}_{0} \times C^{\infty}_{0} (\bbR^{d})$, we have
\begin{equation*}
	\eqref{eq: o en:pf:EucEuc:2} =
	 \nrm{S_{\hyp, 1}(t_{n, \ell} - t_{n, j}) \tau_{h_{n}^{-1}} \T_{\lmb_{n, j}} \vec{V}^{j}_{\euc}(0)}_{L^{\frac{2d}{d-2}}} 
	\nrm{\T_{\lmb_{n, \ell}}^{1} V^{\ell}_{\euc, 1}(0)}_{L^{ \frac{2d}{d+2}}}  + o_{n}(1).
\end{equation*}
Then proceeding similarly as before, it follows that the first term on the right-hand side obeys the bound 
\begin{equation*}
	\aleq (\lmb_{n, j} \abs{t_{n, j} - t_{n, \ell}})^{-\frac{d-1}{d}} \frac{\lmb_{n, j}}{\lmb_{n, \ell}} \to 0 \mas n \to \infty.
\end{equation*}
We omit the routine details.

Finally assume that we have 
\ant{
& \frac{\la_{n, j}}{\la_{n, \ell}}  \to C_0>0,  \quad \la_{n, j} \abs{t_{n, j} - t_{n, \ell}} \to C_1  \ge 0 \, \\
& \mand \la_{n, j}\bfd_{\Hp^d}( h_{n, j} \cdot   \zero , h_{n, \ell} \cdot \zero) \to \infty
}
Without loss of generality, we can assume that $\la_{n, j}  = \la_{n, \ell} \to \infty$ and  $t_{n, j}  = t_{n, \ell}$ for every $n \in \N$.  Using these assumptions we have 
\begin{align*}
& \hskip-2em
\abs{\ang{ \vec{U}^{j}_{\euc, n, L}(0) \mid  \vec{U}^{\ell}_{\euc, n, L}(0) }_{\HH_{V}} }   \\
  = & \abs{ \ang{ \tau_{h_{n, \ell} \circ h_{n, j}^{-1}}  \T_{\la_{n, j}} \vec V_{\euc}^j(0) \mid  \T_{\la_{n, \ell}} \vec V^\ell_{\euc}(0) }_{\HH} } \\
	& + \abs{ \ang{\tau_{h_{n, \ell}} V \, \tau_{h_{n, \ell} \circ h_{n, j}^{-1}} \T_{\la_{n, j}}^{0} V_{\euc, 0}^{j}(0) \mid 
					 \T_{\la_{n, \ell}}^{0} V_{\euc, 0}^{\ell}(0)}_{L^{2}} } \\
=& 0 \mfor n \, \hbox{ large enough}. 
\end{align*}
The last line follows since $\lmb_{n, j} \bfd_{\Hp^d}( h_{n, j}^{-1} \circ h_{n, \ell}   \cdot \zero  , \zero) \to \infty$ as $n \to \infty$ and from our assumption we have $\vec V_{\euc}^j(0), \vec{V}^{\ell}_{\euc}(0) \in C^{\infty}_{0} \times C^{\infty}_{0}(\bbR^{d})$. Indeed for $n$ large enough we see that the supports of $ \tau_{h_{n, j}^{-1} \circ h_{n, \ell}}  \T_{\la_{n, j}} \vec V_{\euc}^j(0)$ and $\T_{\la_{n, \ell}}\vec V^\ell_{\euc}(0)$ are disjoint.

The proof of~\eqref{Uj w} is very similar to that of the proof of~\eqref{Uj Uell} which we have done in detail. There are several cases that one must consider, and the only real difference is that in some cases, the weak convergences~\eqref{w weak V},~\eqref{w weak hyp} and~\eqref{w weak euc} have to be used when $\lmb_{n, j} \abs{t_{n, j}}$ remains bounded and hence dispersive estimates do not apply. We omit the details, and refer the reader to e.g.,~\cite[Proof of Lemma 2.16]{CKLS1} for a similar argument. 
  
This completes the proof of Claim~\ref{claim: o en}. 
\end{proof}

We can now complete  the proof of Theorem~\ref{thm: BG}. It only remains to show~\eqref{b norm to 0}. But this is now easy given the orthogonality of the free energies. Indeed,  by Claim~\ref{claim: o en}  we have 
\EQ{
\sum_{j <k} (\nu_j)^2 \lesssim \sum_{j<k} E_{V}(\vec U^j_{n, L})  \le \limsup_{n \to \infty} E_{V}(\vec u_n)
} 
and the above is~\emph{uniform in} $k$. This means that $\nu_J \to 0 $ as $J \to \infty$, which concludes the proof. 
\end{proof}

%
%
%


\part{Nonlinear profiles and applications: Scattering for semilinear wave equations} \label{p:nonlin}

In the second part of this paper, we give a nonlinear application of the linear machinery developed in the first part. For concreteness, we restrict our attention to $d = 3$ and consider the semilinear wave equation~\eqref{u eq}, which we recall to be
\begin{equation*} \tag{\ref{u eq}}
\begin{aligned}
&u_{tt} - \De_{\Hp^3} u  + V u= - \abs{u}^{4} u,\\
& \vec u(0) = (u_0, u_1),
\end{aligned}
\end{equation*}
where the potential $V$ is either $0$, or more generally, assumed to satisfy the assumptions of Theorem~\ref{main:V}.

This part will culminate in the proof of Theorems~\ref{main} and \ref{main:V}, following the scattering blueprint developed by Kenig and Merle in~\cite{KM06, KM08}. This is an elaborate contradiction scheme that can roughly be divided into three steps: $(1)$ a suitable small data theory; $(2)$ a concentration compactness argument ending with the construction of a minimal non-scattering solution with certain compactness properties in the event that global well-posedness and scattering fails; and $(3)$ a rigidity theory which rules out non-zero solutions with compactness properties like the critical element. 

The main tool used in the concentration compactness step is a nonlinear version of the Bahouri-G\'erard profile decomposition, which we develop in Section~\ref{sec: bg nl}. We point out that, unlike the usual nonlinear profile decomposition on the flat space $\bbR^{1+3}$, the nonlinear profile has to be defined using a different nonlinear equation depending on the type of the linear profile. This feature originates from the fact that the action of non-compact groups responsible for the failure of compactness, as quantified in the linear profile decomposition theorem (Theorem \ref{thm: BG}), do \emph{not} leave the equation~\eqref{u eq} invariant. The relevant nonlinear equation arises as a `limiting equation' under the action of such groups, e.g., the potential-less equation for the action of translational group and the flat Euclidean equation for the scaling group.

We remark that nonlinear profile decompositions such as Theorem~\ref{thm:nonlinbg} have proved to be very versatile tools even outside the original Kenig-Merle scheme, see for example the work of Duyckaerts, Kenig, and Merle,~\cite{DKM1,  DKM3, DKM2,  DKM4} on the Euclidean focusing energy critical semi-linear wave equation, as well as far more complicated derivative nonlinearities such as for wave maps in the work of Krieger and Schlag~\cite{KS}. 


 \section{Local Cauchy theory}\label{sec: wp}

We  formulate the local Cauchy theory as well as the small data theory for~\eqref{u eq}.
For a time interval $0 \ni I \subseteq \R$ , define the norms $S(I)$ and $N(I)$ by
\EQ{
&\|u\|_{S(I)} := \|u\|_{L^5_t(I; L^{10} (\Hp^3))}\\
&\|F\|_{N(I)} := \|F\|_{L^1_t(I; L^{2} (\Hp^3))}
}
In this section, we only require that the potential $V$ satisfies the assumptions \eqref{eq:Vdecay}--\eqref{eq:Str4halfVWave:2}, which are more relaxed than those in Theorem~\ref{main:V}. 

 \begin{prop}[Local Cauchy theory] \label{small data}
Let $\vec u(0) = (u_0, u_1) \in \HH (\Hp^3)$. Then there is a unique  solution $\vec u(t) \in \HH$ to \eqref{u eq} defined on a maximal interval of existence $0 \in I_{\max}( \vec u) = (-T_-, T_+)$.  Moreover, for any compact interval $J \subset I_{\max}$ we have
\ant{
\| u\|_{S(J)} < \infty.
}
Moreover, a globally defined solution $\vec u(t)$ for $t\in [0, \infty)$ scatters as $ t \to \infty$ to a \emph{free} wave, i.e., a solution $\vec u_{L}(t) \in \HH$ of
\ant{
\Box_{\bbH^{3}} u_{L} = \rd_{t}^{2} u_{L} - \lap_{\Hp^3} u_{L} =0
}
if and only if  $ \|u \|_{S([0, \infty))}< \infty$. Here scattering means that
\EQ{
\|\vec u(t) - \vec u_{L}(t) \|_{\HH} \to 0 \mas t \to \infty.
}
In particular, there exists a constant $\de>0$ so that
\EQ{ \label{global small}
 \| \vec u(0) \|_{\HH} < \de \Rightarrow  \| u\|_{S(\R)} \lesssim \|\vec u(0) \|_{\HH} \lesssim \de
 }
and hence $\vec u(t)$  scatters to free waves as $t \to \pm \infty$. Finally, we have the standard finite time blow-up criterion: 
\EQ{ \label{ftbu}
T_+(\vec u)< \infty \Longrightarrow  \| u \|_{S([0, T_+(\vec u)))}  = + \infty
}
An analogous statement holds in the negative time direction. 
\end{prop}
\begin{proof}
 The proof of Proposition~\ref{small data} follows from the usual contraction mapping argument based on the Strichartz estimates in Lemma~\ref{lem:basicVWave} part (\ref{item:basicVWave:2}) (see Proposition~\ref{strich} for the free hyperbolic evolution). Indeed applying Lemma~\ref{lem:basicVWave} part (\ref{item:basicVWave:2}) with $(p, q, \gamma) = (5, 10, 1)$ to any time interval $I$ we have
\begin{align}
\| u \|_{S(I)} + \| \vec u(t)\|_{L^{\infty}(I; \HH)} &\lesssim \|\vec u (0) \|_{\HH} + \| u^5 \|_{N(I)}\nonumber\\
& \lesssim \|\vec u (0) \|_{\HH} + \| u \|_{S(I)}^5 \nonumber
\end{align}
By the usual continuity argument, (expanding $I$) this implies the a priori estimate~\eqref{global small} for small data. 

To prove the scattering statement, assume that $\vec{u}(t)$ exists on $[0, \infty)$ and satisfies $\nrm{u}_{S([0, \infty))} < \infty$. We first prove that $\vec{u}(t)$ scatters to a solution $\vec{u}_{V, L}(t)$ to the perturbed equation $(\Box_{\bbH^{3}} + V) u_{V, L} = 0$ as $t \to \infty$. That is, we seek initial data $\vec u_{V, L}(0) \in \HH$ so that
\ant{
\nrm{\vec u(t) - \vec u_{V, L}(t)}_{\HH_{V}} \to 0 \mas t \to \infty,
}
where $\vec u_{V, L}(t) = S_{V}(t) \vec{u}_{V, L}(0)$. Note that $\HH_{V}$ is equivalent to $\HH$; we have chosen to use $\HH_{V}$ here, as $S_{V}(t)$ forms a unitary group on $\HH_{V}$. In view of the Duhamel representation for $\vec u(t)$, such a $\vec{u}_{V, L}(0)$ is given by the formula
\ant{
\vec u_{V, L}(0) =  \vec u(0) + \int_0^{\infty} S_{V}(-s)(0, -u^5(s)) \, ds,
}
where the integral on the right-hand side above is absolutely convergent in $\HH_{V}$ as long as $\|u\|_{S([0, \infty))}< \infty$.

Next, we prove that $\vec{u}_{V, L}(t)$ scatters to a solution $\vec{u}_{L}(t)$ to the free wave equation $\Box_{\bbH^{3}} u_{L} = 0$. That is, we claim that there exists a free wave $\vec{u}_{L}(t)$ such that
\begin{equation} \label{eq:small-data:lin-scat}
	\nrm{\vec{u}_{V, L}(t) - \vec{u}_{L}}_{\HH} \to 0 \mas t \to \infty.
\end{equation}
By approximation, it suffices to consider the case $\vec{u}_{V, L}(0) \in C^{\infty}_{0} \times C^{\infty}_{0}(\bbH^{3})$.
Let $\vec{v}_{L}(t) := S_{\hyp}(t)\vec{u}_{V, L}(0)$, where we remind the reader that $S_{\hyp}(t)$ is the propagator for the free wave equation. As $\vec{v}_{L}(t)$ solves $(\Box_{\bbH^{3}} + V) v_{L} = V v_{L}$, we have the following Duhamel representation:
\begin{equation*}
	S_{\hyp}(t) \vec{u}_{V, L}(0) = \vec{u}_{V, L}(t) + \int_{0}^{t} S_{V}(t-s)(0, V v_{L}(s)) \, \ud s.
\end{equation*}
Hence we see that \eqref{eq:small-data:lin-scat} holds for $\vec{u}_{L}(t) = S_{\hyp}(t) u_{L}(0)$ with
\begin{equation*}
	\vec{u}_{L}(0) = \vec{u}_{V, L}(0) - \int_{0}^{t} S_{\hyp}(- t) S_{V}(t-s) (0, V v_{L}(0)) \, \ud s,
\end{equation*}
once we prove that the $s$-integral on the right-hand side is strongly convergent in $\HH$ as $t \to \infty$. By unitarity of $S_{\hyp}$ on $\HH$, the equivalence $\nrm{\cdot}_{\HH} \simeq \nrm{\cdot}_{\HH_{V}}$ (see Lemma~\ref{lem:basicVWave}) and unitarity of $S_{V}$ on $\HH_{V}$, it is then sufficient to establish
\begin{equation} \label{eq:small-data:cook-criterion}
	\int_{0}^{\infty} \nrm{V v_{L}(s)}_{L^{2}} \, \ud s < \infty.
\end{equation}
Since $v_{L}(s) = S_{\hyp, 0}(s) \vec{u}_{V, L}(0)$ and $\vec{u}_{V, L}(0) \in C^{\infty}_{0} \times C^{\infty}_{0}(\bbH^{3})$, \eqref{eq:small-data:cook-criterion} is an easy consequence of the dispersive estimate for the free wave, Proposition~\ref{prop: disp}.

Finally, to show that finiteness of the $S$-norm is a necessary condition for scattering, assume that a solution $\vec{u}(t)$ defined on $[0, \infty)$ scatters to a free wave $\vec{u}_{L}(t)$ in $\HH$ as $t \to \infty$. Then by a similar argument as in the proof of \eqref{eq:small-data:lin-scat}, there exists a solution $\vec{u}_{V, L}(t)$ to $(\Box_{\bbH^{3}} + V) u_{V, L} = 0$ such that $\vec{u}(t) - \vec{u}_{V, L}(t) \to 0$ in $\HH$ as $t \to \infty$. Since the $S$-norm of $u_{V, L}$ is finite, by the small data theory (applied to large times) this carries over to $\vec{u}_{L}$, which proves $\nrm{u}_{S([0, \infty))} < \infty$ as desired. \qedhere 
\end{proof}

 Next, we prove a perturbation lemma which will be a key ingredient in the Euclidean approximations at small scales, and the Kenig-Merle concentration compactness argument. This type of result is also standard, see \cite{KM06, KM08, LS}.
\begin{lem}[Perturbation Lemma] \label{lem: pert}
There are continuous functions \\$\e_0,C_0:(0,\I)\to(0,\I)$ such that the following holds:
Let $I \subseteq \R$ be an open interval (possibly unbounded), $u,v\in C(I; H^{1}(\Hp^3))\cap C^{1}(I;L^{2}(\Hp^3))$ functions satisfying for some $A>0$
\begin{align}
  \|\vec v\|_{L^\infty_{t}(I;\HH)} +   \|v\|_{S(I)} & \le A < \infty \label{v e s}\\
 \|\glei(u)\|_{L^1_t(I;L^2 (\bbH^{3}))}
   + \|\glei(v)\|_{L^1_t(I;L^2 (\bbH^{3}))} + \|w_0\|_{S(I)} &\le \e \le \e_0(A), \label{glei}
   \end{align}
where $\glei(u):=\rd_{t}^{2} u - \lap_{\bbH^d} u+ V u + u^{5}$ in the sense of distributions, and $\vec w_0(t):=S_{V}(t-t_0)(\vec u-\vec v)(t_0)$ with $t_0\in I$ arbitrary but fixed.  Then
\ant{
  \|\vec u-\vec v-\vec w_0\|_{L^{\infty}_t(I;\HH)}+\|u-v\|_{S(I)} \le C_0(A)\e.
  }
  In particular,  $\|u\|_{S(I)}<\I$.
\end{lem}

\begin{proof}
Let $S:=L^5_tL^{10}_x$ as before and
\ant{
 w:=u-v, \pq e:=\Box_{\bbH^{3}} (u-v)+u^5-v^5 = \glei(u) - \glei(v).}
There is a partition of the right half of $I$ as follows, where $\delta_{0}>0$ is a small
absolute constant which will be determined below:
\ant{
 \pt t_0<t_1<\cdots<t_n\le \I,\pq I_j=(t_j,t_{j+1}),\pq I\cap(t_0,\I)=(t_0,t_n),
 \pr \|v\|_{S(I_j)} \le \de_{0} \pq(j=0,\dots,n-1), \pq n\le C(A,\de_{0}).}
We omit the estimate on $I\cap(-\I,t_0)$ since it is the same by symmetry.
Let $\vec w_j(t):=S_{V}(t-t_j)\vec w(t_j)$ for all $0\le j <n$.  Then
\EQ{\label{w form}
\vec w(t)
&= \vec w_{0}(t) + \int_{t_0}^{t} S_{V}(t-s) (0,e-(v+w)^{5} + v^{5})(s)\, ds
}
which implies that, for some absolute constant $C_{1}\ge1$,
\EQ{ \label{eq:ww0}
\pn \| w-w_{0}\|_{S(I_0)}
 \pt\lec  \|(v+ w)^5- v^5 -e\|_{L^1_tL^2_x(I_0)}
 \pr\le C_{1} (\de_{0}^{4}+\| w\|_{S(I_0)}^{4})\| w\|_{S(I_0)}+C_{1}\e
}
Note that $\|w\|_{S(I_{0})}<\I$ provided $I_{0}$ is a finite interval. If $I_{0}$ is half-infinite, then we first
need to replace it with an interval of the form $[t_{0},N)$, and  let $N\to\I$ after performing the estimates which are
uniform in~$N$.  Now assume that $C_{1}\delta_{0}^{4}\le \frac14$ and fix $\delta_{0}$ in this fashion.
By means of the continuity method (which refers to using that the $S$-norm is continuous in  the upper endpoint of $I_{0}$),
\eqref{eq:ww0} implies that $\| w\|_{S(I_{0})}\le 8C_{1}\e$.
Furthermore, Duhamel's formula implies  that
\ant{
\vec w_{1}(t)- \vec w_{0}(t) = \int_{t_{0}}^{t_{1}} S_{V}(t-s) (0,e-(v+w)^{5}  + v^{5})(s)\, ds
}
whence also
\EQ{\label{eq:w1w0}
\| w_{1}-w_{0}\|_{S(\R)} \lec \int_{t_0}^{t_1} \| (e-(v+w)^{5}  + v^{5})(s)\|_{L^2}\, ds
}
which is estimated as in~\eqref{eq:ww0}.  We conclude that $ \| w_{1}\|_{S(\R)}\le 8C_{1}\e$.
In a similar fashion one verifies that for all $0\le j<n$
\EQ{ \label{est S'}
 \pn\| w- w_j\|_{S(I_j)} + \| w_{j+1}-w_j\|_{S(\R)}
 \pt\lec  \|  e-(v+w)^{5} + v^{5}  \|_{L^1_tL^2_x(I_j)}
 \pr\le C_{1} (\de_{0}^{4}+\| w\|_{S(I_j)}^{4})\| w\|_{S(I_j)}+C_{1}\e}
 where $C_{1}\ge1$ is as above.
By induction in $j$ one obtains that
 \ant{
 \| w\|_{S(I_{j})} + \| w_{j}\|_{S(\R)} \le C(j)\, \e\quad \forall \; 1\le j<n
 }
 This requires that  $\e<\e_{0}(n)$, which can be done provided $\e_0(A)$ is chosen small enough.
Repeating the estimate~\eqref{est S'} once more,  but with the energy piece $L^{\I}_{t}\HH$ included on the left-hand side,
finishes the proof.
\end{proof}

It will be convenient to state the following small data result as a simple consequence of the Perturbation Lemma~\ref{lem: pert}. 

\begin{cor}\label{cor: lin pert}
There exists a function  $\e_1:(0, \infty) \to (0, \infty)$ with the following property. Let $(f, g) \in \HH(\Hp^3)$ and let $I \subseteq \R$ be an open time interval (possibly unbounded). Let $t_0 \in I$ and let  $\vec u_{V, L}(t) := S_{V}(t- t_0) \vec u(t_0)$
be the linear evolution of the data $\vec u_{V, L}(t_0):= (f, g)$. Suppose that 
\EQ{
\| u_{V, L} \|_{S(I)} \le \e \le \e_0( \| (f, g)\|_{\HH}).
}
Denote by $\vec u(t)$ the solution to~\eqref{u eq} with initial data $\vec u(t_0) = (f, g)$. Then $I \subseteq I_{\max}( \vec u) $ and we have
\EQ{
&\| \vec u - \vec u_{V, L} \|_{L^{\infty}(I ;  \HH)} + \| u- u_{V, L}\|_{ S(I)}  \lesssim C_0(\e) \e\\
&\| u\|_{S(I)}  \lesssim  \e
}
where the implicit constant above depends only on $\|(f, g)\|_{\HH}$. 
\end{cor}

\begin{proof}
This is a simple consequence of the Pertubation Lemma~\ref{lem: pert}. Note that 
\ant{
 \glei( u_{V, L}) = u_{V, L}^5
 }
 and hence we have 
 \EQ{
 \|\glei( u_{V, L})\|_{L^1(I; L^2_x(\Hp^3))} = \|u_{V, L}^5\|_{L^1(I; L^2_x(\Hp^3))} \lesssim  \| u_{V, L}\|_{S(I)}^5 \lesssim \e^5.
 }
 Now apply Lemma~\ref{lem: pert} to $v= u_{V, L}$ and $u = u$.  \qedhere
\end{proof}
%

%



\section{Nonlinear approximation for traveling and concentrating profiles}\label{sec:nonlinApprox}
In this section we continue the theme we began in Section~\ref{subsec:linApprox} and develop tools for approximating solutions to the nonlinear equation \eqref{u eq} whose initial data are either traveling or concentrating profiles. The results of this section will be used in the next section to properly formulate the notion of nonlinear profiles for \eqref{u eq}; see Definition~\ref{def:nonlin}. The analogous theory for  the nonlinear Schr\"odinger equation on $\Hp^3$ (with $V = 0$) was established  in \cite[Section $4$]{IPS}, and the statements in this section closely resemble the ones given there. 

As in the previous section, $V$ is only assumed to obey \eqref{eq:Vdecay}--\eqref{eq:Str4halfVWave:2}.

\subsection{Nonlinear approximation for traveling profiles}
Here, we show that solutions to \eqref{u eq} with traveling profiles as initial data are well-approximated by suitable translations of a solution to the potential-less equation
\begin{equation} \label{eq:noPtnlUEq}
	\rd_{t}^{2} U - \lap_{\bbH^{3}} U = \abs{U}^{4} U.
\end{equation}
The associated conserved energy is given by
\begin{equation*}
	\calE_{\hyp}[\vec{U}] := \int_{\bbH^{3}} \frac{1}{2} (\abs{\nb U_{0}}^{2} + \abs{U_{1}}^{2}) + \frac{1}{6} \abs{U}^{6} \, \ud x.
\end{equation*}

The main result is as follows.
\begin{prop} \label{prop: h h}
Given initial data $\vec{U}_{\hyp}(0) = (f, g) \in \calH(\bbH^{3})$ and a sequence $\set{h_{n}} \subseteq \bbG$ such that $\abs{h_{n}} \to \infty$ as $n \to \infty$, consider the following three objects:
\begin{itemize}
\item Let $\vec{U}_{\hyp}(t) \in \calH(\bbH^{3})$ be the nonlinear potential-free evolution of the data, i.e., $\vec{U}_{\hyp}(t)$ is the solution to~\eqref{eq:noPtnlUEq} with data $(f, g)$, defined on the maximal interval of existence $I_{\max} = (T_{-}, T_{+})$.
\item Applying $\tau_{h_{n}^{-1}}$ to $\vec{U}_{\hyp}$, define the sequence
\begin{equation*}
	\vec{U}_{n, \hyp} = \tau_{h_{n}^{-1}} \vec{U}_{\hyp} \in \HH(\bbH^{3}) \mfor t \in \bbR.
\end{equation*}
\item Let $\vec u_n(t) \in \HH(\Hp^3)$ be the nonlinear hyperbolic evolution of the initial data $\vec u_n(0):= \tau_{h_{n}^{-1}} (f, g)$, i.e, $\vec u_n(t)$ is the solution to~\eqref{u eq} with initial data $\vec u_n(0)$, defined on  the maximal interval of existence $I_{n, \max} = (T_{n, -}, T_{n, +})$.
\end{itemize}
 Then the following conclusions are true: 
\begin{enumerate}
\item \label{item: h h:1}
	Let $I \subseteq (T_{-}, T_{+})$ be an arbitrary \emph{finite} interval such that $I \neq (T_{-}, T_{+})$. Then there exists a positive integer $N = N(I, f, g)$ such that for all $n > N$, we have $I \subseteq I_{n, \max}$ and 
\EQ{
\| \vec{u}_n - \vec{U}_{n, \hyp}\|_{L^{\infty}_t(I; \HH)} +  \| u_n - U_{n, \hyp}\|_{S(I)}  = o_n(1) \mas n \to \infty
}
where the implicit constant in the $o_n(1)$ term depends only on the energy $\E_{\hyp}(f, g)$. 
\item \label{item: h h:2}
	Suppose that $T_{+} = +\infty$ and $\nrm{U_{\hyp}}_{S([0, \infty))} < \infty$. Then for $n$ large enough we have $[0, \infty) \subseteq I_{n, \max}$ and 
\EQ{
\| u_n\|_{S([0, \infty))} \le C( \nrm{U_{\hyp}}_{S([0 ,\infty))} ) .
}
	An analogous statement holds in the negative time direction as well.
 \item \label{item: h h:3}
 	As in (\ref{item: h h:2}), suppose that $T_{+} = +\infty$ and $\nrm{U_{\hyp}}_{S([0, \infty))} < \infty$. Denote by $\vec U_{\infty}$ the free scattering data for the nonlinear evolution $\vec U_{\hyp}(t)$, i.e, 
 \EQ{ \label{eq: h h:3:Uinfty}
 \| \vec U_{\hyp}(t) - S_{\hyp}(t) \vec U_{\infty}\|_{\HH} \to 0 \mas  t \to \infty.
 }
 Then for every $\e>0$ there exist $T, N>0$ such that we have
 \EQ{ \label{eq: h h:3}
\sup_{n \geq N} \| \vec{u}_{n}(t) - S_{V}(t) \tau_{h_{n}^{-1}} \vec U_{\infty}\|_{L^{\infty}_{t}([T, \infty); \HH )}  \le \e.
  }
  	An analogous statement holds in the negative time direction as well.
\end{enumerate}
\end{prop}

\begin{proof} 
We begin by proving part (\ref{item: h h:1}). By approximation, we may assume that $(f, g) \in C^{\infty}_{0} \times C^{\infty}_{0}(\bbH^{3})$. We apply Lemma \ref{lem: pert} (perturbation lemma) with $t_{0} = 0$, $\vec{u}(t) = u_{n}(t)$ and $\vec{v}(t) = \vec{U}_{n, \hyp}(t)$. Since $\vec{u}_{n}$ is already a solution to~\eqref{u eq}, part (\ref{item: h h:1}) would follow once we establish
\begin{equation} \label{eq: h h:1:pf:1}
	\nrm{\glei(U_{n, \hyp})}_{L^{1}_{t} (I; L^{2}(\bbH^{3}))} \to 0 \mas n \to \infty.
\end{equation}
Since $\vec{U}_{\hyp}$ solves the potential-free equation $\Box_{\bbH^{3}} U_{\hyp} + U_{\hyp}^{5} = 0$, we have
\begin{equation*}
\glei(U_{n, \hyp}) = \Box_{V} U_{n ,\hyp} + U_{n, \hyp}^{5} = V \tau_{h_{n}^{-1}} U_{\hyp}(t).
\end{equation*}
As $I$ is finite and $(f, g) \in C^{\infty}_{0} \times C^{\infty}_{0} (\bbH^{3})$, by the finite speed of propagation, it follows that $U_{\hyp}(t)$ is supported in a common compact set $K \subseteq \bbH^{3}$ for all $t \in I$. Using the assumption~\eqref{eq:Vdecay} concerning exponential decay of $V$ and the fact that $\abs{h_{n}} \to \infty$, we have 
\begin{equation*}
	\sup_{t \in I} \nrm{V \tau_{h_{n}^{-1}} U_{\hyp}(t)}_{L^{2}(\bbH^{3})} \to 0 \mas n \to \infty.
\end{equation*}
Integrating in $t$, the desired statement~\eqref{eq: h h:1:pf:1} follows.

Next, we turn to the proof of part (\ref{item: h h:2}). Without loss of generality, we focus on the forward time direction $\set{ t \geq 0}$; the backward time direction $\set{t \leq 0}$ can be handled similarly. Thanks to part (\ref{item: h h:1}), for any fixed $T > 0$, we have
\begin{equation} \label{eq: h h:2:pf:1}
	\nrm{\vec{u}_{n}(T) - \tau_{h_{n}^{-1}} \vec{U}_{\hyp}(T)}_{\HH}
	+ \nrm{u_{n} - \tau_{h_{n}}^{-1} U_{\hyp}}_{S([0, T))} = o_{n}(1) \mas n \to \infty.
\end{equation}
The preceding estimate is good for any finite interval $[0, T)$; hence part (\ref{item: h h:2}) in the positive time direction would follow once we show the following statement:
There exists $T, N > 0$ such that 
\begin{equation} \label{eq: h h:2:pf:2}
	\sup_{n \geq N} \nrm{u_{n}}_{S([T, \infty))} \leq 1.
\end{equation}
To prove~\eqref{eq: h h:2:pf:2}, we begin by noting that, given any $\eps > 0$, there exists a $T > 0$ such that $\nrm{U_{\hyp}}_{S([T, \infty))} < \eps$; this statement is a simple consequence of the fact that $\nrm{U_{\hyp}}_{S([0, \infty))} < \infty$ and the fact that the $S$-norm involves an integration in $t$. Then by the Strichartz estimate and Duhamel's formula
\begin{equation*}
	\vec{U}_{\hyp}(t) = S_{\hyp}(t - T) \vec{U}_{\hyp}(T) + \int_{T}^{t} S_{\hyp}(t-s) (0, U_{\hyp}^{5}(s)) \, \ud s
\end{equation*}
we have, for $\eps > 0$ sufficiently small, 
\begin{equation*}
	\nrm{S_{\hyp, 0}(t -T) \vec{U}_{\hyp}(T)}_{S([T, \infty))} \aleq \eps.
\end{equation*}
By Corollary \ref{cor: h h:lin:smallScat}, for sufficiently large $n$ we have
\begin{equation} \label{eq: h h:2:pf:3}
	\nrm{S_{V, 0}(t -T) \tau_{h_{n}^{-1}} \vec{U}_{\hyp}(T)}_{S([T, \infty))} \aleq \eps.
\end{equation}
We now apply Lemma \ref{lem: pert} (perturbation lemma) with $t_{0} = T$, $\vec{u}(t) = \vec{u}_{n}(t)$ and $\vec{v}(t) = S_{V}(t - T) \tau_{h_{n}^{-1}} \vec{U}_{\hyp}(T)$. Thanks to~\eqref{eq: h h:2:pf:1}, the data at $t_{0}$ get arbitrarily close to each other in $\HH$ as $n \to \infty$; moreover, by~\eqref{eq: h h:2:pf:3}, the $L^{1}_{t} ([T, \infty); L^{2}(\bbH^{3}))$ norm of $\glei(v)$ can be made arbitrarily small if we take $\eps > 0$ to be small enough. Therefore, for sufficiently large $n$ and small enough $\eps > 0$, we have
\begin{equation} \label{eq: h h:2:pf:4}
\begin{aligned}
& \hskip-2em
	\nrm{\vec{u}_{n}(t) - S_{V}(t - T) \tau_{h_{n}^{-1}} \vec{U}_{\hyp}(T)}_{L^{\infty}_{t} ([T, \infty); \HH)}  \\
&	+ \nrm{u_{n}(t) - S_{V, 0}(t-T) \tau_{h_{n}^{-1}} \vec{U}_{\hyp}(T)}_{S([T, \infty))} \aleq \eps.	
\end{aligned}
\end{equation}
The desired conclusion~\eqref{eq: h h:2:pf:2} now follows from~\eqref{eq: h h:2:pf:3} and~\eqref{eq: h h:2:pf:4}, once we choose $\eps > 0$ as necessary.

Finally, we prove part (\ref{item: h h:3}). As in the proof of part (\ref{item: h h:2}), we focus only on the positive time direction. Recalling the argument leading to~\eqref{eq: h h:2:pf:4}, given an $\eps > 0$, there exist $T = T(\eps, \vec{U}_{\hyp}), N(\eps, \vec{U}_{\hyp}) > 0$ such that
\begin{equation} \label{eq: h h:3:pf:1}
	\sup_{n \geq N} \nrm{\vec{u}_{n}(t) - S_{V}(t-T) \tau_{h_{n}^{-1}} \vec{U}_{\hyp}(T)}_{L^{\infty}_{t} ([T, \infty); \HH)} < \frac{\eps}{2}.
\end{equation}
We claim that, for possibly larger $T, N > 0$, we have
\begin{equation} \label{eq: h h:3:pf:2}
	\sup_{n \geq N} \nrm{\tau_{h_{n}}^{-1}  \vec{U}_{\hyp}(T) - S_{V}(T) \tau_{h_{n}^{-1}} \vec{U}_{\infty} }_{\HH} < \frac{\eps}{2},
\end{equation}
where $\vec{U}_{\infty}$ is defined as in~\eqref{eq: h h:3:Uinfty}. Then part (\ref{item: h h:3}) in the positive time direction would immediately follow.

It only remains to establish~\eqref{eq: h h:3:pf:2}. By the definition of $\vec{U}_{\infty}$, we have
\begin{equation*}
	\nrm{\vec{U}_{\hyp}(T) - S_{\hyp}(T) \vec{U}_{\infty}}_{\HH} = o_{T}(1) \mas T \to \infty.
\end{equation*}
On the other hand, by Lemma \ref{lem:pert4traveling}, we have
\begin{equation*}
	\nrm{S_{V}(T) \tau_{h_{n}^{-1}} \vec{U}_{\infty} - \tau_{h_{n}^{-1}} S_{\hyp}(T) \vec{U}_{\infty}}_{\HH} = o_{n}(1) \mas n \to \infty.
\end{equation*}
Combining the preceding two statements,~\eqref{eq: h h:3:pf:2} follows. \qedhere
\end{proof}

\subsection{Euclidean theory}
The underlying scale invariant nonlinear equation is the defocusing energy critical semi-linear wave equation on $\R^{1+3}$, namely:
\EQ{ \label{v eq}
&\Box_{\bbR^{3}} v=-v^5, \\
& \vec v(0)=(v_0, v_1).
}

In the rest of this section, we will prove that for concentrating initial data, the solution to \eqref{v eq} is a good approximation to the corresponding solution to the hyperbolic equation~\eqref{u eq} (in a suitable sense). In this subsection, we review some of the theory for \eqref{v eq} that will be necessary in what follows. 

%

The conserved energy for \eqref{v eq} is given by
\EQ{
\E_{\euc}( \vec v)(t) =  \int_{\R^3} \left[\frac{1}{2}(v_t^2 + \abs{\nabla v}^2 ) + \frac{1}{6} v^6 \right] \, dx = \textrm{constant}.
}
The Cauchy problem~\eqref{v eq} is referred to as  {\em energy critical} since the energy has the same scaling as the equation. Namely, if $\vec v(t)$ is a solution then so is $\vec v_\la(t)$, which is defined as
\EQ{\label{scaling}
\vec v_\la(t, x) = ( \la^{\frac{1}{2}} v( \la t, \la x), \la^{\frac{3}{2}} v_t(  \la t,  \la x))
}
and we have $$\E_{\euc}(\vec v_\la)(t) = \E_{\euc}(\vec v)( \la t) = \E_{\euc}(\vec v)(0).$$
We also define the free energy space $\HH_{\euc}:= \dot{H}^1 \times L^2 (\R^3)$.

The Cauchy problem for~\eqref{v eq} has been extensively studied and it is known that for any finite energy data there is a unique globally defined solution scattering to free waves as time goes to $\pm \infty$. Many authors contributed to the understanding of this equation and we refer the reader to \cite{SS93, SS94, BS, BG, KM08}, and the references therein for more information. 

\begin{thm}\cite{SS93, SS94, BS, BG, KM08}\label{thm: euc}
Given $(v_0, v_1)\in\HH_\euc:= \dot{H}^1 \times L^2 (\R^3)$ there exists a unique global solution $ v\in C(\R;\dot{H}^1(\R^3))\cap C^1(\R;L^2(\R^3))$ to~\eqref{v eq}.
Moreover, $v$ satisfies
\begin{align*}
&\|v\|_{L_t^5(\R; L_x^{10}(\R^3))}\le A(\E_{\euc}(\vec v) )
\end{align*}
where $A:[0, \infty) \to [0, \infty)$ is a non-decreasing function of the conserved energy $\E_{\euc}(\vec v)$ associated to $(v_0, v_1)$. This means that $ \vec v(t)$  scatters to free waves as $t\rightarrow\pm\infty$, i.e., there exist solutions, $\vec v_{L}^{\pm}(t) \in \HH_e$ to the free wave equation,
\ant{
\Box v_L^{\pm} = 0,
}
so that
\ant{
\| \vec v(t) - \vec v_{L}^{\pm}(t) \|_{\HH_\euc} \to 0 \mas t \to \pm \infty.
}
Moreover, additional regularity is preserved. In particular, if $s \ge 1$ and if the initial data $\vec v(0) \in \dot{H}^s \times \dot{H}^{s-1}$, then the solution $\vec v(t) \in \dot{H}^s \times \dot{H}^{s-1}$ for all $t \in \R$ and we have the estimate
\EQ{\label{high s}
 \sup_{t \in \R}  \| \vec v(t)\|_{ \dot{H}^s \times \dot{H}^{s-1}(\R^3)} \lesssim \|\vec v(0)\|_{ \dot{H}^s \times \dot{H}^{s-1}(\R^3)}
 }
\end{thm}
\begin{rem} As mentioned above, the above result was the culmination of many years of work by several authors. For a precise statement of the first part of the theorem, involving the boundedness of the global Strichartz norm, we refer the reader to~\cite[Corollary $2$]{BG}. The scattering statement  is a standard consequence of this fact. 

The preservation of regularity is also a consequence of the global finiteness of the Strichartz norms. Indeed, one can show using Stichartz estimates and interpolation that in addition to finiteness of $L^5_tL^{10}_x(\R)$, every solution, $\vec v(t)$, also has finite $L^4_tL^{12}_x$ norm, and in fact,  
\ant{
\| v\|_{L^4_tL^{12}_x(\R^{1+3})} \lesssim 1, 
}  
where the constant above depends only on the conserved energy $\E_{\euc}(\vec v)$. To establish the preservation of regularity in, say the case  $s=2$, choose a partition, $0 < T_1<T_2 < \dots T_N < \infty$ of   $[0, \infty)$, so that 
\ant{
\| v\|_{L^4_t([T_j, T_{j+1}]; L^{12}_x(\R^3))} \le \e \quad \forall \, 1 \le j \le N
}
where $\e>0$ will be chosen below. Differentiating~\eqref{v eq} and applying Strichartz estimates on the interval $[0, T_1]$ one can deduce the estimates 
\ant{
\sup_{t \in [0, T_1]}\| \na \vec v(t)\|_{\dot{H}^1 \times L^2} &\lesssim \| ( \na v_0, \na v_1)\|_{\dot{H}^1 \times L^2} +  \| \na v v^4\|_{L^1_t([0, T_1]; L^2_x)}\\
&  \lesssim \| ( \na v_0, \na v_1)\|_{\dot{H}^1 \times L^2} + \| \na v\|_{L^{\infty}([0, T_1]; L^6)} \| v\|_{L^4([0, T_1]; L^{12})}^4
}
Now, choosing $\e>0$ small enough, the last term on the right-hand side can be absorbed on the left, which yields control of the $\dot{H}^2 \times \dot{H}^1$ norm on the interval $[0, T_1]$.  Iterating this procedure then yields the  global-in-time estimate,~\eqref{high s}. 
\end{rem}

We also will require the nonlinear Euclidean perturbation theory proved in~\cite{KM08}, namely: 
\begin{lem}[Perturbation Lemma]\cite{KM08}\label{lem: euc pert}
There are continuous functions \\$\e_0,C_0:(0,\I)\to(0,\I)$ such that the following holds:
Let $I \subseteq \R$ be an open interval (possibly unbounded), $u,v\in C(I; H^{1}(\R^3))\cap C^{1}(I;L^{2}(\R^3))$ functions satisfying for some $A>0$
\begin{align*}
  \|\vec v\|_{L^\infty(I;\HH_{\euc})} +   \|v\|_{S(I \times \R^3)} & \le A < \infty \\
 \|\glei(u)\|_{L^1_t(I;L^2_x(\R^3))}
   + \|\glei(v)\|_{L^1_t(I;L^2_x)(\R^3)} + \|w_0\|_{S(I \times \R^3)} &\le \e \le \e_0(A), 
   \end{align*}
where $\glei(u):=\Box_{\R^3} u+\abs{u}^{p} u$ in the sense of distributions, and $\vec w_0(t):=S_{\euc}(t-t_0)(\vec u-\vec v)(t_0)$ with $t_0\in I$ arbitrary but fixed.  Then
\ant{
  \|\vec u-\vec v-\vec w_0\|_{L^{\infty}_t(I;\HH_{\euc})}+\|u-v\|_{S(I \times \R^3)} \le C_0(A)\e.
  }
  In particular,  $\|u\|_{S(I \times \R^3)}<\I$.

\end{lem}


\subsection{Nonlinear Euclidean approximation for concentrating profiles}
Here we establish the following approximation lemma, which is an analogue of Proposition~\ref{prop: h h} in the case of concentrating profiles. 
\begin{prop} \label{prop: e h} 
Given an initial data set $\vec{v}(0) = (f,g) \in \HH_{\euc}(\bbR^{3})$ and sequences $\set{\lmb_{n}} \subseteq [1, \infty)$, $\set{h_{n}} \subseteq \bbG$ such that $\lmb_{n} \to \infty$ as $n \to \infty$, consider the following three objects:
\begin{itemize}
\item Let $\vec v(t) \in \HH_{\euc}(\R^3)$ be the Euclidean evolution of this data, i.e, $\vec v(t)$ is the solution to~\eqref{v eq} with data $(f, g)$ given by Theorem~\ref{thm: euc}, which is defined for all $t \in \bbR$. 
\item Applying $\tau_{h_{n}^{-1}} \T_{\lmb_{n}}$ to $\vec{v}$ and rescaling $t$, define the sequence
\EQ{
\vec v_n(t) := \tau_{h_{n}^{-1}} \T_{\la_n} \vec v(\la_n t) \in \HH(\Hp^3) \mfor t \in \R. 
}
\item Let $\vec{u}_{n}(t) \in \HH(\bbH^{3})$ be the nonlinear hyperbolic evolution of the initial data $\vec{u}_{n}(0) := \tau_{h_{n}^{-1}} \T_{\lmb_{n}} (f, g)$, i.e., $\vec{u}_{n}$ is the solution to~\eqref{u eq} with initial data $\vec u_{n}(0)$, defined on the maximal interval of existence $I_{n, \max} = (T_{n, -}, T_{n, +})$. 
\end{itemize}
Then the following conclusions are true:
 
\begin{enumerate}
\item \label{item: e h:1}
	Let $T_0 >0$ be arbitrary and denote by $I_n$ the interval $I_n:= (-T_0/ \la_n, T_0/\la_n)$. Then, there exists a positive integer $N = N(T_0, f, g)>0$ so that for all $n >N$ we have $I_n \subseteq I_{n, \max}$ and 
\EQ{
\| \vec{u}_n - \vec{v}_n\|_{L^{\infty}_t(I_n; \HH)} +  \| u_n - v_n\|_{S(I_n)}  = o_n(1) \mas n \to \infty
}
\item \label{item: e h:2}
	For $n$ large enough we have $I_{n, \max} = \R$. Moreover, we have
\EQ{
\| u_n\|_{S([0, \infty))} \le C( \E_{\euc}( f, g) ) 
}
 This means that $\vec u_n$ scatters to free waves as $t \to \infty$. An analogous statement holds in the negative time direction as well.
 \item \label{item: e h:3}
 	Finally, denote by $\vec v_{\pm \infty}$ the free scattering data for the nonlinear Euclidean evolution $\vec v(t)$, i.e, 
 \EQ{ \label{eq: e h:3:vinfty}
 \| \vec v(t) - S_{\euc}(t) \vec v_{\pm \infty}\|_{\HH_{\euc}} \to 0 \mas  t \to \pm \infty.
 }
 Then for every $\eps > 0$, there exist numbers $T, N > 0$ such that
 \EQ{ \label{eq: e h:3}
\sup_{n \geq N} \| \vec{u}_{n}(t) - S_{V}(t) \tau_{h_{n}^{-1}} \T_{\la_n} \vec v_{\infty}\|_{L^{\infty}_{t}((T/\lmb_{n}, \infty); \HH)}  < \eps.
  }
  An analogous statement holds in the negative time direction as well. 
\end{enumerate}
\end{prop}

\begin{proof} 
We begin by establishing part (\ref{item: e h:1}). By approximation, we may assume, without loss of generality, that $(f, g) \in C^{\infty}_{0} \times C^{\infty}_{0}(\bbR^{3})$.
We apply Lemma \ref{lem: pert} (perturbation lemma) with $\vec{u} = \vec{u}_{n}$, $\vec{v} = \vec{v}_{n}$ and $\vec{w}_{0} = 0$. As $\vec{u}_{n}$ already solves~\eqref{u eq}, part (\ref{item: e h:1}) would follow once we establish
\begin{equation} \label{eq: e h:pf:1}
	\nrm{\glei(v_{n})}_{L^{1}_{t} (I_{n}; L^{2}(\bbH^{3}))} \to 0 \mas n \to \infty.
\end{equation}
where $\glei(v_{n}) = \Box_{V} v_{n} + v_{n}^{5}$. Using the fact that $v$ solves the Euclidean equation $\Box_{\bbR^{3}} v + v^{5} = 0$, we may expand $\glei(v_{n})$ as follows:
\begin{align}
\glei(v_{n})
=& \tau_{h_{n}^{-1}} \bb( \Box_{\bbH^{3}} (\T_{\lmb_{n}} v(\lmb_{n} t)) - \lmb_{n}^{2} \T_{\lmb_{n}} (\Box_{\bbR^{3}} v)(\lmb_{n} t) \bb)  \label{eq: e h:pf:comm1}\\
	& + \tau_{h_{n}^{-1}} \bb( (\T_{\lmb_{n}} v (\lmb_{n} t))^{5} - \lmb_{n}^{2} \T_{\lmb_{n}} v^{5}(\lmb_{n} t) \bb)  \label{eq: e h:pf:comm2}\\
	& + V \tau_{h_{n}^{-1}} \T_{\lmb_{n}} v (\lmb_{n} t)  	\label{eq: e h:pf:comm3}
\end{align}
We now treat the terms~\eqref{eq: e h:pf:comm1}--\eqref{eq: e h:pf:comm3} in order. By the Euclidean wellposedness theory in Theorem \ref{thm: euc} and the fact that $(f, g) \in C^{\infty}_{0} \times C^{\infty}_{0} (\bbR^{3})$, we have
\begin{equation} \label{eq: e h:pf:bdd}
	\nrm{\vec{v}}_{L^{\infty}_{t}(\bbR; \HH_{\euc})} + \nrm{v}_{L^{\infty}_{t}(\bbR; \dot{H}^{2}(\bbR^{3}))} < \infty.
\end{equation}
Therefore, we can apply Claim \ref{claim: e h:comm} to conclude that
\begin{equation*} 
\nrm{\eqref{eq: e h:pf:comm1}}_{L^{1}_{t}(I_{n}; L^{2}(\bbH^{3}))} \to 0 \mas n \to \infty.
\end{equation*}

Next, using~\eqref{e h:vol}, the $\dot{H}^{1}(\bbR^{3}) \hookrightarrow L^{6}(\bbR^{3})$ Sobolev inequalty and the fact that $\supp \, \chi (\lmb_{n}^{\frac{1}{2}} r) \subseteq \set{r \aleq \lmb_{n}^{-\frac{1}{2}}}$, we estimate the contribution of~\eqref{eq: e h:pf:comm2} as follows:
\begin{align*}
& \hskip-2em
	\nrm{\eqref{eq: e h:pf:comm2}}_{L^{1}_{t}(I_{n}; L^{2}(\bbH^{3}))} \\
	\leq & (T_{0} / \lmb_{n}) \nrm{\chi^{5}(\lmb_{n}^{\frac{1}{2}} r) \frac{\sinh r}{r}}_{L^{6}(\bbR^{3})} \nrm{\lmb_{n}^{\frac{5}{2}} (e^{\lmb_{n}^{-1} \lap} v(\lmb_{n} t, \lmb_{n} r, \omg))^{5}}_{L^{\infty}_{t} (I_{n}; L^{\frac{6}{5}}(\bbR^{3}))} \\
	& + (T_{0} / \lmb_{n}) \nrm{\chi(\lmb_{n}^{\frac{1}{2}} r) \frac{\sinh r}{r}}_{L^{6}(\bbR^{3})} \nrm{\lmb_{n}^{\frac{5}{2}} (e^{\lmb_{n}^{-1} \lap} v^{5})(\lmb_{n} t, \lmb_{n} r, \omg)}_{L^{\infty}_{t} (I_{n}; L^{\frac{6}{5}}(\bbR^{3}))} \\
	\aleq & T_{0} \lmb_{n}^{-\frac{5}{4}} \nrm{v}_{L^{\infty}_{t} (I; \dot{H}^{1}(\bbR^{3}))}^{5} \to 0 \mas n \to \infty.
\end{align*}

Finally, for~\eqref{eq: e h:pf:comm3}, we use the fact that $V$ is bounded, Poincar\'e's inequality and~\eqref{eq: e h:pf:bdd} to estimate
\begin{align*}
	\nrm{\eqref{eq: e h:pf:comm3}}_{L^{1}_{t}(I_{n}; L^{2}(\bbH^{3}))}
	\leq & (T_{0} / \lmb_{n}) \nrm{V}_{L^{\infty}(\bbH^{3})} \nrm{\T_{\lmb_{n}} v(\lmb_{n} t)}_{L^{\infty}_{t}(I_{n}; L^{2}(\bbH^{3}))}  \\
	\leq & (T_{0} / \lmb_{n}) \nrm{V}_{L^{\infty}(\bbH^{3})} \nrm{v}_{L^{\infty}_{t}(I; \dot{H}^{1}(\bbR^{3}))} \to 0 \mas n \to \infty,
\end{align*}
which completes the proof of part (\ref{item: e h:1}).

Next, we prove part (\ref{item: e h:2}). Without loss of generality, we focus on the forward time direction $\set{t > 0}$.
By Theorem \ref{thm: euc}, we have
\begin{equation} \label{eq: e h:2:pf:1}
	\nrm{v}_{L^{5}_{t} (\bbR; L^{10}(\bbR^{3}))} \leq A(\E_{\euc}(\vec{v})) < \infty.
\end{equation}
By part (\ref{item: e h:1}), for any fixed $T > 0$ we have
\begin{equation} \label{eq: e h:2:pf:2}
	\nrm{\vec{u}_{n}(T / \lmb_{n}) - \tau_{h_{n}^{-1}} \T_{\lmb_{n}} \vec{v} (T)}_{\HH} 
	+ \nrm{u_{n} - v_{n}}_{S([0, T/\lmb_{n}))}= o_{n}(1) \mas n \to \infty.
\end{equation}
We claim that there exists a $T > 0$ and $N > 0$ such that
\begin{equation} \label{eq: e h:2:pf:3}
	\sup_{n \geq N} \nrm{u_{n}(t)}_{S([T/\lmb_{n}, \infty))} \leq 1.
\end{equation}
Then part (\ref{item: e h:2}) in the positive time direction would follow, since we have, by~\eqref{eq: e h:2:pf:2},
\begin{equation} \label{eq: e h:2:pf:4}
\begin{aligned}
	\nrm{u_{n}}_{S([0, T/\lmb_{n}))} 
	=& \nrm{v_{n}}_{S([0, T/\lmb_{n}))} + o_{n}(1) \\
	\aleq & \nrm{v}_{L^{5}_{t}([0, T); L^{10}(\bbR^{3}))} + o_{n}(1) \\ 
	\aleq & A(\E_{\euc}(\vec{v})) + o_{n}(1).
\end{aligned}
\end{equation}

It remains to prove~\eqref{eq: e h:2:pf:3}. By~\eqref{eq: e h:2:pf:1}, given any $\eps > 0$, we can find a $T = T(\eps, \vec{v}) > 0$ sufficiently large so that we have $\nrm{v}_{L^{5}_{t} ([T, \infty); L^{10}(\bbR^{3}))} < \eps$. Then by the Strichartz estimate and Duhamel's formula 
\begin{equation*}
	\vec{v}(t) = S_{\euc}(t - T) \vec{v}(T) + \int_{T}^{t} S_{\euc}(t-s)(0, v^{5}(s)) \, \ud s,
\end{equation*}
it follows, for $\eps > 0$ sufficiently small, that
\begin{equation*}
	\nrm{S_{\euc} (t-T) \vec{v}(T)}_{L^{5}_{t}([T, \infty), L^{10}(\bbR^{3}))} \aleq \eps.
\end{equation*}
Then by Corollary \ref{cor: e h:lin:smallScat}, for sufficiently large $n$ we have
\begin{equation} \label{eq: e h:2:pf:5}
	\nrm{S_{V}(t - T/\lmb_{n}) \tau_{h_{n}^{-1}} \T_{\lmb_{n}} \vec{v}(T)}_{S([T / \lmb_{n}, \infty))} \aleq \eps.
\end{equation}
Recall from~\eqref{eq: e h:2:pf:2} that $\vec{u}_{n}(T / \lmb_{n})$ and $\tau_{h_{n}^{-1}} \T_{\lmb_{n}} \vec{v} (T)$ get arbitrarily close to each other in $\HH$ as $n \to \infty$. Applying Lemma \ref{lem: pert} (perturbation lemma) with $t_{0} = T$ and
\begin{equation*}
(\vec{u}(t), \vec{v}(t)) = (\vec{u}_{n}(t), S_{V}(t - T/\lmb_{n}) \tau_{h_{n}^{-1}} \T_{\lmb_{n}} \vec{v}(T)),
\end{equation*}
we conclude that, for sufficiently large $n$ and small enough $\eps > 0$, we have
\begin{equation} \label{eq: e h:2:pf:6}
\begin{aligned}
	& \hskip-2em
	\nrm{\vec{u}_{n}(t) - S_{V}(t - T/\lmb_{n}) \tau_{h_{n}^{-1}} \T_{\lmb_{n}} \vec{v}(T)}_{L^{\infty}_{t}([T/\lmb_{n}, \infty); \HH)} \\
	& + \nrm{u_{n}(t) - S_{V, 0}(t- T/\lmb_{n}) \tau_{h_{n}^{-1}} \T_{\lmb_{n}} \vec{v}(T) }_{S([T/\lmb_{n}, \infty)}  \aleq \eps.
\end{aligned}
\end{equation}
Therefore, we obtain the following estimate for $\nrm{u_{n}(t)}_{S([T/\lmb_{n}, \infty))}$:
\begin{align*}
	\nrm{u_{n}(t)}_{S([T / \lmb_{n}, \infty))} 
	\aleq \eps
	 + \nrm{S_{V, 0}(t - T/\lmb_{n}) \tau_{h_{n}^{-1}} \T_{\lmb_{n}} \vec{v}(T)}_{S([T/\lmb_{n}, \infty))}
	\aleq \eps.
\end{align*}
Taking $\eps > 0$ smaller if necessary, the desired claim~\eqref{eq: e h:2:pf:3} follows.

Finally, we prove part (\ref{item: e h:3}). As in the proof of part (\ref{item: e h:2}), we focus only on the positive time direction. Fix an $\eps > 0$. Recalling the argument leading to~\eqref{eq: e h:2:pf:6}, we see that there exists a $T = T(\eps, \vec{v}) > 0$ and an $N = N(\eps, \vec{v})$ such that
\begin{equation*}
	\sup_{n \geq N} \nrm{\vec{u}_{n}(t) - S_{V}(t - T/\lmb_{n}) \tau_{h_{n}^{-1}} \T_{\lmb_{n}} \vec{v}(T)}_{L^{\infty}_{t}([T/\lmb_{n}, \infty), \HH ) } < \frac{\eps}{2}.
\end{equation*}
Hence the desired conclusion~\eqref{eq: e h:3} will follow once we show that, after taking $T, N$ larger if necessary,
\begin{equation} \label{eq: e h:3:pf:1}
	\sup_{n \geq N} \nrm{\tau_{h_{n}}^{-1} \T_{\lmb_{n}} \vec{v}(T) - S_{V}(T/\lmb_{n}) \tau_{h_{n}^{-1}} \T_{\lmb_{n}} \vec{v}_{\infty} }_{\HH} < \frac{\eps}{2},
\end{equation}
where $\vec{v}_{\infty}$ is defined as in~\eqref{eq: e h:3:vinfty}. 

To verify~\eqref{eq: e h:3:pf:1}, recall first that, by the definition of $\vec{v}_{\infty}$, we have
\begin{equation*}
	\nrm{\vec{v}(T) - S_{\euc}(T) \vec{v}_{\infty}}_{\HH_{\euc}} = o_{T}(1) \mas T \to \infty.
\end{equation*}
On the other hand, by Lemma \ref{lem: e h:lin}, we have
\begin{equation*}
	\nrm{S_{V}(T/\lmb_{n}) \tau_{h_{n}^{-1}} \T_{\lmb_{n}} \vec{v}_{\infty} - \tau_{h_{n}^{-1}} \T_{\lmb_{n}} S_{\euc}(T) \vec{v}_{\infty}}_{\HH} = o_{n}(1) \mas n \to \infty.
\end{equation*}
Combining the preceding two statements,~\eqref{eq: e h:3:pf:1} follows. \qedhere
\end{proof}

\section{A nonlinear profile decomposition} \label{sec: bg nl}
In this section, we prove a nonlinear version of the linear profile decomposition in Theorem~\ref{thm: BG} for the equation~\eqref{u eq}. As in the previous two sections, the potential $V$ is only assumed to satisfy~\eqref{eq:Vdecay}--\eqref{eq:Str4halfVWave:2}. We begin with a definition of the nonlinear profiles.

\begin{defn} \label{def:nonlin}
Consider a sequence $\set{\vec u_{n}} \subseteq \HH(\Hp^3)$ with bounded energy $E_{V}(\vec{u}_{n}) \leq C$ and the corresponding profile decomposition, 
\EQ{
\vec u_n (0) =  \sum_{0<j<J  }\vec U^j_{n, L}(0)+  \vec w_{n}^J
}
as in Theorem~\ref{thm: BG}. To each profile $\vec U^j_{n, L}(0) =  \vec U^j_{\Box, n, L}(0)$, with $\Box = V, \hyp$ or $\euc$, we can associate nonlinear profiles $\vec U^j_{\Box, n, \nl}$ of the respective type, which are defined as follows. 

\vskip.5em
\noindent\textit{Case 1: A perturbed hyperbolic (or stationary) profile.} In this case, recall that
\begin{equation*}
	\vec{U}^{j}_{V, n, L}(0) = S_{V}(-t_{n, j}) \vec{U}^{j}_{V}(0).
\end{equation*}
To the limiting profile $\vec{U}^{j}_{V}(0)$, we associate a \emph{limiting nonlinear profile} $\vec{U}^{j}_{V, \nl}(t)$, which is the unique solution to~\eqref{u eq} so that $- t_{n, j} \in I_{\max}(\vec{U}^{j}_{V, \nl})$ for every $n$ and 
\begin{equation*}
\nrm{\vec{U}^{j}_{V, n, L}(0) - \vec{U}^{j}_{V, \nl}(- t_{n,j})}_{\HH} \to 0 \mas n \to \infty.
\end{equation*}
We remark that such a limiting nonlinear profile always exists by the local Cauchy theory for~\eqref{u eq}. We then define the \emph{associated perturbed hyperbolic nonlinear profiles} as
\begin{equation} \label{eq:nonlin V}
	\vec{U}^{j}_{V, n, \nl}(t) := \vec{U}^{j}_{V, \nl}(t - t_{n, j}).
\end{equation}

\vskip.5em
\noindent\textit{Case 2: A free hyperbolic (or traveling) profile.} 
In this case, recall that $\abs{h_{n}} \to \infty$ as $n \to \infty$, and
\begin{equation*}
	\vec{U}^{j}_{\hyp, n, L}(0) = S_{V}(-t_{n, j}) \tau_{h_{n, j}^{-1}} \vec{U}^{j}_{\hyp}(0).
\end{equation*}
To the limiting profile $\vec{U}^{j}_{\hyp}(0)$, we associate a \emph{limiting nonlinear profile} $\vec{U}^{j}_{\hyp, \nl}(t)$, which is the unique solution to the potential-free equation $\Box_{\bbH^{3}} U + U^{5} = 0$ so that $-t_{n, j} \in I_{\max}(\vec{U}^{j}_{\hyp, \nl})$ for every $n$ and 
\begin{equation*}
	\nrm{S_{\hyp}(-t_{n, j}) \vec{U}^{j}_{\hyp}(0) - \vec{U}^{j}_{\hyp, \nl}(-t_{n, j})}_{\HH} \to 0 \mas n \to \infty.
\end{equation*}
Such a limiting nonlinear profile always exists by the local Cauchy theory for the equation $\Box_{\bbH^{3}} U + U^{5} = 0$. 
Define $\vec{W}^{j}_{\hyp, n, \nl}(t)$ to be the unique solution to~\eqref{u eq} with the data
\begin{equation} 
	\vec{W}^{j}_{\hyp, n, \nl}(0) := \tau_{h_{n, j}^{-1}} \vec{U}^{j}_{\hyp, \nl}(0).
\end{equation}
Then we set the \emph{associated free hyperbolic nonlinear profiles} to be
\begin{equation} \label{eq:nonlin hyp}
	\vec{U}^{j}_{\hyp, n, \nl}(t) := \vec{W}^{j}_{\hyp, n, \nl}(t - t_{n, j}).
\end{equation}

\vskip.5em
\noindent\textit{Case 3: A Euclidean (or concentrating) profile.} 
We proceed as in Case 2. Here we have $\lmb_{n} \to \infty$ as $n \to \infty$, and
\begin{equation*}
	\vec{U}^{j}_{\euc, n, L}(0) = S_{V}(-t_{n, j}) \tau_{h_{n, j}^{-1}} \T_{\lmb_{n}} \vec{V}^{j}_{\euc}(0).
\end{equation*}
To the limiting profile $\vec{V}^{j}_{\euc}(0)$, we associate a \emph{limiting nonlinear profile} $\vec{V}^{j}_{\euc, \nl}(t)$, which is the unique solution to the flat space equation $\Box_{\bbR^{3}} v + v^{5} = 0$ so that $- \lmb_{n, j} t_{n, j} \in I_{\max}(\vec{V}^{j}_{\euc, \nl})$ for every $n$ and 
\begin{equation*}
	\nrm{S_{\euc}(-\lmb_{n, j} t_{n, j}) \vec{V}^{j}_{\euc}(0) - \vec{V}^{j}_{\euc, \nl}(- \lmb_{n, j} t_{n, j})}_{\HH_{\euc}} \to 0 \mas n \to \infty.
\end{equation*}
Such a limiting nonlinear profile always exists by Theorem \ref{thm: euc}.
Let $\vec{W}^{j}_{\euc, n, \nl}(t)$ be the unique solution to~\eqref{u eq} with the data
\begin{equation} 
	\vec{W}^{j}_{\euc, n, \nl}(0) := \tau_{h_{n, j}^{-1}} \T_{\lmb_{n}} \vec{V}^{j}_{\euc, \nl}(0).
\end{equation}
Then we define the \emph{associated Euclidean nonlinear profiles} as
\begin{equation} \label{eq:nonlin euc}
	\vec{U}^{j}_{\euc, n, \nl}(t) := \vec{W}^{j}_{\euc, n, \nl}(t - t_{n, j}).
\end{equation}
\end{defn}

\begin{rem} 
In all three cases, $\vec{U}_{\Box, n, \nl}(t)$ solves the nonlinear equation~\eqref{u eq}. 
However, in the case of Euclidean or free hyperbolic profiles (i.e., $\Box = \euc$ or $\hyp$), the nonlinear profiles $\vec{U}_{\Box, n, \nl}$ can be well-approximated by the respective underlying limiting nonlinear profiles $\vec{V}_{\Box, \nl}$, which solve different nonlinear equations. 
In particular, thanks to the Euclidean well-posedness theory (Theorem \ref{thm: euc}) and Proposition \ref{prop: e h}, the Euclidean nonlinear profiles $\vec{U}_{\euc, n, \nl}$ always satisfy
\begin{equation}
	I_{\max}(\vec{U}_{\euc, n, \nl}) = \bbR, \quad \limsup_{n \to \infty} \nrm{\vec{U}_{\euc, n, \nl}}_{S(\bbR)} \leq A(\E_{\euc}(\vec{V}^{j}_{\euc}(0) ) ) < \infty.
\end{equation}

Moreover, once we establish global well-posedness and scattering for the potential-free equation $\Box_{\bbH^{3}} U + U^{5} = 0$, which is precisely Theorem~\ref{main}, a similar statement holds for the free hyperbolic profiles. Namely, thanks to Proposition~\ref{prop: h h}, the free hyperbolic profiles $\vec{U}_{\euc, n, \nl}$ will always satisfy
\begin{equation}
	I_{\max}(\vec{U}_{\hyp, n, \nl}) = \bbR, \quad \limsup_{n \to \infty} \nrm{\vec{U}_{\hyp, n, \nl}}_{S(\bbR)} \leq A(\E(\vec{U}^{j}_{\hyp}(0) ) ) < \infty.
\end{equation}
once Theorem~\ref{main} is proved. This observation will be used in our proof of Theorem~\ref{main:V} in Section~\ref{sec:main:V:pf}, which will follow the proof of Theorem~\ref{main} in Section~\ref{sec:main:pf}.
\end{rem}

An important feature of the above defintion is that, in all three cases, the nonlinear profile asymptotically approaches the corresponding linear profile at $t = 0$ as $n \to \infty$. More precisely, the following lemma holds.
\begin{lem} \label{lem:nonlin hyp:t=0}
Let $\vec{U}^{j}_{\Box, n, L}$ be defined as in Definition \ref{def:nonlin} above.
Then for each of the types $\Box = V$, $\hyp$ or $\euc$, we have
\begin{equation} \label{eq:nonlin hyp:t=0}
	\nrm{\vec{U}^{j}_{\Box, n, L}(0) - \vec{U}^{j}_{\Box, n, \nl}(0)}_{\HH} \to 0 \mas n \to \infty.
\end{equation}
\end{lem}
\begin{proof} 
Indeed, \eqref{eq:nonlin hyp:t=0} is obvious for $\Box = V$, so it only remains to consider the cases $\Box = \hyp$ or $\euc$. In both cases, after possibly passing to a subsequence, we divide further into two scenarios: $ - \lmb_{n,j} t_{n, j} \to t_{0} \in \bbR$ or $\lmb_{n, j} \abs{t_{n, j}} \to \infty$, as $n \to \infty$. In the former scenario, \eqref{eq:nonlin hyp:t=0} follows from Proposition~\ref{prop: h h}(\ref{item: h h:1}) or \ref{prop: e h}(\ref{item: e h:1}), respectively. In the latter scenario, we use Proposition~\ref{prop: h h}(\ref{item: h h:3}) or \ref{prop: e h}(\ref{item: e h:3}), respectively.  \qedhere
\end{proof}

The main result of this section is the following nonlinear profile decomposition. 

\begin{thm}[Nonlinear profile decomposition]\label{thm:nonlinbg}
Let $V$ be a potential satisfying the assumptions~\eqref{eq:Vdecay}--\eqref{eq:Str4halfVWave:2}.
Consider a sequence $\vec u_n(0) \in \HH$ of initial data sets with uniformly bounded energy, i.e., $E_{V}(u_{n}) \leq C < \infty$, with a corresponding  linear profile decomposition given by Theorem~\ref{thm: BG}. Let $\vec U^j_{\Box, n, \nl}$ (where $\Box = V, \hyp$ or $\euc$) be the associated nonlinear profiles given by Definition~\ref{def:nonlin}. 

Let $\{s_n\} \subseteq (0, \infty)$ be any sequence of times so that for each positive integer $j \in \N$ corresponding to $\Box = V$ or $\hyp$, we have for all $n$
\EQ{ \label{boundedS}
s_n - t_{n, j} \leq T_+( \vec U^j_{\Box, \nl}) \mand  \limsup_{n \to \infty} \| U^j_{\Box, \nl}\|_{S([- t_{n, j}, s_n- t_{n, j}))} < \infty.
 }
If we denote by $\vec u_n(t) \in \HH$ the solution to~\eqref{u eq} with initial data $\vec u_n(0)$ and maximal interval of existence $I_{\max}(\vec u_n)$, then for each $n$, $\vec u_n(t)$ is defined on $[0, s_n)$, i.e., $[0, s_n) \subseteq I_{\max}(\vec u_n)$, and 
\EQ{
\limsup_{n \to \infty} \| u_n\|_{S([0, s_n))} < \infty. 
}
Moreover, the following nonlinear profile decomposition holds: For $\vec \gamma^J_n$ defined by 
\EQ{
 \vec u_n(t) =  \sum_{j <J} \vec U^j_{\Box, n, \nl}(t) + \vec w_{n, L}^J(t) + \vec  \ga_{n}^J(t)
 }
we have 
\EQ{ \label{eq:nonlinbg:w-gmm}
\lim_{J \to \infty} \limsup_{n \to \infty}  \Big(  \| \vec \ga^J_n\|_{L^{\infty}_t([0, s_n); \HH)} + \| \ga^J_n\|_{S([0, s_n))}\Big)  = 0
}
where $w_{n, L}^J(t) \in \HH$ is as in Theorem~\ref{thm: BG}.
\end{thm}

\begin{rem}
We note that although Theorem~\ref{thm:nonlinbg} is stated only for solutions to a defocusing equation,~\eqref{u eq}, the statement and proof can easily be adapted to focusing-type equations as well. Of course one must take into account the different possible behaviors of the underlying Euclidean equation and thus of the Euclidean profiles, e.g.,  type I and type II blowup, which do not occur in the defocusing setting.  

\end{rem} 

\begin{proof}[Proof of Theorem~\ref{thm:nonlinbg}] 
The proof is again a consequence of Lemma~\ref{lem: pert} (perturbation lemma). For $J \geq 1$, we define
\EQ{ \label{vnjdef}
\vec v_{n}^J(t):= \sum_{j  =1 }^{J} \vec U^j_{\Box, n, \nl}(t)
}
and we check that the conditions of Lemma~\ref{lem: pert} are satisfied by the pair $\vec u_n(t)$ and $\vec v_n^J(t)$ for large $n, J$. For all $n$, we set $I_n:= [0, s_n)$. 

First we claim that 
\EQ{ \label{gleismall}
\| \glei(v_n^J)\|_{L^1_t(I_n;  L^2(\Hp^3))}  = o_n(1) \mas n \to \infty
}
Since each nonlinear profile solves the equation~\eqref{u eq}, we see that
\ant{
 \glei( v^J_n) &=  (\Box_{\Hp^3} + V)v^J_n + (v_n^J)^5 \\
 & = \bb( \sum_{j=1}^{J} U^{j}_{\Box, n, \nl} \bb)^{5} - \sum_{j=1}^{J} \bb( U^{j}_{\Box, n, \nl} \bb)^{5}.
 }
Expanding out the last line, we obtain a sum of quintic terms of the form 
\begin{equation*}
	U^{j_1}_{\Box, n, \nl}U^{j_2}_{\Box, n, \nl}U^{j_3}_{\Box, n, \nl}U^{j_4}_{\Box, n, \nl}U^{j_5}_{\Box, n, \nl}
\end{equation*}
where at least two of the profiles are distinct. Without loss of generality, let $j_{1}$ and $j_{2}$ denote indices corresponding to distinct profiles. We then have
\begin{equation} 
\begin{aligned}
	& \hskip-2em
	\nrm{U^{j_1}_{\Box, n, \nl}U^{j_2}_{\Box, n, \nl}U^{j_3}_{\Box, n, \nl}U^{j_4}_{\Box, n, \nl}U^{j_5}_{\Box, n, \nl}}_{L^{1}_{t}(I_{n}; L^{2}(\bbH^{3}))} \\
	\aleq & \nrm{U^{j_{1}}_{\Box, n, \nl} U^{j_{2}}_{\Box, n, \nl}}_{L^{\frac{5}{2}}_{t}(I_{n}; L^{5}(\bbH^{3}))} \prod_{k=3}^{5} \nrm{U^{j_{k}}_{\Box, n, \nl}}_{S(I_{n})}.
\end{aligned}
\end{equation}
By the hypothesis~\eqref{boundedS} and the fact that all Euclidean profiles have bounded scattering norm, we have $\limsup_{n \to \infty} \nrm{U^{j_{k}}_{\Box, n, \nl}}_{S(I_{n})} < \infty$ for $k = 3,4,5$. Hence, to prove~\eqref{gleismall}, it suffices to show that for each fixed $J$ and distinct $1 \leq j_1, j_2 < J$, we have 
\EQ{ \label{525} 
 \| U^{j_1}_{\Box, n, \nl}U^{j_2}_{\Box, n, \nl}\|_{L^{\frac{5}{2}}_{t}(I_n; L^5(\bbH^{3}))} = o_n(1) \mas n \to \infty.
 }
 The proof of~\eqref{525}  follows from the asymptotic orthogonality of the parameters $\{t_{n, j_1}, h_{n, j_1}, \la_{n, j_1}\}$ and $\{t_{n, j_2}, h_{n, j_2}, \la_{n, j_2}\}$, and one must consider separately several different cases, much as in the proof of the orthogonality of the free energy of the linear profiles; see the proof of Claim~\ref{claim: o en}, \cite[Proof of Lemma 5.4(ii)]{IPS} or \cite[Proof of Lemma~$4.1$]{BG} for the version of this argument in the Euclidean case.
 
 Next, we need to show that 
 \EQ{ \label{boundedSJ}
  \limsup_{n \to \infty}  \| \sum_{j  =1 }^{J} U^j_{\Box, n, \nl}  \|_{S(I_n)} \le C < \infty
  }
  \textit{uniformly in $J$}. This follows from the small data theory,  the orthogonality of the energy of the linear profiles at time $t=0$, and the definition of the nonlinear profiles. The point is that we can divide the sum into two parts, one which only has indices $1 \le j <J_0$ and the other over $J_0 \le j  <J$. This division is performed by examining the $\HH$ norm, by noting that the orthogonality of energy~\eqref{o en} allows us to find $J_0>0$  so that 
  \EQ{
   \limsup_{n \to \infty}  \sum_{J_0  \le j < J} \| \vec U_{\Box, n, L}(0)\|_{\HH}^2  < \de_0^2
   }
   where $\de_0 > 0$ is a small number  as in the Proposition~\ref{small data}.  Using again the orthogonality of parameters~\eqref{o pa} as in the proof of~\eqref{525}, along with Proposition~\ref{small data} (small data theory) as well as  the definition of the nonlinear profiles, we now obtain
\begin{align*}
   \limsup_{n \to \infty} &  \| \sum_{1 \leq j \leq J} U^j_{\Box, n, \nl} \|_{S(I_n)}^5  
   = \limsup_{n \to \infty}  \sum_{1 \leq j \leq J} \|U^j_{\Box, n, \nl}\|_{S(I_n)}^5 \\
       \leq & \limsup_{n \to \infty}  \sum_{1 \leq j  < J_{0} } \|U^j_{\Box, n, \nl}\|_{S(I_n)}^5  
       		+ C \limsup_{n \to \infty}  \sum_{J_0  \le j <J} \| \vec U_{\Box, n, L}(0) \|_{\HH}^5 
\end{align*}
  with a constant $C$ that is absolute. Using the assumption~\eqref{boundedS} and the fact that all Euclidean profiles have bounded $S$-norm on the first $J_0 - 1$ profiles  gives~\eqref{boundedSJ} uniformly in $J$. Essentially the  same argument works for the $L^{\infty}_t( I_n; \HH)$ norm of the approximate solution $\vec v_n^J$. 
  
  Finally, we note that we can choose $J$ large enough to ensure that the $S(I_n)$ norm of $S_{V}(t)( \vec u_n(0) - \vec v_n^J(0))$ is small enough to apply Lemma~\ref{lem: pert}. Indeed, due to the definition of the nonlinear profiles,  the initial difference $\vec u_n(0) - \vec v_n^J(0)$ is given by $\vec w_n^J$ plus a small energy error. We know from~\eqref{s norm to 0} that $\vec S_{V}(t) \vec w_n^J$ has small $S(I_n)$ norm for  $n$ and $J>J_1$ large enough, and the smallness of the evolution of the difference $\vec u_n(0)-\vec v_n^J(0)$ follows from Lemma~\ref{lem:nonlin hyp:t=0}. The conclusions of Theorem~\ref{thm:nonlinbg} now follow directly from Lemma~\ref{lem: pert}. 
 \end{proof}

\section{Proof of Theorem~\ref{main} via Kenig-Merle argument} \label{sec:main:pf}
The purpose of this section is to prove the first nonlinear application of the tools developed so far, namely Theorem~\ref{main} concerning global well-posedness and scattering of solutions to the energy-critical defocusing semilinear wave equation without potential 
\begin{equation*} \tag{\ref{u eq:noV}}
	u_{tt} - \lap_{\bbH^{3}} u = - \abs{u}^{4} u.
\end{equation*}
The proof proceeds via the Kenig-Merle concentration compactness/rigidity method, \cite{KM06, KM08},  which is a contradiction argument. 
In Section~\ref{sec: ce}, we show that in the event Theorem~\ref{main} fails, there exists a global-in-time, nonzero solution $\vec{u}_{\ast}(t) \in \HH$ to \eqref{u eq:noV}, which we will referred to as a \emph{critical element}, that does not scatter forward in time and whose trajectory is pre-compact in $\HH$ up to translations. Then in Section~\ref{sec: rig}, we show contradiction by proving that such a solution cannot exist, thereby finishing the proof of Theorem~\ref{main}.

\subsection{Induction on energy: construction of a critical element}\label{sec: ce}
The goal of this section is to prove the following proposition. 
\begin{prop}[Construction of critical element] \label{prop:ce}
Suppose that Theorem~\ref{main} fails. Then there exists a global-in-time, nonzero solution $\vec u_{*}(t) \in \HH$ to \eqref{u eq:noV} (referred to as a \emph{critical element}) which does not scatter in forward time. In particular we have 
\EQ{
\| u_{*}\|_{S([0, \infty))} = \infty.
}
Moreover, there exists a function $h: [0, \infty) \to \GG$ so that the set 
\EQ{ \label{K}
K:= \{  \tau_{h(t)} \vec u_{*}(t) \mid t \in[0, \infty)\}
}
is pre-compact in $\HH$. 
\end{prop}

The main tool for establishing Proposition~\ref{prop:ce} will be the linear and nonlinear profile decompositions, namely Theorems~\ref{thm: BG} and~\ref{thm:nonlinbg}. Before we begin the proof, however, we point out a simplification due to the absence of a potential $V$. Namely, nonlinear profiles corresponding to stationary (or perturbed hyperbolic) and traveling (or free hyperbolic) are both given by
\begin{equation*}
	U_{\Box, n, \nl}^{j} = \tau_{h_{n, j}^{-1}} U_{\Box, \nl}^{j}(t),
\end{equation*}
where the limiting profile $U_{\Box, \nl}$ solves the same equation \eqref{u eq:noV}. 
Abusing the terminology a bit, we will refer to both types of profiles as \emph{hyperbolic} and write 
\begin{equation*}
U_{\hyp, n, \nl}^{j}(t) = \tau_{h_{n, j}^{-1}} U_{\hyp, \nl}^{j}(t).
\end{equation*}

To prove Proposition~\ref{prop:ce} we will essentially follow an argument due to Kenig and Merle from~\cite{KM11a} which is a more robust  version  of the induction on energy arguments from~\cite{KM06, KM08}. The main difference is that here we induct on the $L^{\infty}_t \HH$-norm of the solution rather than on the conserved nonlinear energy. This different procedure is useful in situations other than in the energy critical case, or when the conserved energy contains complicated nonlinear terms. We note that the argument can be slightly simplified by just inducting on the conserved nonlinear energy, but we present this more robust technique here to simplify potential further applications of the theory.
In particular, we will follow the argument from~\cite{L13}, which uses this procedure in the case where there are two different types of nonlinear profiles.
A similar argument is also found in~\cite{IPS}. 

We begin with some notation, following~\cite{KM11a}. Given initial data $(u_0, u_1) \in \HH$ we denote by $\vec u(t) \in \HH$ the solution to~\eqref{u eq} given by Proposition~\ref{small data} on its maximal interval of existence, $I_{\max}(\vec u) = (T_-(\vec u) , T_+(\vec u))$.  For all $A>0$ define 
\ant{
\B(A) := \{ (u_0, u_1) \in \HH \mid  \| \vec u( \cdot)\|_{L^{\infty}_t ([0, T_+( \vec u)); \HH(\Hp^3))}  \le A \}
}
\begin{defn} We will say that $\SC(A)$ holds if for all $(u_0, u_1) \in \B(A)$ one has $T_+(\vec u) = + \infty$ and $\| u \|_{S([0, \infty))}< \infty$. Also, we say that $\SC(A; \vec u)$ holds if $\vec u \in \B(A)$ and one has $T_+(\vec u) = + \infty$ and $\| u \|_{S([0, \infty))}< \infty$.
\end{defn} 

\begin{rem} We note that Theorem~\ref{main} is true in the positive time direction  if and only if $\SC(A)$ holds for all $A>0$. The reduction in the negative time direction is analogous. 
\end{rem}

Now suppose that Theorem~\ref{main} is \emph{false}. By the small data theory, i.e., Proposition~\ref{small data}, there exists a number $A_0>0$ small enough so that $\SC(A_0)$ holds. As we are assuming that Theorem~\ref{main} fails, we can then find a \emph{critical value} $A_C>0$ so that for $A<A_C$, $\SC(A)$ holds and for $A>A_C$, $\SC(A)$ fails. We note that $0< A_0<A_C$.   

We can now reformulate Proposition~\ref{prop:ce} in this new language, i.e., we find a special solution $\vec u$ to~\eqref{u eq:noV} that is an element of $\B(A_C)$ for which $\SC(A_C, \vec u)$ fails under the assumption that Theorem~\ref{main} is false.  
\begin{prop} \label{lem:ce}
Suppose that Theorem~\ref{main} fails and let $A_C>0$ be defined as above. Then there exists a solution $\vec u_*(t) \in \HH$ to \eqref{u eq:noV}, with $\vec u_*(0)\in\B(A_C)$, defined on $[0, \infty)$ so that $\SC(A_C;  \vec u_*)$ fails. In particular we have 
\EQ{
\| u_*\|_{S([0, \infty))} = \infty.
}
Moreover, there exists a function $[0, \infty): \R \to \GG$ so that the set 
\EQ{ \label{K*}
K:= \{  \tau_{h(t)} \vec u_*(t) \mid t \in [0, \infty)\}
}
is pre-compact in $\HH$. 
\end{prop} 
\begin{rem}
We note that Proposition~\ref{lem:ce} is simply a rewording of Proposition~\ref{prop:ce} and therefore it suffices to prove Proposition~\ref{lem:ce}. 
\end{rem}
 \begin{proof}[Proof of Lemma~\ref{lem:ce} and hence also of Proposition~\ref{prop:ce}]
 By the definition of $A_C$ we can find a sequence $A_n \searrow A_C$ and a sequence of initial data $\vec u_n(0) \in \HH$ with solutions $\vec u_n(t)  \in \HH$ to~\eqref{u eq} defined on maximal intervals $I_{\max, n} := (T_-(\vec u_n), T_+( \vec u_n))$ so that 
 \EQ{ \label{uncont}
  \sup_{t \in [0, T_+(\vec u_n))} \| \vec u_n(t) \|_{\HH}  \le A_n, \mand \|u_n \|_{S([0, T_+( \vec u_n))}  = \infty
  }
 By Theorem~\ref{thm: BG} we can find a linear profile decomposition for the sequence $\vec u_n(0)$, 
 \EQ{
 \vec u_n(0) =  \sum_{j < J}  \vec U^j_{n, L}(0) + \vec w_{n,L}^{J}(0),
 } 
and by Theorem~\ref{thm:nonlinbg} we also have a corresponding nonlinear profile decomposition 
 \EQ{
  \vec u_n(t)  = \sum_{j < J}  \vec U^j_{n, \nl}(t) + \vec w_{n, L}^J(t) + \vec \gamma_{n}^J(t)
 }
where $\vec w_{n, L}^J$ and $\vec \gamma_{n}^J$ obeys \eqref{eq:nonlinbg:w-gmm}  on any interval $[0, s_n)$ as in Theorem~\ref{thm:nonlinbg}. The main idea will be to use the minimality of $A_C$ to deduce that there can only be one nonzero profile above, and it must be non-scattering. As all Euclidean profiles necessary scatter, it then follows that the remaining profile is hyperbolic. 

We begin by deducing, using the orthogonality of the free energy in~\eqref{o en}, that there can only be finitely many non-scattering profiles. Indeed, by~\eqref{o en}, there must be an integer $J_0>0$ so that for all $j \ge J_0$ we have $\| \vec U^j_{n, L}(0)\|_{\HH} < \de_0$ where $\de_0>0$ is the small number in Proposition~\ref{small data}. It follows then from Proposition~\ref{small data} that for all $j > J_0$ and $n$ large enough we have 
 \ant{
I_{\max}( \vec U^j_{n, \nl}) = \R \mand \|  \vec U^j_{n, \nl} \|_{L^{\infty}_{t}(\R; \HH)} + \| U^j_{n, \nl}\|_{S(\R)} \le C \| \vec U_{n, L}^j(0)\|_{\HH}  \lesssim \de_{0}
}

Besides this fact, we also know by  Definition~\ref{def:nonlin} that \emph{all} of the nonlinear Euclidean profiles are global and scattering solutions. That is, we have 
\EQ{
\| U^\ell_{\euc, n, \nl}\|_{S(\R)} < C( \|\vec V^\ell_{\euc}(0)\|_{\HH}) \quad  \forall  \ell \in \N
}

However, we claim that there is at least one hyperbolic profile that does not scatter. 
\begin{claim}\label{nscatprof}
Let $J_0 \in \N$ be as above and let $J_1 = J_1(J_0)$ denote the number of hyperbolic profiles with indices  $j<J_0$. Then $J_{1} \geq 1$. Moreover, it is impossible that for all $1 \le j  \le J_1$ we have 
\EQ{ \label{boundedsprof}
\| U^j_{\hyp,  \nl}\|_{S(-t_{n, j}, T_+(\vec U^j_{\hyp, \nl}))}  < \infty
}
where we recall that  $ \vec U^j_{\hyp, n, \nl}(t) =   \tau_{h_{n, j}^{-1}} \vec U^j_{\hyp, \nl}(t -t_{n, j})$. 
\end{claim}
\begin{proof}[Proof of Claim~\ref{nscatprof}]
First, if $J_{1} = 0$ then it follows that all profiles are Euclidean and hence scattering. Therefore, we can apply the conclusions of Theorem~\ref{thm:nonlinbg} with the time sequence $s_n  = + \infty$. It follows that $\vec u_n(t)$ is defined for all positive times for $n$ large enough and 
\ant{
 \limsup_{n \to \infty} \| u_n\|_{S([0, \infty))}  \le C < \infty
 }
which contradicts the definition of the $\vec u_n$ in~\eqref{uncont}.

Now assume that $J_{1} \geq 1$, and let $1 \leq j \leq J_{1}$. Note that~\eqref{boundedsprof} holds for all profiles $\vec U^j_{\hyp, \nl}$ for which $-t_{n, j} \to+ \infty$. Hence we have to prove that~\eqref{boundedsprof} fails for at least one profile with either $t_{n, j} = 0$ for all $n$ or in the case where $-t_{n, j} \to - \infty$.  If~\eqref{boundedsprof} holds for all of the hyperbolic profiles then $T_+(\vec U^j_{\hyp, \nl})  = + \infty$ for all $j \in \N$ and then we again can apply the conclusions of Theorem~\ref{thm:nonlinbg} with the time sequence $s_n  = + \infty$ and derive a contradiction, which proves the claim. 
\end{proof}

Given the conclusions of Claim~\ref{nscatprof} we can find an integer $ \ti J$ and relabel and rearrange the profiles, so that each profile for $1 \le j \le \ti J$ is hyperbolic and satisfies 
\ant{
\| U^j_{\hyp, \nl} \|_{S((-t_{n, j}, T_+( \vec U^j_{\hyp, \nl}))} =  + \infty \mfor 1 \le j \le \ti J
}
and every hyperbolic profile, $\vec U^j_{\hyp, \nl}$ with index $j> \ti J$ satisfies  $T_+(\vec U^j_{\hyp, \nl}) = \infty$ and scatters in forward time. Note again that all of the Euclidean nonlinear profiles exist globally and scatter in both time directions for $n$ large enough. We also note that for each index $j  \le  \ti J$ we can assume that either $t_{n, j} = 0$ for all $n$ or $- t_{n, j} \to - \infty$. 

Now, we define  sequences of times, $T^+_{j, k}$ and $s^n_{j, k}$ by 
\EQ{
&T_{j, k}^+ := \begin{cases} T_+(\vec U^j_{\hyp, \nl}) - \frac{1}{k} &\mif T_+(\vec U^j_{\hyp, \nl}) < \infty \\ k &\mif T_+(\vec U^j_{\hyp, \nl}) = + \infty\end{cases} \\
& s^k_{n, j}  - t_{n, j}  = T^+_{j, k} \\
& s^{k}_n:= \min_{1 \le j \le  \ti J} s^{k}_{n, j}
}
The point here is that for $n$ large enough, for all $j \ge 1$  and for each $k \in \N$ we have 
\ant{
s^n_k - t_{n, j} < T_+(\vec U^j_{\hyp, \nl}) \, \mand \| U^j_{\hyp, \nl}\|_{S(-t_{n, j}, s^n_k - t_{n, j} )}< \infty
}
and thus the sequence $s_{n}^k$ satisfies the conditions for the nonlinear profile decomposition, Theorem~\ref{thm:nonlinbg} and again with $J_1 + J_2 = J-1$ the decomposition 
\EQ{
  \vec u_n(t)  = \sum_{j  \le J_1}  \vec U^j_{\hyp, n, \nl}(t) + \sum_{\ell \le J_2} \vec U^\ell_{\euc, n, \nl}(t) + \vec w_{n, L}^J(t) + \vec \gamma_{n}^J(t)
}
holds on the interval $[0, s_n^k)$ for each $k$ with 
\ant{
\lim_{J \to \infty} \limsup_{n \to \infty} \left( \| w_{n, L}^J\|_{S([0, s_n^k))} + \| \ga_{n}^J\|_{S([0, s_n^k))} + \| \vec \ga_{n}^J\|_{L^{\infty}([0, s_n^k); \HH)}  \right) = 0
}
Next, note that  there exists $j_0$ with $1 \le j_0 \le  \ti J$ and a subsequence $k_{\al}$ so that we have 
\EQ{
s_{n}^{k_\al} = s_{n, j_0}^{k_{\al}} \mfor \, \, \textrm{all}   \, \, \al
}
Indeed this follows from the pigeonhole principle, due to the fact that $\ti J$ is finite. This allows us to single out the  profile $\vec U^{j_0}_{\hyp, \nl}$ for closer examination.  By definition of $\ti J$, and the fact  that either $t_{n, j} = 0$ for all $n$ or $- t_{n, j} \to - \infty$  we have $0 \ge -t_{n, j}$ for $n$ large enough and
\EQ{
 \| U^{j_0}_{\hyp, \nl}\|_{S([0, T_+(\vec U^{j_0}_{\hyp, \nl})))}  = \infty
 }
 Now, by the definition of $A_C$ we then have 
 \EQ{ \label{A2}
 A^2:= \sup_{t \in [0, T_+(\vec U^{j_0}_{\hyp, \nl}))} \| \vec U^{j_0}_{\hyp,\nl}(t) \|_{\HH}^2  \ge A_C^2
 }
 and by the definition of $T^+_{j_0, k}$, 
 \EQ{
 A_k^2 := \sup_{t \in [0, T^+_{j_0, k})} \| \vec U^{j_0}_{\hyp,\nl}(t) \|_{\HH}^2 \to A^2 \mas k \to \infty
 }
For each $k$ we can find a time $T_{j_0, k} \in [0, T^+_{j_0, k}]$ so that 
 \ant{
 A_k^2 =  \| \vec U^{j_0}_{\hyp,\nl}(T_{j_0, k}) \|_{\HH}^2
 }
 and  for each $n$ we denote by $\tau^k_{n, j_0}$ the time so that 
 \ant{
  \tau_{n, j_0}^k  - t_{n, j_0} = T_{j_0, k}
  }
  Note that for $n$ large enough and $k_{\al}$ as before, we have $0 \le \tau_{n, j_0}^{k_\al} \le s^{k_\al}_{n, j_0} = s_n^{k_\al}$. Hence the nonlinear profiles are all defined at the times $\tau_{n, j_0}^{k_{\al}}$.  

 Next, we claim the following almost orthogonality for each fixed $k_{\al}$ and large $J$ at the time $\tau_{n, j_0}^{k_{\al}}$: 
 \EQ{ \label{oHtau}
 \|  \vec u_n( \tau_{n, j_0}^{k_{\al}}) \|_{\HH}^2  = \sum_{1 \le j < J} \| \vec U^j_{n,  \nl}( \tau_{n, j_0}^{k_{\al}}) \|_{\HH}^2 
 +  \| \vec w_{n, L}^J(  \tau_{n, j_0}^{k_{\al}}) \|_{\HH}^2 + o_{n, J}
 } 
 where $\lim_{J \to \infty} \limsup_{n \to \infty} o_{n, J} = 0$. The proof of~\eqref{oHtau} again boils down to the orthogonality of the parameters as in the proof of~\eqref{o en} and various cases must be considered. The initial step is to first split the sum into two pieces $1 \le j \leq J_0$ and $ J_0 < j < J$ where $J_0$ is chosen large enough to ensure that the sum of the energies of all profiles after $J_0$ is very small.  Indeed for each $\e >0$ we can find $J_0$ large so that 
 \ant{
   \sum_{j > J_0} \| \vec U^j_{n, \nl}( \tau_{n, j_0}^{k_{\al}}) \|_{\HH}^2 < \e
   }
 This is a simple consequence of~\eqref{o en}, the definition of the nonlinear profiles, and Proposition~\ref{small data}. For the first $J_0$ profiles, one cannot simply rely on the triangle inequality and we must consider the cross terms. Indeed, it suffices now to show that for all distinct $j, \ell  \leq J_0$ we have 
 \EQ{ \label{crossterms}
  &\ang{  \vec U^j_{n, \nl}(  \tau_{n, j_0}^{k_{\al}}) \mid   \vec U^{\ell}_{n, \nl}( \tau_{n, j_0}^{k_{\al}})}_{\HH}  = o_n(1) \mas n \to \infty\\
  & \ang{ \vec U^j_{n, \nl}(  \tau_{n, j_0}^{k_{\al}}) \mid   \vec w_{n, L}^J(  \tau_{n, j_0}^{k_{\al}})}_{\HH} = o_n(1) \mas n \to \infty
  }
  The proof of cross terms is similar to the proof of~\eqref{o en} and involves analyzing different cases based on how the parameters diverge from each other and whether the profiles are hyperbolic or Euclidean. The only new ingredient (as compared to the proof of~\eqref{o en})   is the definition of the nonlinear profiles and we thus omit the details and  refer the reader to~\cite[Proof of $(3.22)$]{KM11a} for the details of an analogous  argument in the Euclidean case.

 From~\eqref{oHtau} we can deduce that  
\ant{
 A_n^2 \ge A_{k_{\al}}^2 + o_{n, J}
 }
 Letting $n \to \infty$ and then $J \to \infty$ yields $A_C^2 \geq A_{k_{\al}}^2$. Then taking $k_{\al} \to \infty$ we have $A_C^2  \ge A^2$. Combining this last point with~\eqref{A2} then shows that 
 \ant{
 A = A_C
 }
 From this we can immediately conclude that all of the profiles other that $\vec U^{j_0}_{\hyp, \nl}$ must be identically $= 0$ as well as $\|\vec w_{n}^J \|_{\HH} \to 0$  as $n \to \infty$ for each $J$. To see this suppose that some $j \neq j_0$ corresponds to a nonzero profile, $\vec U^j_{n, \nl}$. For any $\e>0$, choose $k_{\al}$ large enough so that 
 $\abs{A_C^2 - A_{k_{\al}}^2}< \e$. Using~\eqref{oHtau} we then have 
 \EQ{
 A_n^2 \ge A^2_C - \e + \| \vec U^{j}_{n, \nl}(\tau_{n, j_0}^{k_{\al}})\|_{\HH}^2 + o_{n, J}
 }
 Taking $n$ and $J$ large and $\e$ small we can make $\| \vec U^{j}_{n, \nl}(\tau_{n, j_0}^{k_{\al}})\|_{\HH}^2$ arbitrarily small. By the small data theory, we can then ensure that the evolution $ \sup_{t \in \R} \| \vec U^{j}_{n, \nl}\|_{\HH}$ is arbitrarily small, which then will imply that the associated limiting profile is identically zero. A similar argument shows the vanishing of the error $\vec w_{n}^J$ in the energy space. Therefore, there is only one nonzero profile and it is hyperbolic. 
 
 We now define $\vec u_*(t) :=  \vec U^{j_0}_{\hyp, \nl}(t)$.  One can apply the now standard arguments from~\cite[Proof of Proposition~$4.2$]{KM06, KM08} to deduce the compactness property~\eqref{K*}, and we omit many details below. The most relevant reference is perhaps~\cite[Proof of Lemma~$2.25$ and Proof of Corollary~$2.26$]{NakS}. Indeed the first order of business is the define the translation parameter $h(t)$. Indeed, we claim that one can find $h(t): [0, T_+(\vec u_*)) \to \GG$ so that for all $\e>0$ there exists $R(\e)$ so that 
 \EQ{ \label{NS1}
 \| \tau_{h(t)} \vec u_*(t,  \cdot)\|_{\HH( \bfd_{\Hp^3}(x,  \zero) \ge R(\e))}^2 < \e
 }
 for all $t  \in [0, T_+(\vec u_*))$. The proof of~\eqref{NS1} is precisely ~\cite[Proof of Lemma $2.25$]{NakS} adapted to the hyperbolic space setting. From~\eqref{NS1} it follows from the exact argument from~\cite[Proof of Corollary $2.26$]{NakS} that the trajectory 
 \EQ{
 K_+:= \{  \tau_{h(t)} \vec u_*(t) \mid t \in [0, T_+(\vec u_*)) \}
 } 
 is pre-compact in  $\HH$.   Finally, it remains to show that in fact we have $T_+(\vec u_*) = + \infty$.  For this argument, we refer the reader to~\cite[p. $31$-$32$]{L13} where an identical argument can be applied. The idea in all of the omitted arguments above is to apply linear profile decompositions to sequences $\vec u_*(t_n)$, where $t_n \to T_+(\vec u_*)$ using  the fact that  $\vec u_*$ has infinite Strichartz norm along with the minimality of the critical energy $A_C$. 
 This completes the proof of Lemma~\ref{lem:ce}. 
  \end{proof}


\subsection{Rigidity}\label{sec: rig}

In this section we complete the proof of Theorem~\ref{main} by showing that the critical element $\vec u_*(t)$ constructed in Proposition~\ref{prop:ce} cannot exist. We will achieve this by proving the following rigidity result for solutions to~\eqref{u eq} that have pre-compact trajectories in $\HH$ modulo translation symmetries. 
\begin{prop}[Rigidity of critical element] \label{rig} Let $\vec u(t) \in \HH$ be a global-in-forward time solution to~\eqref{u eq:noV}. Suppose in addition that there exists a function $h: [0, \infty) \to \GG$ so that the trajectory 
\begin{equation} \label{eq:rig:trj}
K:=  \left\{  \tau_{h(t)} \vec u(t) \mid t \in [0, \infty) \right\}
\end{equation}
is pre-compact in the energy space $\HH$. Then $\vec u(t) \equiv (0, 0)$. 
\end{prop} 

The key ingredient of the proof of Proposition~\ref{rig} is the following \emph{Morawetz-type} estimate, which is of the same type as those in \cite{IS, IPS, ShenS14, Shen14}.
\begin{lem}[Morawetz estimate] \label{lem:mwtz}
Given an interval $J \subseteq \bbR$, let $\vec u(t)$ be a solution to~\eqref{u eq:noV} in the class $C_{t} ( J ; \calH)$. Then $u \in L^{6}(J \times \bbH^{3})$ and satisfies
\begin{equation*} 
	\nrm{u}_{L^{6}(J \times \bbH^{3})}^{6} \leq C \calE(\vec u).
\end{equation*}
\end{lem}
We defer the proof of Lemma~\ref{lem:mwtz} until Section~\ref{sec:morawetz}. In what follows, we assume Lemma~\ref{lem:mwtz} and prove Proposition~\ref{rig}.


\begin{proof} [Proof of Proposition~\ref{rig}]
Let $\vec{u}(t) \in \calH$ a global-in-time solution to~\eqref{u eq:noV}, which satisfies the hypothesis of Proposition~\ref{rig}. Applying Lemma \ref{lem:mwtz} with $J = \bbR_+$, we obtain the apriori bound
\begin{equation} \label{eq:rig:mwtz}
	\nrm{u}_{L^{6}_{t, x}(\bbR_+ \times \bbH^{3})} \leq C_{\calE(\vec{u})} < \infty.
\end{equation}

As the left-hand side involves an integration in time, this bound implies some decay of $\vec{u}(t)$ in time. Note, however, that~\eqref{eq:rig:mwtz} is not yet sufficient to guarantee that $\vec{u}(t)$ scatters as $t \to \pm \infty$, as the equation is \emph{super-critical} with respect to the bound~\eqref{eq:rig:mwtz}. In particular,~\eqref{eq:rig:mwtz} cannot be used to rule out concentration of energy into smaller and smaller scales as $t \to \pm \infty$. 
Fortunately, we have already precluded this scenario by ruling out Euclidean nonlinear profiles, so we expect~\eqref{eq:rig:mwtz} to be still effective. Indeed, as we shall see below,~\eqref{eq:rig:mwtz} is incompatible with the pre-compactness of the trajectory $K$ unless $\vec{u} \equiv (0, 0)$, which is sufficient for our purpose.

For $M \geq 1$, we introduce a smooth approximation of the identity $Q_{M}$, defined as
\begin{equation*}
	Q_{M} f = K_{M} \ast f
\end{equation*}
where $K_{M}(x)$ is a non-negative radial kernel on $\bbH^{3}$ supported in the ball $\set{\dist(0, x) < \frac{1}{M}}$ with $\int K_{M} = 1$. We may assume, without loss of generality, that $\nrm{K_{M}}_{L^{\infty}_{x}}$ is uniformly bounded for $M \geq 1$. Then by Young's inequality, for any $1 \leq p \leq q \leq \infty$ we have 
\begin{equation*}
\nrm{Q_{M}}_{L^{p}(\bbH^{3}) \to L^{q}(\bbH^{3})} \aleq M^{\frac{3}{p}-\frac{3}{q}}
\end{equation*}
We remark that $Q_{M}$ commutes with translations, i.e., $\tau_{h} Q_{M} f = Q_{M} \tau_{h} f$ for any $h \in \bbG$.

Given $\eps > 0$, we may use the pre-compactness of $K = \set{\tau_{h(t)}\vec{u}(t) \mid t \in \bbR_+}$ to choose $M > 0$ sufficiently large so that
\begin{equation*}
	\sup_{t \in \bbR_+} \nrm{(1-Q_{M}) \tau_{h(t)} u(t)}_{L^{6}(\bbH^{3})} < \frac{\eps}{2}.
\end{equation*}
Then thanks to the fact that $\tau_{h(t)}$ commutes with $Q_{M}$ and $L^{6}(\bbH^{3})$ is translation-invariant, we see that
\begin{equation*}
	\sup_{t \in \bbR_+} \nrm{(1-Q_{M}) u(t)}_{L^{6}(\bbH^{3})} < \frac{\eps}{2}.
\end{equation*}

Next, by Young's inequality, note that
\begin{equation*}
	\sup_{t \in \bbR_+} \nrm{Q_{M} u(t)}_{L^{\infty}(\bbH^{3})} \leq C_{M, \calE(\vec{u})} \,.
\end{equation*}
Interpolating with~\eqref{eq:rig:mwtz}, we see that the $L^{8}(\bbR_+ \times \bbH^{3})$ norm of $Q_{M} u$ is finite, i.e.,
\begin{equation*}
	\nrm{Q_{M} u}_{L^{8}(\bbR_+ \times \bbH^{3})} \leq C_{M, \calE(\vec{u})}.
\end{equation*}
Therefore, there exists $T = T(\eps, M, \calE(\vec{u}), \vec{u}) > 0$ such that
\begin{equation*}
	\nrm{Q_{M} u}_{L^{8}((T, \infty) \times \bbH^{3})} < \frac{\eps}{2}. 
\end{equation*}
Note that the $L^{8}(\bbR_+ \times \bbH^{3})$ norm is the unique Strichartz norm that is critical with respect to the scaling of~\eqref{u eq:noV} and has the same space and time Lebesgue exponents.

By homogeneous Strichartz estimates, we have
\begin{align*}
\nrm{S_{\bfg}[t-T] \vec{u}(T)}_{S(\bbR_+)}
+ \nrm{S_{\bfg} [t-T] \vec{u}(T)}_{L^{4}_{t} L^{12}_{x} \cap L^{8}_{t,x}(\bbR_+)}
\leq C \calE(\vec{u}),
\end{align*}
where $C > 0$ is independent of $T \in \bbR_+$ and for any interval $J \subseteq \bbR_+$ we write
\begin{equation*}
	\nrm{\cdot}_{L^{4}_{t} L^{12}_{x} \cap L^{8}_{t,x}(J)} := \nrm{\cdot}_{L^{4}_{t} (J; L^{12}(\bbH^{3}))} + \nrm{\cdot}_{L^{8}(J \times \bbH^{3})}.
\end{equation*}
Then for every $T' > T$, it follows from Duhamel's principle that 
\begin{align*}
	\nrm{u}_{S(T, T')}
+ \nrm{u}_{L^{4}_{t} L^{12}_{x} \cap L^{8}_{t,x}(T, T'))}
\leq& C \calE(\vec{u}) + C \nrm{\abs{u}^{4} u}_{L^{1}_{t} ((T, T') ; L^{2}(\bbH^{3}))}.
\end{align*}
We now estimate
\begin{align*}
\nrm{\abs{u}^{4} & u}_{L^{1}_{t} ((T, T'); L^{2}(\bbH^{3}))} \\
\leq & \nrm{\abs{u}^{4} Q_{M} u}_{L^{1}_{t} ((T, T'); L^{2}(\bbH^{3}))} 
	+ \nrm{\abs{u}^{4} (1-Q_{M}) u}_{L^{1}_{t} ((T, T'); L^{2}(\bbH^{3}))} \\
\leq & \nrm{u}_{L^{4}_{t}((T, T'); L^{12}(\bbH^{3})}^{3} \nrm{u}_{L^{8}((T, T') \times \bbH^{3})} \nrm{Q_{M} u}_{L^{8}((T, T') \times \bbH^{3})} \\
	&+ \nrm{u}_{L^{4}_{t} ((T, T'); L^{12}(\bbH^{3}))}^{4} \nrm{(1-Q_{M}) u}_{L^{\infty}_{t}((T, T') ; L^{6}(\bbH^{3}))}  \\
\leq & \eps \nrm{u}_{L^{4}_{t} L^{12}_{x} \cap L^{8}_{t,x} (T, T')}^{4} .
\end{align*}

Choosing $\eps >0$ sufficiently small compared to $\calE(\vec{u}) >0$ (which amounts to taking $M$ and $T$ sufficiently large), we may employ a standard continuity  argument in $T'$ to conclude that 
\begin{align*}
\nrm{u}_{S(T, \infty)} + \nrm{u}_{L^{4}_{t} L^{12}_{x} \cap L^{8}_{t,x} (T, \infty)}
\leq C \calE(\vec{u}) < \infty.
\end{align*}
Since $\vec{u}$ has a pre-compact trajectory, it follows that $\vec{u} \equiv (0, 0)$ as desired. \qedhere
\end{proof}


\section{Proof of Theorem \ref{main:V}: Inclusion of potential} \label{sec:main:V:pf}
In this section, we prove Theorem~\ref{main:V}, which concerns the energy-critical defocusing semi-linear wave equation \eqref{u eq}, with a smooth, compactly supported repulsive potential $V$. The argument so far is flexible enough to include a potential in almost every step without much modification, and as we sketch below, Theorem~\ref{main:V} can be proved essentially by a repetition of the proof of Theorem~\ref{main}. The key difference is that we are now able to use Theorem~\ref{main}, in conjunction with Proposition~\ref{prop: h h}, to conclude that all traveling (or free hyperbolic) nonlinear profiles must scatter, which simplifies some points below; compare \eqref{K:V} with \eqref{K}.

\begin{proof}[Proof of Theorem~\ref{main:V}]
As in the proof of Theorem~\ref{main} in the previous section, we begin by assuming, for the sake of contradiction, that the conclusions of Theorem~\ref{main:V} fail. Then we can deduce the existence of a non-trivial \emph{critical element} by the following proposition.
\begin{prop}[Construction of critical element] \label{prop:ce:V}
Suppose that Theorem~\ref{main:V} fails. Then there exists a global-in-time, nonzero solution $\vec u_{*}(t) \in \HH$ to \eqref{u eq} (referred to as a \emph{critical element}) which does not scatter in forward time. In particular we have 
\EQ{
\| u_{*}\|_{S([0, \infty))} = \infty.
}
Moreover, the set
\EQ{ \label{K:V}
K:= \{  \vec u_{*}(t) \mid t \in[0, \infty)\}
}
is pre-compact in $\HH$. 
\end{prop}
\begin{proof} [Sketch of proof]
Proposition~\ref{prop:ce:V} can be proved by essentially following the proof of Proposition~\ref{prop:ce} in the previous section line by line. Namely, we form a sequence $\vec{u}_{n}(t)$ of non-scattering solutions to \eqref{u eq} which minimize the $L^{\infty}_{t} \HH$ norm, apply the nonlinear profile decomposition (Theorem~\ref{thm:nonlinbg}) and show that there exists only one non-scattering profile, which possesses the minimum $L^{\infty}_{t} \HH$ norm. This remaining profile is defined to be the critical element $\vec{u}_{\ast}(t)$. 

The key difference in the proof, which also results in the simpler compactness property \eqref{K:V}, is that two out of the three distinct types of profiles are always scattering; namely, traveling (or free hyperbolic) profiles always scatter by Theorem~\ref{main} and Proposition~\ref{prop: h h}, and concentrating (or Euclidean) profiles by Theorem~\ref{thm: euc} and Proposition~\ref{prop: e h}. Hence, at the end, the single remaining non-scattering profile must necessarily be stationary (or perturbed hyperbolic), i.e., $\vec{u}_{\ast}(t) = \vec{U}^{j_{0}}_{V, \nl}(t)$. Repeating the proof of \eqref{K} in Proposition~\ref{prop:ce}, but using the fact that only stationary profiles can arise, we obtain \eqref{K:V} with no translation parameters $h(t)$. \qedhere 
\end{proof}

Now we proceed to show that such a critical element cannot exist. The key ingredient is a Morawetz-type estimate analogous to Lemma~\ref{lem:mwtz}, which holds for \eqref{u eq} thanks to the assumption \eqref{eq:Vdefocusing}.
\begin{lem}[Morawetz estimate] \label{lem:mwtz:V}
Let $V$ be a smooth, compactly supported repulsive potential as in \eqref{eq:Vdefocusing}. Given an interval $J \subseteq \bbR$, let $\vec u(t)$ be a solution to~\eqref{u eq} in the class $C_{t} ( J ; \calH)$. Then $u \in L^{6}(J \times \bbH^{3})$ and satisfies
\begin{equation} \label{eq:rig:mwtz:V}
	\nrm{u}_{L^{6}(J \times \bbH^{3})}^{6} \leq C \calE(\vec u).
\end{equation}
\end{lem}
As before, the proof will be given in Section~\ref{sec:morawetz}. Using Lemma~\ref{lem:mwtz:V} and proceeding in the identical manner as in the proof of Proposition~\ref{rig}, we obtain the following rigidity theorem.
\begin{prop}[Rigidity of critical element] \label{rig:V} 
Let $V$ be a smooth, compactly supported repulsive potential as in \eqref{eq:Vdefocusing}, and let $\vec u(t) \in \HH$ be a global-in-forward time solution to~\eqref{u eq}. Suppose that the trajectory 
\begin{equation} \label{eq:rig:trj:V}
K:=  \left\{  \vec u(t) \mid t \in [0, \infty) \right\}
\end{equation}
is pre-compact in the energy space $\HH$. Then $\vec u(t) \equiv (0, 0)$. 
\end{prop} 
Hence, the critical element constructed in Proposition~\ref{prop:ce:V}, assuming the failure of Theorem~\ref{main:V}, cannot exist. This completes the proof of Theorem~\ref{main:V}.
\end{proof}

\section{Integrated local energy decay, Morawetz estimates \\ and Strichartz estimates} \label{sec:morawetz}
In this section, we prove Lemmas~\ref{lem:defocusingV},~\ref{lem:mwtz} and~\ref{lem:mwtz:V}. The latter two are Morawetz estimates for defocusing semi-linear equations, which played a crucial role in the proof of rigidity in Sections~\ref{sec:main:pf} and~\ref{sec:main:V:pf}. Lemma~\ref{lem:defocusingV} says that a smooth, compactly supported potential which is non-negative and repulsive satisfies the assumptions \eqref{eq:Vdecay}--\eqref{eq:Str4halfVWave:2}. A key property we need in the proof is an \emph{integrated local energy decay} estimate; see \eqref{eq:ILED}. As we will see, a single crucial multiplier argument, which goes back to \cite{IS, ShenS14}, underlies all three proofs; see Lemma~\ref{lem:mwtz-general} below.

Consider the (possibly semi-linear) wave equation
\begin{equation} \label{eq:mwtz:u-eq}
	u_{tt} - \lap_{\bbH^{d}} u + V u + F(u) = 0
\end{equation}
and define $G(\sgm) := \int_{0}^{\sgm} F(\sgm') \, \ud \sgm'$. The proofs in this section will be based on the following general result concerning~\eqref{eq:mwtz:u-eq}.
\begin{lem} \label{lem:mwtz-general}
Let $V$ be a bounded $C^{1}$ potential that is \emph{repulsive} in the sense that $-\rd_{r} V \geq 0$, where $\rd_{r}$ is the radial direction derivative in the polar coordinates on $\bbH^{d}$. Given an interval $J \subseteq \bbR$, consider a solution $\vec u(t)$ to~\eqref{eq:mwtz:u-eq} in the class $C_{t} ( J ; \calH)$. Then the solution $u$ satisfies the inequality
\begin{equation} \label{eq:mwtz-general}
\begin{aligned}
	\iint_{J \times \bbH^{d}}  
			&		\bfc^{2} \abs{\nb u}^{2}_{\bfg}
					+ \tanh r (- \rd_{r} V) \abs{u}^{2} \\
			&		+ \bb( \frac{1}{2} u F(u) - G(u) \bb) \, \ud t \,  \mu(\ud x) \leq C \nrm{\vec{u}}_{L^{\infty}_{t} (J; \HH)}^{2},
\end{aligned}
\end{equation}
where $\abs{\nb u}^{2}_{\bfg} := \bfg^{ij} \nb_{i} u \nb_{j} u$ and $\bfc = \bfc(r)$ is a positive radial function given by
\begin{equation} \label{eq:mwtz-general:c}
	\bfc^{2} (r) = \frac{\cosh r}{\sinh^{d} r} \int_{0}^{r} \frac{\sinh^{d-1} r'}{\cosh^{2} r'} \, \ud r'.
\end{equation}
\end{lem}

The proof relies on the existence of a radial function $\bfa$ on $\bbH^{d}$ satisfying 
\begin{equation*}
	\De_{\bbH^{d}} \bfa = 1, \quad \bfa_{r} = \frac{1}{d} r + o_{r}(1) \mas r \to 0
\end{equation*}
which remarkably turns out to have positive and \emph{bounded} radial derivative, i.e., $0 < \bfa_{r} \leq C < \infty$ for $r >0$,  in contrast to the the case of $\bbR^{d}$. The radial derivative $\bfa_{r}$ is explicitly given by
\begin{equation} \label{eq:bfa}
	\bfa_{r} = \frac{1}{\sinh^{d-1} r} \int_{0}^{r} \sinh^{d-1} r' \, \ud r'.
\end{equation}

\begin{proof} 
In this proof, we will employ the geodesic polar coordinates $(t, r, \omg)$, where $r(x) = \bfd_{\bbH^{d}} (x, \zero)$ and $\omg = (\omg^{1}, \ldots, \omg^{d-1})$  is a system of coordinates on $\bbS^{d-1}$. We will use the notation $\displaystyle{\snb}$, $\displaystyle{\sdiv}\!$ and $\displaystyle{\slap}$ for the angular gradient, divergence and Laplacian on $\bbS^{d-1}$, respectively. Note that $\displaystyle{\slap = \ \sdiv \snb}$ and $\ \displaystyle{\sdiv}\!$ is the adjoint of $\snb$. Moreover, the Laplacian $\De_{\bbH^{d}}$ on $\bbH^{d}$ decomposes into
\begin{equation*}
	\De_{\bbH^{d}} = \frac{1}{\sinh^{d-1} r} \rd_{r} (\sinh^{d-1} r \rd_{r}) + \frac{1}{\sinh^{2} r} \slap.
\end{equation*}

Multiplying the equation~\eqref{eq:mwtz:u-eq} by $\bfa_{r} \rd_{r} u$, we compute
\begin{align*}
0 
=& 	\bb( \rd_{t}^{2} u - \frac{1}{\sinh^{d-1} r} \rd_{r} (\sinh^{d-1} r \rd_{r} u) - \frac{1}{\sinh^{2} r} \slap u + V u + F(u)  \bb) \bfa_{r} \rd_{r} u   \\
=&	\rd_{t} \bb( \bfa_{r} \rd_{t} u \rd_{r} u \bb) 
	- \frac{1}{\sinh^{d-1} r} \rd_{r} \bb( \sinh^{d-1} r \bfa_{r} (\rd_{r} u)^{2} \bb) 
	- \sdiv \bb(\frac{\bfa_{r}}{\sinh^{2} r} \snb u \rd_{r} u \bb)  \\  
&	- \frac{1}{2} \bfa_{r} \rd_{r} (\rd_{t} u)^{2}
	+ \frac{1}{2} \bfa_{r} \rd_{r} (\rd_{r} u)^{2}
	+ \frac{1}{2} \frac{\bfa_{r}}{\sinh^{2} r} \rd_{r} \abs{\snb u}^{2} 
	+ \frac{1}{2} \bfa_{r} V \rd_{r} \abs{u}^{2} \\
&	+ \bfa_{r} \rd_{r} (G(u))
	+ \bfa_{rr} (\rd_{r} u)^{2}  \\
=&	\rd_{t} \bb( \bfa_{r} \rd_{t} u \rd_{r} u \bb) 
	- \frac{1}{\sinh^{d-1} r} \rd_{r} (\sinh^{d-1} r J_{r}) 
	- \sdiv \bb(\frac{\bfa_{r}}{\sinh^{2} r} \snb u \rd_{r} u \bb)  \\  
&	+ \frac{ \rd_{r} (\sinh^{d-1} r \, \bfa_{r}) }{\sinh^{d-1} r} 
	\bb( \frac{1}{2} (\rd_{t} u)^{2} - \frac{1}{2} (\rd_{r} u)^{2} - \frac{1}{2 \sinh^{2} r} \abs{\snb u}^{2} - \frac{1}{2} V \abs{u}^{2} - G(u) \bb) \\
&	+ \bfa_{rr} (\rd_{r} u)^{2} + \frac{\coth r \, \bfa_{r}}{\sinh^{2} r} \abs{\snb u}^{2} - \frac{1}{2} \bfa_{r} \rd_{r} V \abs{u}^{2},
\end{align*}
where
\begin{equation*}
	J_{r} 
	= \frac{1}{2} \bfa_{r} (\rd_{t} u)^{2} + \frac{1}{2} \bfa_{r} (\rd_{r} u)^{2} - \frac{1}{2} \frac{\bfa_{r}}{\sinh^{2} r} \abs{\snb u}^{2} - \frac{1}{2} \bfa_{r} V\abs{u}^{2} - \bfa_{r} G(u).
\end{equation*}

As $\bfa$ is radial, note that $\frac{1}{\sinh^{d-1} r} \rd_{r} (\sinh^{d-1} r \bfa_{r}) = \De_{\bbH^{d}} \bfa = 1$. Moreover, by Green's identity and~\eqref{u eq}, we have
\begin{align*}
& \hskip-2em
	\frac{1}{2} (\rd_{t} u)^{2} - \frac{1}{2} (\rd_{r} u)^{2} - \frac{1}{2 \sinh^{2} r} \abs{\snb u}^{2} - \frac{1}{2} V \abs{u}^{2} \\
	=&  \frac{1}{4} (\rd_{t}^{2} - \De_{\bbH^{d}}) u^{2} - \frac{1}{2} (\rd_{t}^{2} u - \De_{\bbH^{d}}  u + V u) u \\
	=&  \frac{1}{2} \rd_{t} (u \rd_{t} u ) - \frac{1}{4} \De_{\bbH^{d}} u^{2} + \frac{1}{2} u F(u)
\end{align*} 

Putting together the preceding computations, we arrive at the identity
\begin{align*}
	0 
= &	\rd_{t} \bb( \bfa_{r} \rd_{r} u \rd_{t} u + \frac{1}{2} u \rd_{t} u \bb) \\
&	- \frac{1}{4} \De_{\bbH^{d}} u^{2} 
	- \frac{1}{\sinh^{d-1} r} \rd_{r} (\sinh^{d-1} r J_{r}) 
	- \sdiv \bb(\frac{\bfa_{r}}{\sinh^{2} r} \snb u \rd_{r} u \bb)  \\  
& 	+ \bfa_{rr} (\rd_{r} u)^{2} + \frac{\coth r \, \bfa_{r}}{\sinh^{2} r} \abs{\snb u}^{2} 
	- \frac{1}{2} \bfa_{r} \rd_{r} V \abs{u}^{2}		+ \frac{1}{2} u F(u) - G(u)
\end{align*}

Let $(t_{1}, t_{2}) \subseteq J$ be any finite interval. We now integrate the preceding identity over $(t_{1}, t_{2}) \times \bbH^{d}$. Using the hypothesis $\vec{u} \in \calH$ and the fact that $\abs{\bfa_{r}} \leq C$, we see that all $r$- and $\omg$-boundary terms vanish. Therefore, we have
\begin{align}
	& \hskip-2em
	- \int_{\set{t} \times (0, \infty) \times \bbS^{d-1}} \bb( \bfa_{r} \rd_{t} u \rd_{r} u + \frac{1}{2} u \rd_{t} u \bb) \, \sinh^{d-1} r \, \ud r \ud \omg \, \bb\vert_{t=t_{1}}^{t_{2}} \label{eq:mwtz-general:key:LHS} \\
	& = \iint_{(t_{1}, t_{2}) \times (0, \infty) \times \bbS^{d-1} } \bb( \bfa_{rr} (\rd_{r} u)^{2} + \frac{\coth r \, \bfa_{r}}{\sinh^{2} r} \abs{\snb u}^{2} \bb)
			\, \sinh^{d-1} r \, \ud t \ud r \ud \omg  \label{eq:mwtz-general:key:RHS:1} \\
	& \phantom{= }		
			+ \iint_{(t_{1}, t_{2}) \times (0, \infty) \times \bbS^{d-1} }
				\frac{1}{2} \bfa_{r} (- \rd_{r} V) \abs{u}^{2} 
			\, \sinh^{d-1} r \, \ud t \ud r \ud \omg 	\label{eq:mwtz-general:key:RHS:2} \\
	& \phantom{= }		
			+ \iint_{(t_{1}, t_{2}) \times (0, \infty) \times \bbS^{d-1} }
				\bb( \frac{1}{2} u F(u) - G(u) \bb)
			\, \sinh^{d-1} r \, \ud t \ud r \ud \omg .	\label{eq:mwtz-general:key:RHS:3}
\end{align}
We claim that
\begin{equation} \label{eq:mwtz-general:cothr-a-r}
	\frac{1}{d} \tanh r \leq \bfa_{r} \leq \tanh r
\end{equation}
Indeed, for the first inequality, we write
\begin{equation*}
	\coth r \bfa_{r} \geq \frac{1}{\sinh^{d} r} \int_{0}^{r} \sinh^{d-1} r' \cosh r' \, \ud r' = \frac{1}{d},
\end{equation*}
whereas for the second inequality, we proceed as follows:
\begin{equation*}
	\coth r \bfa_{r} 
	\leq \frac{\cosh r}{\sinh^{2} r} \int_{0}^{r} \sinh r' \, \ud r' = \frac{\cosh r (\cosh r - 1) }{\cosh^{2} r - 1} \leq 1.
\end{equation*}
By the upper bound in \eqref{eq:mwtz-general:cothr-a-r}, the $t$-boundary terms in \eqref{eq:mwtz-general:key:LHS} can be estimated (in absolute value) by $C \nrm{\vec{u}}_{L^{\infty}_{t}(J; \HH)}$. Moreover, from the lower bound in \eqref{eq:mwtz-general:cothr-a-r}, we see that
\begin{equation*}
	\eqref{eq:mwtz-general:key:RHS:2} + \eqref{eq:mwtz-general:key:RHS:3} 
	\geq 	\iint_{(t_{1}, t_{2}) \times \bbH^{d} }
				\frac{1}{2d} \tanh r (-  \rd_{r} V) \abs{u}^{2} +  \bb( \frac{1}{2} u F(u) - G(u) \bb)
			\, \ud t \, \mu (\ud x) .
\end{equation*}
Hence, to complete the proof of \eqref{eq:mwtz-general}, it only remains to show that we have
\begin{equation} \label{eq:mwtz-general:du}
	\eqref{eq:mwtz-general:key:RHS:1} \geq c 	
					\iint_{J \times \bbH^{d}}  \bfc^{2} \abs{\nb u}^{2}_{\bfg} \, \ud t \, \mu(\ud x)
\end{equation}
for some absolute constant $c > 0$. 

An elementary computation shows that $\bfc^{2}(r)$ as defined above coincides with $\bfa_{rr}(r)$, i.e.,
\begin{align*}
	\bfa_{rr} (r)
	=& \frac{\cosh r}{\sinh^{d} r} \int_{0}^{r} \frac{\sinh^{d-1} r'}{\cosh^{2} r'} \, \ud r' = \bfc^{2}(r).
\end{align*}
Observe, in particular that 
\begin{equation*}
	\bfc^{2}(r) \leq \frac{\cosh r}{\sinh r} \int_{0}^{r} \frac{1}{\cosh^{2} r'} \, \ud r' = 1 \leq d \coth r \bfa_{r}
\end{equation*}
where we used \eqref{eq:mwtz-general:cothr-a-r} in the last inequality. Combined with the identity $\abs{\nb u}^{2}_{\bfg} = \abs{\rd_{r} u}^{2} + \frac{1}{\sinh^{2} r} \abs{\snb u}^{2}$, the desired inequality \eqref{eq:mwtz-general:du} now follows. \qedhere
\end{proof}

Now the Morawetz estimates claimed in Lemmas~\ref{lem:mwtz} and~\ref{lem:mwtz:V} follow as an immediate consequence of Lemma~\ref{lem:mwtz-general}.
\begin{proof} [Proofs of Lemmas~\ref{lem:mwtz} and~\ref{lem:mwtz:V}]
In both cases, $d = 3$, $F(u) = u^{5}$ and $G(u) = \frac{1}{6} u^{6}$; hence we have
\begin{equation*}
	\frac{1}{2} u F(u) - G(u) = \bb( \frac{1}{2} - \frac{1}{6} \bb) u^{6} = \frac{1}{3} u^{6}.
\end{equation*}
Furthermore, $V = 0$ for Lemma~\ref{lem:mwtz} and the hypothesis on $V$ ensures that Lemma~\ref{lem:mwtz-general} can be applied. Then the desired Morawetz estimates immediately follow from \eqref{eq:mwtz-general}. \qedhere
\end{proof}

Next, consider the linear wave equation (i.e., $F(u) = G(u) = 0$) with a potential $V$, which is assumed to be $C^{1}$, non-negative and repulsive, i.e., $- \rd_{r} V \geq 0$. Lemma~\ref{lem:mwtz-general} can be used to establish integrated local energy decay estimate for this equation.
\begin{lem} \label{lem:ILED}
Let $d \geq 3$ and $V$ be a $C^{1}$ potential that is non-negative (i.e., $V \geq 0$) and \emph{repulsive} in the sense that $-\rd_{r} V \geq 0$, where $\rd_{r}$ is the radial direction derivative in the polar coordinates on $\bbH^{d}$. 
Then for any $f \in L^{2}(\bbH^{d})$, we have
\begin{equation} \label{eq:ILED}
	\nrm{\frac{1}{\cosh r} e^{\pm i t D_{V}} f}_{L^{2}(\bbR; L^{2}(\bbH^{d}))} \aleq \nrm{f}_{L^{2}(\bbH^{d})},
\end{equation}
where $D_{V} = \sqrt{-\lap_{\bbH^{d}} + V}$.
\end{lem}

\begin{proof}
The proof consists of two steps. In Step 1, we combine Lemma~\ref{lem:mwtz-general} with another multiplier argument to prove an integrated local energy decay estimate for solutions to $\Box_{\bbH^{d}} u + V u = 0$; see \eqref{eq:ILED:step1}. The key point is to control the time derivative $\rd_{t} u$. Then in Step 2, we show how \eqref{eq:ILED:step1} implies the desired estimate \eqref{eq:ILED}.

\vskip.5em
\noindent {\it Step 1.}
Let $u$ be a solution to the linear equation $\Box_{\bbH^{d}} u + V u = 0$ in the class $C_{t}(\bbR; \HH)$. 
The goal of this step is to prove the following inequality for $u$:
\begin{equation} \label{eq:ILED:step1}
	\iint_{\bbR \times \bbH^{d}} \frac{1}{\cosh^{2} r} \bb( \abs{\rd_{t} u}^{2} + \abs{\nb u}^{2}_{\bfg} \bb) \,  \mu(\ud x) \ud t \leq C \nrm{\vec{u}(0)}_{\HH_{V}}^{2}.
\end{equation}

We first claim that the function $\bfc(r)$ introduced in Lemma~\ref{lem:mwtz-general} obeys 
\begin{equation} \label{eq:ILED:lb4c}
\bfc^{2}(r) \geq \frac{1}{d\cosh^{2} r}
\end{equation}
Indeed, we compute
\begin{align*}
\bfc^{2}(r) 
\geq & \frac{1}{\sinh^{d} r} \int_{0}^{r} \frac{\sinh^{d-1} r' \cosh r' }{\cosh^{2} r'} \, \ud r' \\
= & \frac{1}{d\cosh^{2} r} - \frac{1}{d\sinh^{d} r} \int_{0}^{r} \sinh^{d} r'  \bb(- 2 \frac{\sinh r'}{\cosh^{3} r'} \bb) \, \ud r'.
\end{align*}
As the second term on the last line is non-negative, we see that $\bfc^{2}(r) \geq \frac{1}{d \cosh^{2} r}$.

By \eqref{eq:ILED:lb4c}, the conservation of energy, and~\eqref{eq:mwtz-general},  the desired inequality \eqref{eq:ILED:step1} with only $\iint \frac{1}{\cosh^{2} r} \abs{\nb u}^{2}_{\bfg}$ on the left-hand side follows. It remains to show that we can also control the weighted space-time integral of $\abs{\rd_{t} u}^{2}$.
This task will be accomplished by using another multiplier argument, combined with the bound already established in Lemma~\ref{lem:mwtz-general}. 

Consider a radial function $\bfb$ on $\bbH^{d}$ characterized by
\begin{equation*}
	\lap_{\bbH^{d}} \bfb = \frac{1}{\cosh^{2} r}, \quad \bfb_{r} = \frac{1}{d} r + o_{r}(1) \mas r \to 0.
\end{equation*}
The radial derivative $\bfb_{r}$ is then explicitly given by the formula
\begin{equation*}
	\bfb_{r} = \frac{1}{\sinh^{d-1} r} \int_{0}^{r} \frac{\sinh^{d-1} r'}{\cosh^{2} r'} \, \ud r'.
\end{equation*}
Proceeding as in the first part of the proof of Lemma~\ref{lem:mwtz-general} using $\bfb_{r} \rd_{r} u$ as the multiplier, we arrive at the identity
\begin{align*}
& \hskip-2em
\rd_{t} \bb( - \bfb_{r} \rd_{t} u \rd_{r} u \bb) 
	+ \frac{1}{\sinh^{d-1} r} \rd_{r} (\sinh^{d-1} r J_{r}) 
	+ \sdiv \bb(\frac{\bfb_{r}}{\sinh^{2} r} \snb u \rd_{r} u \bb)  \\
= &	\frac{1}{\cosh^{2} r}
	\bb( \frac{1}{2} (\rd_{t} u)^{2} - \frac{1}{2} (\rd_{r} u)^{2} - \frac{1}{2 \sinh^{2} r} \abs{\snb u}^{2} - \frac{1}{2} V \abs{u}^{2} \bb) \\
&	+ \bfb_{rr} (\rd_{r} u)^{2} + \frac{\coth r \, \bfb_{r}}{\sinh^{2} r} \abs{\snb u}^{2} - \frac{1}{2} \bfb_{r} \rd_{r} V \abs{u}^{2},
\end{align*}
where
\begin{equation*}
	J_{r} 
	= \frac{1}{2} \bfb_{r} (\rd_{t} u)^{2} + \frac{1}{2} \bfb_{r} (\rd_{r} u)^{2} - \frac{1}{2} \frac{\bfb_{r}}{\sinh^{2} r} \abs{\snb u}^{2} - \frac{1}{2} \bfb_{r} V\abs{u}^{2}.
\end{equation*}
We now integrate the preceding identity over space-time regions of the form $(t_{1}, t_{2}) \times \bbH^{d}$. As it is evident that $\bfb_{r}$ is bounded, contribution of the left-hand side can be controlled by the conserved energy, and hence by $\nrm{\vec{u}(0)}_{\HH_{V}}^{2}$. Therefore, rearranging terms and recalling that $\abs{\nb u}_{\bfg} = (\rd_{r} u)^{2} + \frac{1}{\sinh^{2} r} \abs{\snb u}^{2}$, we arrive at the inequality
\begin{align}
& \hskip-2em
	\iint_{(t_{1}, t_{2}) \times \bbH^{d}} \frac{1}{2} \frac{1}{\cosh^{2} r} \abs{\rd_{t} u}^{2} \, \ud t \, \mu(\ud x) \notag \\
	\leq & C \nrm{\vec{u}(0)}_{\HH_{V}}^{2} 
	- \iint_{(t_{1}, t_{2}) \times \bbH^{d}} 
		\bb( \frac{\coth r \, \bfb_{r}}{\sinh^{2} r} \abs{\snb u}^{2} - \frac{1}{2} \bfb_{r} \rd_{r} V \abs{u}^{2} \bb)
		\, \ud t \, \mu(\ud x) \label{eq:ILED:step1:easy} \\
	&	+ \iint_{(t_{1}, t_{2}) \times \bbH^{d}} 		
		\bb( \frac{1}{2} \frac{1}{\cosh^{2} r} \abs{\nb u}_{\bfg}^{2} - \bfb_{rr} (\rd_{r} u)^{2}  \bb)
		\, \ud t \, \mu(\ud x) \label{eq:ILED:b-rr} \\
	&	+ \iint_{(t_{1}, t_{2}) \times \bbH^{d}} 				
		\frac{1}{2} \frac{1}{\cosh^{2} r} V \abs{u}^{2} 
		\, \ud t \, \mu(\ud x).  \label{eq:ILED:Vu}
\end{align}
We now treat \eqref{eq:ILED:step1:easy}, \eqref{eq:ILED:b-rr} and \eqref{eq:ILED:Vu} in order. First, the last term in \eqref{eq:ILED:step1:easy} can be easily dropped, since it is evidently non-positive. To prove that $\eqref{eq:ILED:b-rr} \leq C \nrm{\vec{u}}_{\HH_{V}}^{2}$, it suffices, by Lemma~\ref{lem:mwtz-general} and \eqref{eq:ILED:lb4c}, to show
\begin{equation*}
	- \bfb_{rr} \leq (d-1) \bfc^{2}(r).
\end{equation*}
Indeed, we have
\begin{align*}
	-\bfb_{rr}=(d-1)\coth r \bfb_r-\frac{1}{\cosh^2r}=(d-1)\bfc^2(r)-\frac{1}{\cosh^2r}\leq(d-1)\bfc^2(r).
\end{align*}
On the last line, observe that the first term equals $2 \bfc^{2}(r)$, whereas the second term is non-positive. 

Finally, to show that $\eqref{eq:ILED:Vu} \leq C \nrm{\vec{u}(0)}_{\HH_{V}}^{2}$, it suffices by Lemma~\ref{lem:mwtz-general} and \eqref{eq:ILED:lb4c} to establish
\begin{equation} \label{eq:ILED:Vu:key}
\begin{aligned}
& \hskip-2em
	\iint_{\bbR \times \bbH^{d}} \frac{1}{\cosh^{2} r} V \abs{u}^{2} \, \ud t \, \mu(\ud x) \\
	\aleq & \iint_{\bbR \times \bbH^{d}} \tanh r (- \rd_{r} V) \abs{u}^{2} \, \ud t \, \mu(\ud x) 
		+ \iint_{\bbR \times \bbH^{d}} \frac{1}{\cosh^{2} r} \abs{\rd_{r} u}^{2} \, \ud t \, \mu(\ud x)  
\end{aligned}
\end{equation}
To prove \eqref{eq:ILED:Vu:key}, we begin by introducing an auxiliary function $M(r)$ defined by
\begin{equation*}
	M(r) = \int_{0}^{r} \frac{\sinh^{d-1} r'}{\cosh^{2} r'} \, \ud r'.
\end{equation*}
Performing an integration by parts, we see that
\begin{align*}
\int_{0}^{\infty} \frac{V}{\cosh^{2} r} \abs{u}^{2} \, \sinh^{d-1} r \, \ud r 
= & \int_{0}^{\infty} V \abs{u}^{2} \, M'(r) \, \ud r \\
= & - \int_{0}^{\infty} \rd_{r} V \abs{u}^{2} \, M(r) \, \ud r 
	- 2 \int_{0}^{\infty} V u \rd_{r} u \, M(r) \, \ud r 
\end{align*}
Applying Cauchy to the last term and absorbing $\int \frac{V}{\cosh^2r} \abs{u}^{2} \sinh^{d-1} r \, \ud r$ into the left-hand side, we obtain
\begin{align*}
\int_{0}^{\infty} \frac{V}{\cosh^{2} r} \abs{u}^{2} \, \sinh^{d-1} r \, \ud r 
\aleq & \int_{0}^{\infty} \frac{M(r)}{\sinh^{d-1} r} (- \rd_{r} V) \abs{u}^{2} \, \sinh^{d-1} r \, \ud r \\
	& + \int_{0}^{\infty} \frac{M^{2}(r) \cosh^{2} r}{\sinh^{2(d-1)} r} V \abs{\rd_{r} u}^{2} \, \sinh^{d-1} r \, \ud r . 
\end{align*}
We easily see that
\begin{equation*}
\frac{M(r)}{\sinh^{d-1} r} \leq \int_{0}^{r} \frac{\ud r'}{\cosh^{2} r'} = \tanh r.
\end{equation*}
On the other hand, recalling the definitions of $\bfc(r)$ and $M(r)$, we have
\begin{align*}
\frac{M^{2}(r) \cosh^{2} r}{\sinh^{2(d-1)} r} V 
= \frac{M(r) \cosh r}{\sinh^{d - 2} r} V \bfc^{2}(r)
\end{align*}
where
\begin{align*}
\frac{M(r) \cosh r}{\sinh^{d - 2} r} V
\leq & \frac{\cosh r}{\sinh r} V \int_{0}^{r} \tanh^{2} r' \, \ud r'  = (r \coth r - 1) V. 
\end{align*}
As $(r \coth r - 1) \aleq r$ and $V$ decays exponentially by the assumption \eqref{eq:Vdecay}, we see that the last line is $\aleq 1$. Therefore, we conclude that
\begin{equation*}
\int_{0}^{\infty} \frac{V}{\cosh^{2} r} \abs{u}^{2} \, \sinh^{d-1} r \, \ud r 
\leq 	\int_{0}^{\infty} \bb( \tanh r (- \rd_{r} V) \abs{u}^{2} 
					+ \bfc^{2}(r) \abs{\rd_{r} u}^{2} \bb)
					\, \sinh^{d-1} r \, \ud r \\
\end{equation*}
Integrating the preceding inequality in $t$ and the angular variables, we arrive at \eqref{eq:ILED:Vu:key}. 

Observing that the estimates so far are independent of the interval $(t_{1}, t_{2})$, the proof of \eqref{eq:ILED:step1} is now complete.

\vskip.5em
\noindent {\it Step 2.}
In this step, we show that \eqref{eq:ILED} follows from \eqref{eq:ILED:step1}.

As $V$ is non-negative, given any $f \in L^{2}(\bbH^{d})$, we can find $g \in H^{2}(\bbH^{d})$ such that 
\begin{equation*}
	f = (-\lap_{\bbH^{d}} +V) g, \quad \nrm{g}_{H^{2}(\bbH^{d})} \aleq \nrm{f}_{L^{2}(\bbH^{d})}.
\end{equation*}
The existence of such a $g$ can be proved, for instance, by minimizing the functional $\calL[g] = \abs{\nb g}_{\bfg}^{2} + V \abs{g}^{2} - f g$. That $\calL[g]$ is bounded from below is a consequence of Poincar\'e's inequality and $V \geq 0$. The $H^{2}(\bbH^{d})$ bound on $g$ then follows from elliptic regularity. We omit the standard details.

By Euler's formula, we can write
\begin{align*}
	e^{i t D_{V}} f
	&=-i\partial_t (e^{itD_V}D_Vg)=-i\partial_t S_V(t)(D_V g, 0)+\partial_t S_V(t)(0,(-\lap_{\bbH^{d}} + V) g)\\
	&=-i\partial_t u+\partial_t v
\end{align*}
where $u=S_V(t)(D_V g, 0)$ and $v=S_V(t)(0,(-\lap_{\bbH^{d}} + V) g)$. Applying \eqref{eq:ILED:step1} to $u$ and $v$, \eqref{eq:ILED} follows for $e^{itD_V}f.$ The case of $e^{-itD_V}f$ is similar. \qedhere
\end{proof}

Finally the proof of Lemma~\ref{lem:defocusingV}  will be complete once we prove the following lemma.

\begin{lem}\label{lem: Strichartz proof}
Suppose $d \geq 3$ and let $V$ be a non-negative smooth compactly supported potential, which is repulsive as in Lemma~\ref{lem:ILED}. Let $f$ be a smooth, compactly supported complex-valued function on $\bbH^{d}$. If $(p, q, \gmm)$ is a hyperbolic admissible triple, then
\begin{equation*}
	 \nrm{e^{\pm i t D_{V}} f}_{L^{q}_{t} W^{-\gmm, r}_{x}(\bbR \times \bbH^{d})} \leq C \nrm{f}_{L^{2}}
\end{equation*}
and
\begin{equation*} 
	\nrm{\sqrt{-\lap_{\bbH^{d}}} \, e^{\pm i t D_{V}} f}_{L^{q}_{t} W^{-\gmm, r}_{x}(\bbR \times \bbH^{d})}
	\leq C \nrm{D_{V} f}_{L^{2}}.
\end{equation*}

\end{lem}
\begin{proof}
The proof follows the outline of the proof of Proposition 4.2 in \cite{LOS1}. To simplify the notation, we let $X:=L^{q}_{t} W^{-\gmm, r}_{x}(\bbR \times \bbH^{d})$ and omit writing $\bbH^{d}$ in $\lap_{\bbH^{d}}$ and $\Box_{\bbH^{d}}$. By time reversal symmetry, it suffices to consider one type of half-waves, say, $e^{i t D_{V}} f$.

Given $w = e^{i t D_{V}} f$, note that it satisfies the equation $\Box w = - V w$. Applying Duhamel's formula to the real and imaginary parts of $w$, we obtain
\begin{equation*}
	w(t) = S_{\hyp, 0}(t) (f, i D_{V} f) + \int_{0}^{t} S_{\hyp, 0}(t - s)(0, - V w(s)) \, \ud s.
\end{equation*}
We need to show that $\nrm{w}_{X} \aleq \nrm{f}_{L^{2}}$ and $\nrm{\sqrt{-\lap} w}_{X} \aleq \nrm{D_{V} f}_{L^{2}}$. Since the two estimates are similar, we concentrate on the latter. By the Strichartz estimate for $S_{\hyp}$ (Proposition~\ref{strich}) and Lemma~\ref{lem:basicVWave}, the term $S_{\hyp,0}(t)(f, i D_{V} f)$ can be easily treated. Hence it remains to prove
\begin{equation}  \label{eq:strich-before-CK}
	\left\Vert \sqrt{-\lap} \int_{0}^{t} S_{\hyp, 0}(t-s) (0, -V w(s)) \, \ud s \right\Vert_{X} \aleq \nrm{D_{V} f}_{L^{2}}.
\end{equation}
By Euler's formula applied to $S_{\hyp, 0}(t-s)(0, - Vw(s))$ and the Christ-Kiselev lemma \cite{CK01,Sogge}, it suffices to prove that
\begin{equation} \label{Strich final}
	\left\Vert \int_{-\infty}^{\infty} e^{\pm i (t-s) \sqrt{-\lap}} V w(s) \, \ud s \right\Vert_{X} \aleq \nrm{D_{V} f}_{L^{2}}
\end{equation}
Henceforth, we only consider the case $e^{-i(t-s) \sqrt{-\lap}}$, the other being similar. We write $V=V_1V_2$ where $V_1$ and $V_2$ are two compactly supported non-negative functions. We then have
\ali{\label{reduced Strich final 1}
\left\|\int_{-\infty} ^\infty e^{-i(t-s)\sqrt{-\lap}}Vw(s)ds\right\|_X\leq \|K\|_{L_{t,x}^2\rightarrow X}\|V_2 w\|_{L_{t,x}^2},
}
where
\ant{
(Kg)(t):=\int_{-\infty}^\infty e^{-i(t-s)\sqrt{-\lap}}V_1g(s)ds.
}
The factor $\nrm{V_{2} w}_{L^{2}_{t,x}}$ can be bounded using the local energy decay estimate proved in Lemma~\ref{lem:ILED} as
\ali{\label{lede app}
\|V_2 w\|_{L_{t,x}^2}=\|V_2 e^{i t D_{V}} f \|_{L_{t,x}^2}\lesssim \|f\|_{L^2} \lesssim \| D_{V} f \|_{L^{2}},
}
where we used \eqref{eq:Vpos} for the last inequality. For $\nrm{K}_{L^{2}_{t,x} \to X}$, note that
\ali{\label{reduced Strich final 2}
\|Kg\|_{X}\leq \|e^{-it\sqrt{-\lap}}\|_{L_{x}^2\rightarrow X}\left\|\int_{-\infty}^\infty e^{is\sqrt{- \lap}}V_1g(s)ds\right\|_{L_x^2}.
}
Now $\|e^{-it\sqrt{-\lap}}\|_{L_{x}^2\rightarrow X}$ is bounded by a constant in view of the Strichartz estimates for $S_\hyp$. We also claim that
\ali{\label{reduced Strich final 3}
\left\|\int_{-\infty}^\infty e^{is\sqrt{-\lap}}V_1g(s)ds\right\|_{L_x^2}\leq C\|g\|_{L_{t,x}^2}.
}
Indeed, by duality this is equivalent to
\ant{
\|V_1e^{-it\sqrt{-\lap}}\phi\|_{L_{t,x}^2}\leq C\|\phi\|_{L^2_x},\qquad\forall\phi\in L^2_x,
}
which is a consequence of the local energy decay estimate proved in Lemma~\ref{lem:ILED} in the special case $V\equiv0.$ The estimate \eqref{Strich final} now follows from \eqref{reduced Strich final 1}--\eqref{reduced Strich final 3}. \qedhere
\end{proof}

\bibliographystyle{plain}
\bibliography{researchbib}

 \bigskip

\centerline{\scshape Andrew Lawrie, Sung-Jin Oh}
\smallskip
{\footnotesize
 \centerline{Department of Mathematics, The University of California, Berkeley}
\centerline{970 Evans Hall \#3840, Berkeley, CA 94720, U.S.A.}
\centerline{\email{ alawrie@math.berkeley.edu, sjoh@math.berkeley.edu}}
} 

 \medskip

\centerline{\scshape Sohrab Shahshahani}
\medskip
{\footnotesize
 \centerline{Department of Mathematics, The University of Michigan}
\centerline{2074 East Hall, 530 Church Street
Ann Arbor, MI  48109-1043, U.S.A.}
\centerline{\email{shahshah@umich.edu}}
} 

\end{document}